\documentclass[final,1p,times]{elsarticle}%
\usepackage{amsmath}
\usepackage{amssymb}
\usepackage{amscd}
\usepackage{amsfonts}
\usepackage{enumerate}
\usepackage{mathrsfs}
\usepackage{xy}
\usepackage{stmaryrd}
\usepackage{graphicx}
\usepackage{fancybox}
\usepackage{color}
\usepackage{multirow}
\usepackage{exscale}
\usepackage{enumitem,array}%
\usepackage{subfigure}
\usepackage{extarrows}
\usepackage{amsfonts}

\newtheorem{mytheo}{Theorem}[section]
\newtheorem{myexp}{Experiment}[section]

\newtheorem{lem}[mytheo]{Lemma}

\newcounter{remark}

\newcounter{problem}

\makeatletter
\def\@upcite#1#2{\textsuperscript{[{#1\if@tempswa , #2\fi}]}}
\makeatother
\newenvironment{proof}{\vspace{1ex}
{\it Proof. }\hspace{0.3em}}{\vspace{1ex}} \journal{Journal of
Computational Physics}
\begin{document}
\begin{frontmatter}
\title{ General local energy-preserving integrators for solving\\ multi-symplectic Hamiltonian PDEs }
\author[Li]{Yu--Wen~Li}
\author[Li]{Xinyuan Wu\corref{cor1}}
%\author[Wu]{}

\tnotetext[]{The research is supported in part by the Natural
Science Foundation of China under Grant 11271186, by NSFC and RS
International Exchanges Project under Grant 113111162,  by the
Specialized Research Foundation for the Doctoral Program of Higher
Education under Grant 20130091110041, and by the 985 Project at
Nanjing University under Grant 9112020301.}

\address[Li]{Department of Mathematics, Nanjing University; State Key Laboratory
for Novel Software Technology at Nanjing University, Nanjing 210093,
P.R.China}

 \ead{farseer1118@sina.cn, xywu@nju.edu.cn}
\cortext[cor1]{Corresponding author.}

\begin{abstract}
In this paper we propose and investigate a general approach to constructing local
energy-preserving algorithms which can be of arbitrarily high order in time for solving
Hamiltonian PDEs. This approach is based on the temporal discretization using continuous
Runge-Kutta-type methods, and the spatial discretization using pseudospectral methods
or Gauss–Legendre collocation methods. The local energy conservation law of our new
schemes is analyzed in detail. The effectiveness of the novel local energy-preserving
integrators is demonstrated by coupled nonlinear Schr\"odinger equations and 2D nonlinear
Schr\"odinger equations with external fields. Our new schemes are compared with some
classical multi-symplectic and symplectic schemes in numerical experiments. The numerical
results show the remarkable \emph{long-term} behaviour of our new schemes.
\end{abstract}
%\allowdisplaybreaks[4]
%\renewcommand{\baselinestretch}{1.1}

\end{frontmatter}
 \vskip0.5cm \pagestyle{myheadings}
\thispagestyle{plain} \markboth{}{\centerline{\small Y. W. Li and X.
Wu}}

%=============================================================================================
\section{Introduction}
Since the multi-symplectic structure was developed by Bridges and
Marsden {\em et al.} {\cite{Bridges1997,Marsden1999}} for a class of
PDEs, the construction and analysis of multi-symplectic numerical
integrators which conserve the discrete multi-symplectic structure
have become one of the central topics in PDE algorithms. Many
multi-symplectic schemes have been proposed such as multi-symplectic
RK/PRK/RKN methods, finite volume methods, spectral/ pseudospectral
methods, splitting methods and wavelet collocation methods (see,
e.g.
{\cite{Bridges2001,Reich2001,Chen2001,Hong2005,Hong2007,Maclachlan2014,Reich2000,Ryland2007,Zhu2011}}).
All of these methods focus on the preservation of some {kinds of}
discrete multi-symplecticity. However, {multi-symplectic} PDEs have
many other important properties such as the local energy
conservation law (ECL) and the local momentum conservation law
(MCL). In general, multi-symplectic integrators can only preserve
exactly quadratic conservation laws and invariants. In the paper
{\cite{Reich2000}}, Reich firstly proposed two methods that preserve
the discrete ECL and MCL respectively. In {\cite{Wang2008}}, Wang
{\em et al.} generalized Reich's work. In \cite{Cai1,Cai2,Chen2011},
Chen{ \em et al}., and  Cai {\em et al}. constructed some local
structure-preserving schemes for special multi-symplectic PDEs.  In
\cite{Gong2014}, Gong {\em et al}. developed a general approach to
constructing local structure-preserving algorithms. Local
energy-preserving algorithms preserve the discrete global energy
under suitable boundary conditions. Thus in the case of
multi-symplectic PDEs, they cover the traditional global
energy-preserving algorithms (see, e.g.
{\cite{Celledoni2010,Fei1991,Guo1983,Kara2013}}). However, most of
the local and global energy-preserving methods are based on the
discrete gradient for the temporal discretization. Therefore, they
can have only second order accuracy in time. We note that Hairer
\cite{Hairer2010} developed a family of energy-preserving continuous
Runge--Kutta--type methods of arbitrarily high order for Hamiltonian
ODEs. Motivated by Hairer's work, in this paper, we consider general
local energy-preserving methods for multi-symplectic Hamiltonian
PDEs, and we are hopeful of obtaining new high-order schemes which
exactly preserve the ECL.

Besides, most of the existing local energy-preserving
algorithms are based on the spatial discretization using the
implicit midpoint rule. Although the authors in
\cite{Celledoni2010,Gong2014} mentioned a class of global
energy-preserving schemes based on the (pseudo) spectral
discretization for the spatial derivative, it seems that there is
little work investigating the local energy-preserving property of
these schemes in the literature. In this paper,  we
investigate the preservation of the discrete ECL for our
new schemes which are based on the pseudospectral spatial discretization.
Meanwhile, we also design a class of local energy-preserving
schemes based on the general Gauss-Legendre collocation spatial discretization.

The paper is organized as follows. In Section \ref{MSECRK}, we
briefly introduce multi-symplectic PDEs and energy-preserving
continuous Runge--Kutta methods. In Section \ref{CLA}, we present a
general approach to constructing local energy-preserving schemes.
This approach is illustrated by coupled nonlinear Schr\"{o}dinger
equations and 2D nonlinear Schr\"{o}dinger equations in Section
\ref{LECNS} and \ref{LE2DS}, respectively. We compare
our new schemes with classical multi-symplectic and symplectic
schemes in Section \ref{Numer_Experi} and \ref{Numer_Experi2}.
The last section is concerned with the conclusion.

\section{Multi-symplectic PDEs and energy-preserving continuous Runge--Kutta
methods}\label{MSECRK} A multi-symplectic PDE with one temporal
variable and two spatial variables can be written in the form:
\begin{equation}\label{MSTSV}Mz_{t}+Kz_{x}+Lz_{y}=\nabla_{z} S(z),\quad z\in \mathbb{R}^{d},\end{equation}
where $M,\ K,$ and $L$ are skew-symmetric $d$ by $d$
real matrices, $\ S: \mathbb{R}^{d}\rightarrow \mathbb{R}$ is a
smooth scalar-valued function of the state variable
variable $z$ and $\nabla_{z}$ is the gradient operator. Three
differential 2-forms are defined by
\begin{equation*}
\omega=\frac{1}{2}dz\wedge Mdz,\quad\kappa=\frac{1}{2}dz\wedge
Kdz,\quad\tau=\frac{1}{2}dz\wedge Ldz.
\end{equation*}
\eqref{MSTSV} then has the multi-symplectic conservation law (MSCL):
\begin{equation}\label{MSCL}
\partial_{t}\omega+\partial_{x}\kappa+\partial_{y}\tau=0.
\end{equation}
Another important local conservation law is the ECL:
\begin{equation*}
\partial_{t}E+\partial_{x}F+\partial_{y}G=0,
\end{equation*}
where
$$E=S(z)-\frac{1}{2}z^{\intercal}Kz_{x}-\frac{1}{2}z^{\intercal}Lz_{y},\quad F=\frac{1}{2}z^{\intercal}Kz_{t},\quad G=\frac{1}{2}z^{\intercal}Lz_{t}.$$
When $L=0,$ the equation \eqref{MSTSV} reduces to the case of one
spatial dimension:
\begin{equation}\label{CMS}
Mz_{t}+Kz_{x}=\nabla S(z).
\end{equation}
Correspondingly, the ECL  reduces to:
\begin{equation*}
\partial_{t}E+\partial_{x}F=0,
\end{equation*}
where
$$E=S(z)-\frac{1}{2}z^{\intercal}Kz_{x},\quad F=\frac{1}{2}z^{\intercal}Kz_{t}.$$

Note that the energy density $E$ is related to the gradient
of $S$. If one is interested in constructing
schemes which can preserve the discrete ECL, a natural idea is
replacing $\nabla S$ by the discrete gradient (DG) $\bar{\nabla}
S.$ For details of the discrete gradient,
readers are referred to {\cite{Gonzalez1996,Maclachlan1999}}.

A limitation of the DG method is that it can only achieve
second-order accuracy in general. Therefore, classical
local energy-preserving methods {based on the DG} cannot reach
an order higher than $2$ in temporal direction unless
the composition technique is applied, which is not our interest in
this paper.

In contrast to the DG method, Hairer's seminal work overcomes the order barrier. In what follows, we introduce
the approach summarily.

Consider  autonomous ODEs:
\begin{equation}
\left\{
\begin{aligned}
&y^\prime=f(y),\quad y\in \mathbb{R}^{d},\\
&y(t_{0})=y_{0}.\\
\end{aligned}\right.
\end{equation}
Hairer's approach can be regarded as  a continuous Runge--Kutta
method :
\begin{equation}\label{CRK}
\left\{
\begin{aligned}
&y_{\tau}=y_{0}+h\int_{0}^{1}A_{\tau,\sigma}f(y_{\sigma})d\sigma,\\
&y_{1}=y_{0}+h\int_{0}^{1}B_{\sigma}f(y_{\sigma})d\sigma,\\
&B_{\sigma}\equiv1,A_{\tau,\sigma}=\sum_{i=1}^{s}\frac{1}{b_{i}}\int_{0}^{\tau}l_{i}(\alpha)d\alpha l_{i}(\sigma),\\
\end{aligned}\right.
\end{equation}
where $h$ is the stepsize, $\{l_{i}(\tau)\}_{i=1}^{s}$ are Lagrange
interpolating polynomials based on the $s$ distinct points
$c_{1},c_{2},\ldots,c_{s}$, $b_{i}=\int_{0}^{1}l_{i}(\tau)d\tau$ for
$i=1,2,\ldots,s$, and $y_{\tau}$ approximates the value of
$y(t_{0}+\tau h)$ for {$\tau\in[0,1]$.} The continuous RK method can
be expressed in a Buchter tableau as
\[%
\begin{tabular}
[c]{l}%
\\
\\[2mm]%
\begin{tabular}
[c]{c|c}%
$C_{\tau}$ & $A_{\tau,\sigma}$\\\hline &
$\raisebox{-1.3ex}[0pt]{$B_{\tau}$}$\\
\end{tabular}
\end{tabular}%
\]
with
$$\tau=C_{\tau}=\int_{0}^{1}A_{\tau,\sigma}d\sigma.$$ If
$f(y)=J^{-1}\nabla H(y)$, $J$ is a skew-symmetric matrix, then this
method preserves the Hamiltonian: $H(y_{1})=H(y_{0})$. Let $r$ be
the order of the quadrature formula
$(b_{i},c_{i})_{i=1}^{s}$, then the order of this continuous
method is given by:
\begin{equation}\label{Order}
\left\{
\begin{aligned}
&2s,\quad\quad\quad\quad     \text{for $r\geq 2s-1$,}\\
&2r-2s+2,\ \text{ for $r\leq 2s-2$}.\\
\end{aligned}\right.
\end{equation}
Moreover, if the quadrature nodes are symmetric, i.e.
$c_{i}=1-c_{s+1-i}$ {for $i=1,2,\ldots,s$, then the} method
\eqref{CRK} is also symmetric. Clearly, by choosing an s-point
Gauss-Legendre quadrature formula, we get a symmetric continuous RK method of
order $2s$. Besides, although this method is not symplectic, it is
conjugate-symplectic up to at least order $2s+2$. The proof can be
found in {\cite{Hairer2010}}. In view of these prominent properties,
we select \eqref{CRK} as our elementary method for the {time
integration
 of Hamiltonian PDEs}. We denote this method by CRK and call
$(b_{i},c_{i})_{i=1}^{s}$ as the generating quadrature fomula in the
remainder of this paper.

\section{Construction of local energy-preserving algorithms for Hamiltonian
PDEs}\label{CLA}
\subsection{Pseudospectral spatial discretization}
For simplicity, we first consider the following PDE
with one spatial variable:
\begin{equation}\label{ONEDIM}Mz_{t}+Kz_{x}=\nabla_{z} S(z,x).\end{equation}
In the classical multi-symplectic PDE \eqref{CMS}, the Hamiltonian $S$ is
independent of the variable $x$. It should be noted that
\eqref{ONEDIM} does not have the MSCL and the MCL, but the local energy conservation law still
holds:

\begin{equation}\label{1DECL}
\partial_{t}E+\partial_{x}F=0,
\end{equation}
where
$$E=S(z,x)-\frac{1}{2}z^{\intercal}Kz_{x},\quad F=\frac{1}{2}z^{\intercal}Kz_{t}.$$
Thus local energy-preserving methods can be more widely used than
classical multi-symplectic methods. Most of multi-symplectic methods
can be constructed by concatenating two ODE methods {in time and
space, respectively}. The temporal method is always symplectic,
while the spatial one may be not. However, in our new schemes, we
use the CRK method instead of the symplectic method for {the} time
integration. In this subsection, we consider a class of convenient
methods for the spatial discretization under the periodic boundary
condition.  They are the Fourier spectral, the pseudospectral, and
the wavelet collocation method (see, e.g.
\cite{Reich2001,Chen2001,Zhu2011}). A common
characteristic of the three methods is the substitution of a
skew-symmetric differential matrix $D$ for the operator
$\partial_{x}$ . For example, assuming $z(x_{0},t)=z(x_{0}+L,t)$,
\eqref{ONEDIM} becomes a system of ODEs in time after the
pseudospectral spatial discretization :
\begin{equation}\label{SDS}
M\frac{d}{dt}z_{j}+K\sum_{k=0}^{N-1}D_{jk}z_{k}=\nabla_{z}
S(z_{j},x_{j}),
\end{equation}
for $j=0,1,\ldots,N-1,$
where $N$ is an even integer, $x_{j}=x_{0}+j\Delta
x,j=0,1,\cdots ,N-1,\Delta x=\frac{L}{N},$ $z_{j}\approx
z(x_{j},t),\ D$ is a skew-symmetric matrix whose entries are
determined by (see, e.g. \cite{Chen2001})

\begin{equation*}
D_{jk}=\left\{\begin{aligned}
&\frac{\pi}{L}(-1)^{j+k}cot(\pi\frac{x_{j}-x_{k}}{L}),\ j\neq k,\\
&0,\quad\quad\quad\quad\quad\quad\quad\quad\quad j=k.\\
\end{aligned}\right.
\end{equation*}

Multiplying both sides of \eqref{SDS} by
$\frac{d}{dt}z_{j}^{\intercal}$, we get $N$ semi-discrete ECLs
(see, e.g. \cite{Chen2010,Kong2013}):
\begin{equation}\label{sd_ECLs}
\frac{d}{dt}E_{j}+\sum_{k=0}^{N-1}D_{jk}F_{jk}=0,
\end{equation}
for $j=0,1,\ldots,N-1,$ where
\begin{equation*}
\begin{aligned}
&E_{j}=S(z_{j},x_{j})-\frac{1}{2}z_{j}^{\intercal}K\sum_{k=0}^{N-1}D_{jk}z_{k},\\
&F_{jk}=\frac{1}{2}z_{k}^{\intercal}K\frac{d}{dt}z_{j}+\frac{1}{2}z_{j}^{\intercal}K\frac{d}{dt}z_{k}.\\
\end{aligned}
\end{equation*}
{The term $\sum_{k=0}^{N-1}D_{jk}F_{jk}$} can be considered as the
discrete $\partial_{x}F(z_{j})$ :
\begin{equation}\label{SDECL}\sum_{k=0}^{N-1}D_{jk}F_{jk}=\frac{1}{2}\delta_{x}z_{j}^{\intercal}K\frac{d}{dt}z_{j}+
\frac{1}{2}z_{j}^{\intercal}K\frac{d}{dt}\delta_{x}z_{j}\approx\frac{1}{2}\partial_{x}z(x_{j},t)^{\intercal}K\frac{d}{dt}z_{j}+
\frac{1}{2}z_{j}^{\intercal}K\frac{d}{dt}\partial_{x}z(x_{j},t)=\partial_{x}F(z_{j}),\end{equation}
where $\sum_{k=0}^{N-1}D_{jk}z_{k}=\delta_{x}z_{j}\approx\partial_{x}z(x_{j},t).$

If $S$ is independent of the variable $x$, then $N$ semi-discrete
MSCLs {(see, e.g. {\cite{Reich2001,Chen2001}}) also hold}:
\begin{equation*}
\begin{aligned}
&\frac{d}{dt}\omega_{j}+\sum_{k=0}^{N-1}D_{jk}\kappa_{jk}=0,\\
&\omega_{j}=\frac{1}{2}dz_{j}\wedge Mdz_{j},\\
&\kappa_{jk}=\frac{1}{2}(dz_{j}\wedge Kdz_{k}+dz_{k}\wedge Kdz_{j}),\\
\end{aligned}
\end{equation*}
{for $j=0,1,\ldots,N-1$.} Here $\sum_{k=0}^{N-1}D_{jk}\kappa_{jk}$
(the discrete $\partial_{x}\kappa(z_{j})$) can be comprehended in a
similar way to \eqref{SDECL}.

After the temporal discretization using the CRK method
\eqref{CRK}, the full discrete scheme can be written as follows:
\begin{equation}\label{S}
\left\{\begin{aligned}
&z_{j}^{\tau}=z_{j}^{0}+\Delta t\int_{0}^{1}A_{\tau,\sigma}\delta_{t}z_{j}^{\sigma}d\sigma,\\
&z_{j}^{1}=z_{j}^{0}+\Delta t\int_{0}^{1}\delta_{t}z_{j}^{\sigma}d\sigma,\\
&\delta_{x}z_{j}^{\tau}=\sum_{k=0}^{N-1}D_{jk}z_{k}^{\tau},\\
&M\delta_{t}z_{j}^{\tau}+K\delta_{x}z_{j}^{\tau}=\nabla_{z}S(z_{j}^{\tau},x_{j}),\\
\end{aligned}\right.
\end{equation}
{for $j=0,1,\ldots,N-1$,} where $z_{j}^{\tau}\approx
z(x_{j},t_{0}+\tau\Delta
t),\delta_{t}z_{j}^{\tau}\approx\partial_{t}z(x_{j},t_{0}+\tau\Delta
t)$ are polynomials in $\tau$. For the energy-preserving property of
the CRK method, we expect this scheme to preserve some discrete
ECLs. Firstly, note that $b_{i}=\int_{0}^{1}l_{i}(\tau)d\tau$. For
convenience, we denote
$$\frac{1}{b_{i}}\int_{0}^{1}l_{i}(\tau)f(\tau)d\tau$$
(i.e. the weighted average of a function $f$ with the weight
function $l_{i}(\tau)$) as $\langle f\rangle_{i}$ in the remainder
of our paper. Obviously, $\langle\cdot\rangle_{i}$ {is a linear
operator.}

The next theorem shows the $N$-discrete local energy conservation law
of \eqref{S}.

\begin{mytheo}\label{DECL}
{The} scheme \eqref{S} exactly conserves the $N$-discrete local
energy conservation law:

\begin{equation}\label{13}
\frac{E_{j}^{1}-E_{j}^{0}}{\Delta t}+\sum_{k=0}^{N-1}D_{jk}\bar{F}_{jk}=0,
\end{equation}
for $j=0,1,\ldots,N-1,$ where
\begin{equation*}
\begin{aligned}
&E_{j}^{\alpha}=S(z_{j}^{\alpha},x_{j})-\frac{1}{2}z_{j}^{\alpha\intercal}K\delta_{x}z_{j}^{\alpha},\alpha=0,1,\\
&\bar{F}_{jk}=\frac{1}{2}\sum_{i=1}^{s}b_{i}(\langle z_{j}\rangle_{i}^{\intercal}K\langle\delta_{t}z_{k}\rangle_{i}
+\langle z_{k}\rangle_{i}^{\intercal}K\langle\delta_{t}z_{j}\rangle_{i}).\\
\end{aligned}
\end{equation*}
\end{mytheo}

By summing the identities \eqref{13} from $j=0$ to $N-1,$ on
noticing that $\bar{F}_{jk}$ is symmetric with respect to $j,k$ and
$D_{jk}$ is anti-symmetric with respect to $j,k,$ the discrete ECLs
lead to the global energy conservation:
\begin{equation}\label{GEL}
\Delta x\sum_{j=0}^{N-1}E_{j}^{1}-\Delta x\sum_{j=0}^{N-1}E_{j}^{0}=-\Delta x\Delta t\sum_{j=0}^{N-1}\sum_{k=0}^{N-1}D_{jk}\bar{F}_{jk}=0.
\end{equation}
If we evaluate the integrals of $\bar{F}_{jk}$ by the generating
quadrature formula of the CRK method, we have
\begin{equation*}
\bar{F}_{jk}\approx\frac{1}{2}\sum_{i=1}^{s}b_{i}(z_{j}^{c_{i}\intercal}K\delta_{t}z_{k}^{c_{i}}+z_{k}^{c_{i}\intercal}K\delta_{t}z_{j}^{c_{i}}).
\end{equation*}

\begin{proof}
First of all, note that the discrete differential
operator $\delta_{x}$ is linear, thus it holds that
\begin{equation}
\partial_{\tau}\delta_{x}z_{j}^{\tau}=\delta_{x}\partial_{\tau}z_{j}^{\tau}.
\end{equation}
It follows from \eqref{S} that
\begin{equation}\begin{aligned}
&\partial_{\tau}z_{j}^{\tau}=\Delta
t\int_{0}^{1}\sum_{i=1}^{s}\frac{1}{b_{i}}l_{i}(\tau)l_{i}(\sigma)\delta_{t}z_{j}^{\sigma}d\sigma=\Delta
t\sum_{i=1}^{s}l_{i}(\tau)\langle\delta_{t}z_{j}\rangle_{i}.
\end{aligned}\end{equation}
Then we have
\begin{equation}\label{18}
\begin{aligned}
&S(z_{j}^{1},x_{j})-S(z_{j}^{0},x_{j})\\
&=\int_{0}^{1}\partial_{\tau}z_{j}^{\tau\intercal}\nabla_{z}S(z_{j}^{\tau},x_{j})d\tau
=\Delta t\sum_{i=1}^{s}b_{i}\langle\delta_{t}z_{j}\rangle_{i}^{\intercal}\langle\nabla_{z}S_{j}\rangle_{i},\\
\end{aligned}
\end{equation}

\begin{equation}\label{19}
\begin{aligned}
&z_{j}^{1}K\delta_{x}z_{j}^{1}-z_{j}^{0}K\delta_{x}z_{j}^{0}\\
&=\int_{0}^{1}\partial_{\tau}(z_{j}^{\tau\intercal}K\delta_{x}z_{j}^{\tau})d\tau=
\int_{0}^{1}(\partial_{\tau}z_{j}^{\tau\intercal}K\delta_{x}z_{j}^{\tau}+z_{j}^{\tau\intercal}K\delta_{x}\partial_{\tau}z_{j}^{\tau})d\tau\\
&=\Delta t\sum_{i=1}^{s}b_{i}\langle\delta_{t}z_{j}\rangle_{i}^{\intercal}K\langle\delta_{x}z_{j}\rangle_{i}
+\Delta t\sum_{i=1}^{s}b_{i}\langle z_{j}\rangle_{i}^{\intercal}K\langle\delta_{x}\delta_{t}z_{j}\rangle_{i}.\\
\end{aligned}
\end{equation}
With \eqref{18} and \eqref{19}, it follows from
$$\mathbf{a}^{\intercal}M\mathbf{a}=0,\quad \mathbf{a}^{\intercal}M\mathbf{b}=-\mathbf{b}^{\intercal}M\mathbf{a}$$
that, for  $\mathbf{a},\mathbf{b}\in \mathbb{R}^d,$ we have
\begin{equation}
\begin{aligned}
&(E_{j}^{1}-E_{j}^{0})/\Delta t\\
&=(S(z_{j}^{1},x_{j})-S_(z_{j}^{0},x_{j})-\frac{1}{2}(z_{j}^{1\intercal}K\delta_{x}z_{j}^{1}-z_{j}^{0\intercal}K\delta_{x}z_{j}^{0}))/\Delta t\\
&=\sum_{i=1}^{s}b_{i}\langle\delta_{t}z_{j}\rangle_{i}^{\intercal}\langle M\delta_{t}z_{j}+K\delta_{x}z_{j}\rangle_{i}
-\frac{1}{2}\sum_{i=1}^{s}b_{i}\langle\delta_{t}z_{j}\rangle_{i}^{\intercal}K\langle\delta_{x}z_{j}\rangle_{i}
-\frac{1}{2}\sum_{i=1}^{s}b_{i}\langle z_{j}\rangle_{i}^{\intercal}K\langle\delta_{x}\delta_{t}z_{j}\rangle_{i}\\
&=\frac{1}{2}\sum_{i=1}^{s}b_{i}\langle\delta_{t}z_{j}\rangle_{i}^{\intercal}K\langle\delta_{x}z_{j}\rangle_{i}
-\frac{1}{2}\sum_{i=1}^{s}b_{i}\langle z_{j}\rangle_{i}^{\intercal}K\langle\delta_{x}\delta_{t}z_{j}\rangle_{i}\\
&=\frac{1}{2}\sum_{k=0}^{N-1}D_{jk}\sum_{i=1}^{s}b_{i}\langle\delta_{t}z_{j}\rangle_{i}^{\intercal}K\langle z_{k}\rangle_{i}
-\frac{1}{2}\sum_{k=0}^{N-1}D_{jk}\sum_{i=1}^{s}b_{i}\langle z_{j}\rangle_{i}^{\intercal}K\langle\delta_{t}z_{k}\rangle_{i}\\
&=-\sum_{k=0}^{N-1}D_{jk}\bar{F}_{jk}\\
\end{aligned}
\end{equation}\qed

\end{proof}

Note that a crucial property of the pseudospectral method is
replacing the operator $\partial_{x}$ with a linear and
skew-symmetric differential matrix. Fortunately, this property is
shared by spectral methods and wavelet collocation methods, hence
our procedure of constructing the local energy-preserving scheme can
be applied to them without any trouble.

Our approach can also be easily generalized to high dimensional
problems. For example, we consider the following equation {\color{blue}:}
\begin{equation}\label{HDE}
Mz_{t}+Kz_{x}+Lz_{y}=\nabla_{z}S(z,x,y).
\end{equation}
The ECL of this equation is:
\begin{equation}\label{2DECL}
\partial_{t}E+\partial_{x}F+\partial_{y}G=0,
\end{equation}
where
$$E=S(z,x,y)-\frac{1}{2}z^{\intercal}Kz_{x}-\frac{1}{2}z^{\intercal}Lz_{y},\quad F=\frac{1}{2}z^{\intercal}Kz_{t},\quad G=\frac{1}{2}z^{\intercal}Lz_{t}.$$

Applying a CRK method to {$t$-direction} and a pseudospectral method
to $x$ and $y$ directions ( under the periodic
boundary condition
$z(x_{0},y,t)=z(x_{0}+L_{1},y,t),z(x,y_{0},t)=z(x,y_{0}+L_{2},t)$ )
gives the following full discrete scheme:
\begin{equation}\label{FDS}
\left\{\begin{aligned}
&z_{jl}^{\tau}=z_{jl}^{0}+\Delta t\int_{0}^{1}A_{\tau,\sigma}\delta_{t}z_{jl}^{\sigma}d\sigma,\\
&z_{jl}^{1}=z_{jl}^{0}+\Delta t\int_{0}^{1}\delta_{t}z_{jl}^{\sigma}d\sigma,\\
&\delta_{x}z_{jl}^{\tau}=\sum_{k=0}^{N-1}(D_{x})_{jk}z_{kl}^{\tau},\\
&\delta_{y}z_{jl}^{\tau}=\sum_{m=0}^{M-1}(D_{y})_{lm}z_{jm}^{\tau},\\
&M\delta_{t}z_{jl}^{\tau}+K\delta_{x}z_{jl}^{\tau}+L\delta_{y}z_{jl}^{\tau}=\nabla_{z}S(z_{jl}^{\tau},x_{j},y_{l}),\\
\end{aligned}\right.
\end{equation}
{for $j=0,1,\ldots,N-1,\quad l=0,1,\ldots,M-1$}, where
$z_{jl}^{\tau}\approx z(x_{j},y_{l},t_{0}+\tau\Delta t),\
\delta_{t}z_{jl}^{\tau}\approx
\partial_{t}z(x_{j},y_{l},t_{0}+\tau\Delta
t)$ are polynomials in $\tau,\ x_{j}=x_{0}+j\Delta x,\
y_{l}=y_{0}+l\Delta y,\Delta x=\frac{L_{1}}{N}, \Delta
y=\frac{L_{2}}{M}.$ Both $D_{x}$ and $D_{y}$ are pseudospectral
differential matrices related to $x$ and $y$ directions
respectively.

The next theorem presents the discrete local energy conservation laws
of \eqref{FDS}.

\begin{mytheo}\label{HDECL}
The scheme \eqref{FDS} exactly conserves the $NM$-discrete local
energy conservation law:
\begin{equation}\label{2DDECL}
\frac{E_{jl}^{1}-E_{jl}^{0}}{\Delta t}+\sum_{k=0}^{N-1}(D_{x})_{jk}\bar{F}_{jk,l}+\sum_{m=0}^{M-1}(D_{y})_{lm}\bar{G}_{j,lm}=0,
\end{equation}

for $j=0,1,\ldots,N-1,\quad l=0,1,\ldots,M-1,$ where
\begin{equation*}
\begin{aligned}
&E_{j}^{\alpha}=S(z_{jl}^{\alpha},x_{j},y_{l})-\frac{1}{2}z_{jl}^{\alpha\intercal}K\delta_{x}z_{jl}^{\alpha}-
\frac{1}{2}z_{jl}^{\alpha\intercal}L\delta_{y}z_{jl}^{\alpha},\alpha=0,1,\\
&\bar{F}_{jk,l}=\frac{1}{2}\sum_{i=1}^{s}b_{i}(\langle z_{jl}\rangle_{i}^{\intercal}K\langle\delta_{t}z_{kl}\rangle_{i}
+\langle z_{kl}\rangle_{i}^{\intercal}K\langle\delta_{t}z_{jl}\rangle_{i}),\\
&\bar{G}_{j,lm}=\frac{1}{2}\sum_{i=1}^{s}b_{i}(\langle z_{jl}\rangle_{i}^{\intercal}L\langle\delta_{t}z_{jm}\rangle_{i}
+\langle z_{jm}\rangle_{i}^{\intercal}L\langle\delta_{t}z_{jl}\rangle_{i}).\\
\end{aligned}
\end{equation*}
\end{mytheo}

Since the proof of Theorem 3.2 is very similar to that of Theorem
3.1, we omit the details here.

Summing the {identities} \eqref{2DDECL} over all space grid points,
on noticing that $\bar{F}_{jk,l}$ is symmetric with respect to
$j,k$, and $(D_{x})_{jk}$ is anti-symmetric with respect to $j,k,$
$\bar{G}_{j,lm}$ is symmetric with respect to $l,m$, and
$(D_{y})_{lm}$ is anti-symmetric with respect to $l,m,$ again, we
obtain the global energy conservation :
\begin{equation}
\Delta x\Delta y\sum_{j=0}^{N-1}\sum_{l=0}^{M-1}E_{jl}^{1}-\Delta x\Delta y\sum_{j=0}^{N-1}\sum_{l=0}^{M-1}E_{jl}^{0}
=-\Delta t\Delta x\Delta y\sum_{j=0}^{N-1}\sum_{l=0}^{M-1}\sum_{k=0}^{N-1}(D_{x})_{jk}\bar{F}_{jk,l}-\Delta t\Delta x\Delta y\sum_{j=0}^{N-1}\sum_{l=0}^{M-1}\sum_{m=0}^{M-1}(D_{y})_{lm}\bar{G}_{j,lm}=0.
\end{equation}

\subsection{Gauss-Legendre collocation spatial discretization}
In multi-symplectic algorithms, another class of methods
{frequently} applied to spatial discretization is the Gauss-Legendre
(GL) collocation method. We assume that the Butcher tableau of the
GL method is:
\begin{equation}\label{GL}
\begin{tabular}
[c]{c|ccc}%
$\tilde{c}_{1}$ & $\tilde{a}_{11}$ & $\ldots$ & $\tilde{a}_{1r}$\\
$\vdots$ & $\vdots$ & $\ddots$ & $\vdots$\\
$\tilde{c}_{r}$ & $\tilde{a}_{r1}$ & $\cdots$ & $\tilde{a}_{rr}$\\\hline &
$\raisebox{-1.3ex}[0.5pt]{$\tilde{b}_{1}$}$ &
\raisebox{-1.3ex}{$\cdots$} &
$\raisebox{-1.3ex}[0.5pt]{$\tilde{b}_{r}$}$\\
\end{tabular}
\end{equation}
After {the} spatial discretization using the GL method \eqref{GL}
and {the} temporal discretization using the CRK, we obtain the full
discrete scheme of \eqref{ONEDIM} :
\begin{equation}\label{FDSCHEME}
\left\{\begin{aligned}
&z_{n,j}^{\tau}=z_{n,j}^{0}+\Delta t\int_{0}^{1}A_{\tau,\sigma}\delta_{t}z_{n,j}^{\sigma}d\sigma,\\
&z_{n,j}^{1}=z_{n,j}^{0}+\Delta t\int_{0}^{1}\delta_{t}z_{n,j}^{\sigma}d\sigma,\\
&z_{n,j}^{\tau}=z_{n}^{\tau}+\Delta x\sum_{k=1}^{r}\tilde{a}_{jk}\delta_{x}z_{n,k}^{\tau},\\
&z_{n+1}^{\tau}=z_{n}^{\tau}+\Delta x\sum_{j=1}^{r}\tilde{b}_{j}\delta_{x}z_{n,j}^{\tau},\\
&M\delta_{t}z_{n,j}^{\tau}+K\delta_{x}z_{n,j}^{\tau}=\nabla_{z}S(z_{n,j}^{\tau},x_{n}+\tilde{c}_{j}\Delta x),\\
\end{aligned}\right.
\end{equation}
for $j=1,2,\ldots,r,$ where $z_{n}^{\tau}\approx
z(x_{n},t_{0}+\tau\Delta t),z_{n+1}^{\tau}\approx z(x_{n}+\Delta
x,t_{0}+\tau\Delta t), z_{n,j}^{\tau}\approx
z(x_{n}+\tilde{c}_{j}\Delta x,t_{0}+\tau\Delta t),
\delta_{t}z_{n,j}^{\tau}\approx
\partial_{t}z(x_{n}+\tilde{c}_{j}\Delta x,t_{0}+\tau\Delta t),
\delta_{x}z_{n,j}^{\tau}\approx
\partial_{x}z(x_{n}+\tilde{c}_{j}\Delta x,t_{0}+\tau\Delta
t)$ are polynomials in $\tau.$
This is a local scheme on the box $[x_{n},x_{n}+\Delta
x]\times[t_{0},t_{0}+\Delta t]$.
To show that \eqref{FDSCHEME}
exactly conserves the discrete ECL, we should make sure that there
is some law of commutation between $\delta_{t}$ and $\delta_{x}$. To this
end we introduce the following auxiliary system:
\begin{equation}
\begin{aligned}
&\delta_{x}z_{n,j}^{\tau}=\delta_{x}z_{n,j}^{0}+\Delta t\int_{0}^{1}A_{\tau,\sigma}\delta_{t}\delta_{x}z_{n,j}^{\sigma}d\sigma,\\
&\delta_{x}z_{n,j}^{1}=\delta_{x}z_{n,j}^{0}+\Delta t\int_{0}^{1}\delta_{t}\delta_{x}z_{n,j}^{\sigma}d\sigma,\\
&\delta_{t}z_{n,j}^{\tau}=\delta_{t}z_{n}^{\tau}+\Delta x\sum_{k=1}^{r}\tilde{a}_{jk}\delta_{x}\delta_{t}z_{n,k}^{\tau},\\
&\delta_{t}z_{n+1}^{\tau}=\delta_{t}z_{n}^{\tau}+\Delta x\sum_{j=1}^{r}\tilde{b}_{j}\delta_{x}\delta_{t}z_{n,j}^{\tau},\\
\end{aligned}
\end{equation}
for $j=1,2,\ldots,r,$ where
$\delta_{t}\delta_{x}z_{n,j}^{\sigma}\approx\partial_{t}\partial_{x}z(x_{n}+\tilde{c}_{j}\Delta
x,t_{0}+\sigma
t),\delta_{x}\delta_{t}z_{n,j}^{\sigma}\approx\partial_{x}\partial_{t}z(x_{n}+\tilde{c}_{j}\Delta
x,t_{0}+\sigma t)$. Then
\begin{equation}\label{29}
\begin{aligned}
z_{n,j}^{\tau}
&=z_{n,j}^{0}+\Delta t\int_{0}^{1}A_{\tau,\sigma}\delta_{t}z_{n,j}^{\sigma}d\sigma\\
&=z_{n,j}^{0}+\Delta t\int_{0}^{1}A_{\tau,\sigma}(\delta_{t}z_{n}^{\sigma}+\Delta x\sum_{k=1}^{r}\tilde{a}_{jk}\delta_{x}\delta_{t}z_{n,k}^{\sigma})d\sigma\\
&=z_{n,j}^{0}+z_{n}^{\tau}-z_{n}^{0}+\Delta t\Delta x\sum_{k=1}^{r}\tilde{a}_{jk}\int_{0}^{1}A_{\tau,\sigma}\delta_{x}\delta_{t}z_{n,k}^{\sigma}d\sigma.\\
\end{aligned}
\end{equation}
Likewise,
\begin{equation}\label{30}
\begin{aligned}
z_{n,j}^{\tau}
&=z_{n}^{\tau}+z_{n,j}^{0}-z_{n}^{0}+\Delta x\Delta t\sum_{k=1}^{r}\tilde{a}_{jk}\int_{0}^{1}A_{\tau,\sigma}\delta_{t}\delta_{x}z_{n,k}^{\sigma}d\sigma.\\
\end{aligned}
\end{equation}
\eqref{29}, \eqref{30} lead to
\begin{equation*}
\begin{aligned}
&\sum_{k=1}^{r}\tilde{a}_{jk}\int_{0}^{1}A_{\tau,\sigma}\delta_{x}\delta_{t}z_{n,k}^{\sigma}d\sigma=
\sum_{k=1}^{r}\tilde{a}_{jk}\int_{0}^{1}A_{\tau,\sigma}\delta_{t}\delta_{x}z_{n,k}^{\sigma}d\sigma.\\
\end{aligned}
\end{equation*}
Since the matrix $(\tilde{a}_{jk})_{1\leq j,\ k\leq r}$ is invertible, we have
\begin{equation}\label{32}
\int_{0}^{1}A_{\tau,\sigma}\delta_{x}\delta_{t}z_{n,k}^{\sigma}d\sigma=\int_{0}^{1}A_{\tau,\sigma}\delta_{t}\delta_{x}z_{n,k}^{\sigma}d\sigma,
\quad\tau\in[0,1],
\end{equation}
for $k=1,2,\ldots,r.$ Taking derivatives with respect to $\tau$ on
both sides of \eqref{32}, {we} arrive at
\begin{equation*}
\int_{0}^{1}(\sum_{i=1}^{s}\frac{1}{b_{i}}l_{i}(\tau)l_{i}(\sigma))\delta_{x}\delta_{t}z_{n,k}^{\sigma}d\sigma=
\int_{0}^{1}(\sum_{i=1}^{s}\frac{1}{b_{i}}l_{i}(\tau)l_{i}(\sigma))\delta_{t}\delta_{x}z_{n,k}^{\sigma}d\sigma,\quad
\tau\in[0,1],
\end{equation*}
for $ k=1,2,\ldots ,r$. Finally, setting
$\tau=c_{1},\ldots,c_{s},$ we have the following lemma:

\begin{lem}\label{3.3}
The following discrete commutability between $\delta_{t}$ and $\delta_{x}$ holds:
\begin{equation}
\langle\delta_{x}\delta_{t}z_{n,j}\rangle_{i}=
\langle\delta_{t}\delta_{x}z_{n,j}\rangle_{i},
\end{equation}
for $i=1,\ldots,s,\quad j=1,2,\ldots,r.$
\end{lem}

\begin{mytheo}
The scheme \eqref{FDSCHEME} conserves the following discrete local
energy conservation law :

\begin{equation}\label{33}
\Delta x\sum_{j=1}^{r}\tilde{b}_{j}(E_{n,j}^{1}-E_{n,j}^{0})+\Delta t(\bar{F}_{n+1}-\bar{F}_{n})=0,
\end{equation}
where
\begin{equation*}
\begin{aligned}
&E_{n,j}^{\alpha}=S(z_{n,j}^{\alpha},x_{n}+\tilde{c}_{j}\Delta x)-\frac{1}{2}z_{n,j}^{\alpha\intercal}K\delta_{x}z_{n,j}^{\alpha},\alpha=0,1,\\
&\bar{F}_{\beta}=\frac{1}{2}\sum_{i=1}^{s}b_{i}\langle z_{\beta}\rangle_{i}^{\intercal}K\langle\delta_{t}z_{\beta}\rangle_{i},\beta=n,n+1.\\
\end{aligned}
\end{equation*}
\end{mytheo}

\begin{proof}
It follows from the first equation of \eqref{FDSCHEME} that,
\begin{equation}
\partial_{\tau}z_{n,j}^{\tau}=\Delta t\sum_{i=1}^{s}l_{i}(\tau)\langle\delta_{t}z_{n,j}\rangle_{i}.
\end{equation}

The result in the temporal direction is almost the same as the
pseudospectral case :
\begin{equation*}
\begin{aligned}
&S(z_{n,j}^{1},x_{n}+\tilde{c}_{j}\Delta x)-S(z_{n,j}^{0},x_{n}+\tilde{c}_{j}\Delta x)
=\Delta t\sum_{i=1}^{s}b_{i}\langle\delta_{t}z_{n,j}\rangle_{i}^{\intercal}\langle\nabla S_{n,j}\rangle_{i},\\
\end{aligned}
\end{equation*}

\begin{equation*}
\begin{aligned}
&z_{n,j}^{1}K\delta_{x}z_{n,j}^{1}-z_{n,j}^{0}K\delta_{x}z_{n,j}^{0}
=\Delta t\sum_{i=1}^{s}b_{i}\langle\delta_{t}z_{n,j}\rangle_{i}^{\intercal}K\langle\delta_{x}z_{n,j}\rangle_{i}
+\Delta t\sum_{i=1}^{s}b_{i}\langle z_{n,j}\rangle_{i}^{\intercal}K \langle\delta_{t}\delta_{x}z_{n,j}\rangle_{i}.\\
\end{aligned}
\end{equation*}

Hence
\begin{equation}\label{36}
\begin{aligned}
&E_{n,j}^{1}-E_{n,j}^{0}\\
&=S(z_{n,j}^{1},x_{n}+\tilde{c}_{j}\Delta x)-S(z_{n,j}^{0},x_{n}+\tilde{c}_{j}\Delta x)-\frac{1}{2}(z_{n,j}^{1}K\delta_{x}z_{n,j}^{1}-z_{n,j}^{0}K\delta_{x}z_{n,j}^{0})\\
&=\Delta t\sum_{i=1}^{s}b_{i}\langle\delta_{t}z_{n,j}\rangle_{i}^{\intercal}\langle M\delta_{t}z_{n,j}+K\delta_{x}z_{n,j}\rangle_{i}
-\frac{1}{2}\Delta t\sum_{i=1}^{s}b_{i}\langle\delta_{t}z_{n,j}\rangle_{i}^{\intercal}K\langle\delta_{x}z_{n,j}\rangle_{i}
-\frac{1}{2}\Delta t\sum_{i=1}^{s}b_{i}\langle z_{n,j}\rangle_{i}^{\intercal}K \langle\delta_{t}\delta_{x}z_{n,j}\rangle_{i}\\
&=\frac{1}{2}\Delta t\sum_{i=1}^{s}b_{i}\langle\delta_{t}z_{n,j}\rangle_{i}^{\intercal}K\langle\delta_{x}z_{n,j}\rangle_{i}
-\frac{1}{2}\Delta t\sum_{i=1}^{s}b_{i}\langle z_{n,j}\rangle_{i}^{\intercal}K\langle\delta_{t}\delta_{x}z_{n,j}\rangle_{i}.\\
\end{aligned}
\end{equation}
On the other hand,
\begin{equation}\label{l49}
\begin{aligned}
&z_{n+1}^{\tau\intercal}K\delta_{t}z_{n+1}^{\sigma}-z_{n}^{\tau\intercal}K\delta_{t}z_{n}^{\sigma}\\
&=(z_{n}^{\tau\intercal}+\Delta x\sum_{j=1}^{r}\tilde{b}_{j}\delta_{x}z_{n,j}^{\tau\intercal})K(\delta_{t}z_{n}^{\sigma}+\Delta x\sum_{j=1}^{r}\tilde{b}_{j}\delta_{x}\delta_{t}z_{n,j}^{\sigma})-z_{n}^{\tau\intercal}K\delta_{t}z_{n}^{\sigma}\\
&=\Delta x\sum_{j=1}^{r}\tilde{b}_{j}z_{n}^{\tau\intercal}K\delta_{x}\delta_{t}z_{n,j}^{\sigma}+\Delta x\sum_{j=1}^{r}\tilde{b}_{j}\delta_{x}z_{n,j}^{\tau\intercal}K\delta_{t}z_{n}^{\sigma}+
\Delta x^{2}\sum_{j,k=1}^{r}\tilde{b}_{j}\tilde{b}_{k}\delta_{x}z_{n,j}^{\tau\intercal}K\delta_{x}\delta_{t}z_{n,k}^{\sigma}\\
&=\Delta x\sum_{j=1}^{r}\tilde{b}_{j}(z_{n,j}^{\tau\intercal}-\Delta
x\sum_{k=1}^{r}\tilde{a}_{jk}\delta_{x}z_{n,k}^{\tau\intercal})K\delta_{x}\delta_{t}z_{n,j}^{\sigma}+\Delta
x\sum_{j=1}^{r}\tilde{b}_{j}\delta_{x}z_{n,j}^{\tau\intercal}K(\delta_{t}z_{n,j}^{\sigma}-\Delta
x\sum_{k=1}^{r}\tilde{a}_{jk}\delta_{x}\delta_{t}z_{n,k}^{\sigma})\\&\quad+
\Delta x^{2}\sum_{j,k=1}^{r}\tilde{b}_{j}\tilde{b}_{k}\delta_{x}z_{n,j}^{\tau\intercal}K\delta_{x}\delta_{t}z_{n,k}^{\sigma}\\
&=\Delta x\sum_{j=1}^{r}\tilde{b}_{j}z_{n,j}^{\tau\intercal}K\delta_{x}\delta_{t}z_{n,j}^{\sigma}+\Delta x\sum_{j=1}^{r}\tilde{b}_{j}\delta_{x}z_{n,j}^{\tau\intercal}K\delta_{t}z_{n,j}^{\sigma}+
\Delta x^{2}\sum_{j,k=1}^{r}(\tilde{b}_{j}\tilde{b}_{k}-\tilde{b}_{j}\tilde{a}_{jk}-\tilde{b}_{k}\tilde{a}_{kj})\delta_{x}z_{n,j}^{\tau\intercal}K\delta_{x}\delta_{t}z_{n,k}^{\sigma}\\
&=\Delta x\sum_{j=1}^{r}\tilde{b}_{j}z_{n,j}^{\tau\intercal}K\delta_{x}\delta_{t}z_{n,j}^{\sigma}+\Delta x\sum_{j=1}^{r}\tilde{b}_{j}\delta_{x}z_{n,j}^{\tau\intercal}K\delta_{t}z_{n,j}^{\sigma}.\\
\end{aligned}
\end{equation}
It follows from \eqref{l49} that
\begin{equation}\label{388}
\begin{aligned}
&\bar{F}_{n+1}-\bar{F}_{n}\\
&=\frac{1}{2}\sum_{i=1}^{s}b_{i}(\langle z_{n+1}\rangle_{i}^{\intercal}K\langle\delta_{t}z_{n+1}\rangle_{i}-
\langle z_{n}\rangle_{i}^{\intercal}K\langle\delta_{t}z_{n}\rangle_{i})\\
&=\frac{1}{2}\Delta x\sum_{j=1}^{r}\sum_{i=1}^{s}\tilde{b}_{j}b_{i}(\langle z_{n,j}\rangle_{i}^{\intercal} K\langle\delta_{x}\delta_{t}z_{n,j}\rangle_{i}+\langle\delta_{x}z_{n,j}\rangle_{i}^{\intercal}K\langle\delta_{t}z_{n,j}\rangle_{i})\\
\end{aligned}
\end{equation}
From \eqref{36} and \eqref{388}, using Lemma \ref{3.3},
we have
\begin{equation}
\begin{aligned}
&\Delta x\sum_{j=1}^{r}\tilde{b}_{j}(E_{n,j}^{1}-E_{n,j}^{0})+\Delta t(\bar{F}_{n+1}-\bar{F}_{n})\\
&=\frac{1}{2}\Delta t\Delta x\sum_{i=1}^{s}\sum_{j=1}^{r}\tilde{b}_{j}b_{i}(\langle\delta_{t}z_{n,j}\rangle_{i}^{\intercal}K\langle\delta_{x}z_{n,j}\rangle_{i}
-\langle z_{n,j}\rangle_{i}^{\intercal}K\langle\delta_{t}\delta_{x}z_{n,j}\rangle_{i}+\langle z_{n,j}\rangle_{i}^{\intercal} K\langle\delta_{x}\delta_{t}z_{n,j}\rangle_{i}+\langle\delta_{x}z_{n,j}\rangle_{i}^{\intercal}K\langle\delta_{t}z_{n,j}\rangle_{i})\\
&=\frac{1}{2}\Delta x\Delta t\sum_{j=1}^{r}\sum_{i=1}^{s}\tilde{b}_{j}b_{i}\langle z_{n,j}\rangle_{i}^{\intercal} K(\langle\delta_{x}\delta_{t}z_{n,j}\rangle_{i}-\langle\delta_{t}\delta_{x}z_{n,j}\rangle_{i})=0.\\
\end{aligned}
\end{equation}
\end{proof}\qed

Assume that the spatial domain is divided equally into $N$ intervals
and the corresponding grids are $x_{0},x_{1},\ldots,x_{N}$. By
summing the identities \eqref{33} from $n=0$ to $N-1$, we obtain the
global energy conservation of the scheme \eqref{FDSCHEME} under the
periodic boundary condition :
\begin{equation}\label{38}
\Delta x\sum_{n=0}^{N-1}\sum_{j=1}^{r}\tilde{b}_{j}E_{n,j}^{1}-\Delta x\sum_{n=0}^{N-1}\sum_{j=1}^{r}\tilde{b}_{j}E_{n,j}^{0}
=-\Delta t\sum_{n=0}^{N-1}(\bar{F}_{n+1}-\bar{F}_{n})=0,
\end{equation}

The GL spatial discretization is not restricted to the periodic
boundary condition  (PBC). Thus the discrete ECL \eqref{33} is
{\color{red}superior to} the discrete global energy conservation
\eqref{38}. However, discretizing space by high-order GL methods may
lead to singular and massive ODE systems which are expensive to
solve (see, e.g. {\cite{Maclachlan2014,Ryland2007}}). For this
reason, we will not include the scheme \eqref{FDSCHEME} in our
numerical experiments in Section \ref{Numer_Experi},
\ref{Numer_Experi2}.

\section{Local energy-preserving schemes for coupled nonlinear Schr\"{o}dinger
equations}\label{LECNS} An important class of multi-symplectic PDEs
is the (coupled) nonlinear Schr\"{o}dinger equation ((C)NLS). A
great number of them have polynomial nonlinear terms, hence we can
calculate the integrals exactly in our method  {(for example, by
symbol calculations)}. Here we summarily introduce the
multi-symplectic structure of the 2-coupled NLS:

\begin{equation}\label{2CNLS}
\left\{\begin{aligned}
&iu_{t}+i\alpha u_{x}+\frac{1}{2}u_{xx}+(|u|^{2}+\beta|v|^{2})u=0,\\
&iv_{t}-i\alpha v_{x}+\frac{1}{2}v_{xx}+(\beta|u|^{2}+|v|^{2})v=0,\\
\end{aligned}\right.
\end{equation}
where $u,v$ are complex variables,
$i\ $ is the imaginary unit.
Assuming $u=q_{1}+iq_{2}$ and $v=q_{3}+iq_{4},
\partial_{x}q_{i}=2p_{i},\ q_{i}$ are real variables for $i=1,2,3,4,$
we can formulate this equation to a multi-symplectic form (see, e.g.
{\cite{Chen2010}}):
\begin{equation*}
\begin{aligned}
&\left(\begin{array}{cc}\mathbf{J}_{1}&\mathbf{O}\\\mathbf{O}&\mathbf{O}\\\end{array}\right)
z_{t}+\left(\begin{array}{cc}
\mathbf{J}_{2}&\mathbf{I}\\ -\mathbf{I}&\mathbf{O}\\\end{array}\right)z_{x}
=\nabla S(z),\\
\end{aligned}
\end{equation*}
where
\begin{equation*}
\mathbf{J}_{1}=\left(\begin{array}{cccc}0&-1&0&0\\1&0&0&0\\0&0&0&-1\\0&0&1&0\\\end{array}\right),
\mathbf{J}_{2}=\left(\begin{array}{cccc}0&-\alpha&0&0\\ \alpha&0&0&0\\0&0&0&\alpha\\0&0&-\alpha&0\\\end{array}\right),
\mathbf{O}=\left(\begin{array}{cccc}0&0&0&0\\0&0&0&0\\0&0&0&0\\0&0&0&0\\\end{array}\right),
\mathbf{I}=\left(\begin{array}{cccc}1&0&0&0\\0&1&0&0\\0&0&1&0\\0&0&0&1\\\end{array}\right),
\end{equation*}
and
$$z=(q_{1},q_{2},q_{3},q_{4},p_{1},p_{2},p_{3},p_{4})^{\intercal},\quad S=-\frac{1}{4}(q_{1}^{2}+q_{2}^{2})^{2}-\frac{1}{4}(q_{3}^{2}+q_{4}^{2})^{2}-\frac{1}{2}\beta(q_{1}^{2}+q_{2}^{2})(q_{3}^{2}+q_{4}^{2})-
(p_{1}^{2}+p_{2}^{2}+p_{3}^{2}+p_{4}^{2}).$$
The corresponding
energy density $E$ and flux $F$ in the ECL \eqref{1DECL} are:

\begin{equation*}
\begin{aligned}
&E=S-\alpha(q_{2}p_{1}-q_{1}p_{2}+q_{3}p_{4}-q_{4}p_{3})-\frac{1}{2}\sum_{i=1}^{4}(q_{i}\partial_{x}p_{i}-2p_{i}^{2}),\\
&F=\frac{1}{2}(\alpha(q_{2}\partial_{t}q_{1}-q_{1}\partial_{t}q_{2}+q_{3}\partial_{t}q_{4}-q_{4}\partial_{t}q_{3})+
\sum_{i=1}^{4}(q_{i}\partial_{t}p_{i}-p_{i}\partial_{t}q_{i})).\\
\end{aligned}
\end{equation*}
The corresponding MCL of this equation is:
\begin{equation*}
\partial_{t}I+\partial_{x}G=0,
\end{equation*}
where
\begin{equation*}
\begin{aligned}
&I=q_{2}p_{1}-q_{1}p_{2}+q_{4}p_{3}-q_{3}p_{4},\\
&G=S-\frac{1}{2}(q_{2}\partial_{t}q_{1}-q_{1}\partial_{t}q_{2}+q_{4}\partial_{t}q_{3}-q_{3}\partial_{t}q_{4}).\\
\end{aligned}
\end{equation*}
Integrating the ECL and MCL with respect to the variable $x$ under the PBC leads to
the global energy and the momentum conservation:
\begin{equation*}
\int_{x_{0}}^{x_{0}+L}E(x,t)dx=\int_{x_{0}}^{x_{0}+L}E(x,0)dx,\quad \int_{x_{0}}^{x_{0}+L}I(x,t)dx=\int_{x_{0}}^{x_{0}+L}I(x,0)dx.
\end{equation*}
Besides, the global charges of $u$ and $v$ are constant under the PBC:
\begin{equation*}
\int_{x_{0}}^{x_{0}+L}|u(x,t)|^{2}dx=\int_{x_{0}}^{x_{0}+L}|u(x,0)|^{2}dx,\quad
\int_{x_{0}}^{x_{0}+L}|v(x,t)|^{2}dx=\int_{x_{0}}^{x_{0}+L}|v(x,0)|^{2}dx.
\end{equation*}

Applying our discrete procedure to the equation \eqref{2CNLS} gives
the following scheme in vector form:
\begin{equation}\label{DCNLS}
\left\{\begin{aligned}
&q_{1}^{\tau}=q_{1}^{0}+\Delta t\int_{0}^{1}A_{\tau,\sigma}(-\alpha Dq_{1}^{\sigma}-Dp_{2}^{\sigma}
-(((q_{1}^{\sigma})^{\cdot2}+(q_{2}^{\sigma})^{\cdot2})+\beta((q_{3}^{\sigma})^{\cdot2}+(q_{4}^{\sigma})^{\cdot2}))\cdot q_{2}^{\sigma})d\sigma,\\
&q_{2}^{\tau}=q_{2}^{0}+\Delta t\int_{0}^{1}A_{\tau,\sigma}(-\alpha Dq_{2}^{\sigma}+Dp_{1}^{\sigma}
+(((q_{1}^{\sigma})^{\cdot2}+(q_{2}^{\sigma})^{\cdot2})+\beta((q_{3}^{\sigma})^{\cdot2}+(q_{4}^{\sigma})^{\cdot2}))\cdot q_{1}^{\sigma})d\sigma,\\
&q_{3}^{\tau}=q_{3}^{0}+\Delta t\int_{0}^{1}A_{\tau,\sigma}(\alpha Dq_{3}^{\sigma}-Dp_{4}^{\sigma}
-(\beta((q_{1}^{\sigma})^{\cdot2}+(q_{2}^{\sigma})^{\cdot2})+((q_{3}^{\sigma})^{\cdot2}+(q_{4}^{\sigma})^{\cdot2}))\cdot q_{4}^{\sigma})d\sigma,\\
&q_{4}^{\tau}=q_{4}^{0}+\Delta t\int_{0}^{1}A_{\tau,\sigma}(\alpha Dq_{4}^{\sigma}+Dp_{3}^{\sigma}
+(\beta((q_{1}^{\sigma})^{\cdot2}+(q_{2}^{\sigma})^{\cdot2})+((q_{3}^{\sigma})^{\cdot2}+(q_{4}^{\sigma})^{\cdot2}))\cdot q_{3}^{\sigma})d\sigma,\\
&q_{1}^{1}=q_{1}^{0}+\Delta t\int_{0}^{1}(-\alpha Dq_{1}^{\sigma}-Dp_{2}^{\sigma}
-(((q_{1}^{\sigma})^{\cdot2}+(q_{2}^{\sigma})^{\cdot2})+\beta((q_{3}^{\sigma})^{\cdot2}+(q_{4}^{\sigma})^{\cdot2}))\cdot q_{2}^{\sigma})d\sigma,\\
&q_{2}^{1}=q_{2}^{0}+\Delta t\int_{0}^{1}(-\alpha Dq_{2}^{\sigma}+Dp_{1,k}^{\sigma}
+(((q_{1}^{\sigma})^{\cdot2}+(q_{2}^{\sigma})^{\cdot2})+\beta((q_{3}^{\sigma})^{\cdot2}+(q_{4}^{\sigma})^{\cdot2}))\cdot q_{1}^{\sigma})d\sigma,\\
&q_{3}^{1}=q_{3}^{0}+\Delta t\int_{0}^{1}(\alpha Dq_{3}^{\sigma}-Dp_{4}^{\sigma}
-(\beta((q_{1}^{\sigma})^{\cdot2}+(q_{2}^{\sigma})^{\cdot2})+((q_{3}^{\sigma})^{\cdot2}+(q_{4}^{\sigma})^{\cdot2}))\cdot q_{4}^{\sigma})d\sigma,\\
&q_{4}^{1}=q_{4}^{0}+\Delta t\int_{0}^{1}(\alpha Dq_{4}^{\sigma}+Dp_{3}^{\sigma}
+(\beta((q_{1}^{\sigma})^{\cdot2}+(q_{2}^{\sigma})^{\cdot2})+((q_{3}^{\sigma})^{\cdot2}+(q_{4}^{\sigma})^{\cdot2}))\cdot q_{3}^{\sigma})d\sigma,\\
&\delta_{x}q_{i}^{\sigma}=Dq_{i}^{\sigma}=2p_{i}^{\sigma},i=1,2,3,4,\\
\end{aligned}\right.
\end{equation}
where
$q_{i}^{a}=(q_{i,0}^{a},q_{i,1}^{a},\ldots,q_{i,N-1}^{a})^{\intercal},
p_{i}^{a}=(p_{i,0}^{a},p_{i,1}^{a},\ldots,p_{i,N-1}^{a})^{\intercal},a=0,1,\tau$
for $i=1,2,3,4,j=1,2,\ldots,N-1.$
$\ q_{i,j}^{\tau},p_{i,j}^{\tau}$ are polynomials in $\tau.$
The symbols
``$\cdot2$'' and ``$\cdot$'' indicate the entrywise square operation and the
entrywise multiplication operation, respectively.

It can be observed that $p_{i}^{\tau}$ can be eliminated from \eqref{DCNLS}.
If the generating quadrature formula has $s$ nodes, then
$A_{\tau,\sigma}$ is a polynomial of degree $s$ in variable $\tau$,
so are $q_{i,j}^{\tau}$ for $i=1,\ldots,4,j=0,1,\ldots,N-1$. These
polynomials are uniquely determined by their values at $s+1$ points.
For convenience, we choose $0,\frac{1}{s},\frac{2}{s},\ldots ,1.$
Then $q_{i,j}^{\tau}$ can be expressed {as Lagrange interpolating
polynomials based on} these $s+1$ points. Fixing $\tau$ at
$\frac{1}{s},\frac{2}{s},\ldots,\frac{s-1}{s}$, we get a system of
algebraic equations in
$q_{i,j}^{c},c=0,\frac{1}{s},\frac{2}{s},\ldots, 1$. The polynomial
integrals in this system can be calculated accurately. Solving the
algebraic system by an iteration method, we finally obtain the
numerical solution $q_{i,j}^{1}$ .

For example, if we select the CRK method generated by a 2-point GL
quadrature formula, then $q_{i}^{\sigma},p_{i}^{\sigma}$ are vectors
whose entries are polynomials of degree $2$. Thus we
have
$$q_{i}^{\sigma}=q_{i}^{0}\tilde{l}_{1}(\sigma)+q_{i}^{\frac{1}{2}}\tilde{l}_{2}(\sigma)+q_{i}^{1}\tilde{l}_{3}(\sigma),$$
where $\tilde{l}_{1}(\sigma),\tilde{l}_{2}(\sigma),\tilde{l}_{3}(\sigma)$ are Lagrange interpolating polynomials based on the nodes $0,\frac{1}{2},1.$
Let $\tau=\frac{1}{2},$ the first four equations of \eqref{DCNLS} can be written in practical forms :
\begin{equation}\label{60}
\begin{aligned}
&q_{1}^{\frac{1}{2}}=q_{1}^{0}+\Delta t\int_{0}^{1}A_{\frac{1}{2},\sigma}(-\alpha Dq_{1}^{\sigma}-Dp_{2}^{\sigma}
-(((q_{1}^{\sigma})^{\cdot2}+(q_{2}^{\sigma})^{\cdot2})+\beta((q_{3}^{\sigma})^{\cdot2}+(q_{4}^{\sigma})^{\cdot2}))\cdot q_{2}^{\sigma})d\sigma,\\
&q_{2}^{\frac{1}{2}}=q_{2}^{0}+\Delta t\int_{0}^{1}A_{\frac{1}{2},\sigma}(-\alpha Dq_{2}^{\sigma}+Dp_{1}^{\sigma}
+(((q_{1}^{\sigma})^{\cdot2}+(q_{2}^{\sigma})^{\cdot2})+\beta((q_{3}^{\sigma})^{\cdot2}+(q_{4}^{\sigma})^{\cdot2}))\cdot q_{1}^{\sigma})d\sigma,\\
&q_{3}^{\frac{1}{2}}=q_{3}^{0}+\Delta t\int_{0}^{1}A_{\frac{1}{2},\sigma}(\alpha Dq_{3}^{\sigma}-Dp_{4}^{\sigma}
-(\beta((q_{1}^{\sigma})^{\cdot2}+(q_{2}^{\sigma})^{\cdot2})+((q_{3}^{\sigma})^{\cdot2}+(q_{4}^{\sigma})^{\cdot2}))\cdot q_{4}^{\sigma})d\sigma,\\
&q_{4}^{\frac{1}{2}}=q_{4}^{0}+\Delta t\int_{0}^{1}A_{\frac{1}{2},\sigma}(\alpha Dq_{4}^{\sigma}+Dp_{3}^{\sigma}
+(\beta((q_{1}^{\sigma})^{\cdot2}+(q_{2}^{\sigma})^{\cdot2})+((q_{3}^{\sigma})^{\cdot2}+(q_{4}^{\sigma})^{\cdot2}))\cdot q_{3}^{\sigma})d\sigma.\\
\end{aligned}
\end{equation}
After integrating the linear and nonlinear terms about $\sigma,$
\eqref{60} becomes an undetermined system of equations in unknown
vectors
$q_{1}^{\frac{1}{2}},q_{2}^{\frac{1}{2}},q_{3}^{\frac{1}{2}},q_{4}^{\frac{1}{2}},q_{1}^{1},q_{2}^{1},q_{3}^{1},q_{4}^{1}.$
By combining them with the $5th,6th,7th,8th$ equations of
\eqref{DCNLS}, we obtain an entirely determined algebraic system
about them which can be easily solved  by a fixed-point iteration in
practical computations. If the generating quadrature formula has
only one node, for example, the implicit midpoint
rule, then $q_{i}^{\sigma}=(1-\sigma)q_{i}^{0}+\sigma q_{i}^{1}.$ In
this particular case, the first four equations are not necessary to
be taken into account.

According to Theorem \ref{DECL}, the scheme \eqref{DCNLS} preserves
the discrete ECLs:
\begin{equation}
\frac{E_{j}^{1}-E_{j}^{0}}{\Delta
t}+\sum_{k=0}^{N-1}D_{jk}\bar{F}_{jk}=0,
\end{equation}
for $j=0,1,\ldots,N-1,$ where \begin{equation*}\begin{aligned}
&E_{j}^{a}=S_{j}^{a}-\alpha(q_{2,j}^{a}p_{1,j}^{a}-q_{1}^{a}p_{2,j}^{a}+q_{3,j}^{a}p_{4,j}^{a}-q_{4,j}^{a}p_{3,j}^{a})-
\frac{1}{2}\sum_{i=1}^{4}(q_{i,j}^{a}\sum_{k=0}^{N-1}D_{jk}p_{i,k}^{a}-2(p_{i,j}^{a})^{2}),a=0,1,\\
&\bar{F}_{jk}=\frac{1}{2}\sum_{i=1}^{s}b_{i}(\alpha(\langle q_{2,j}\rangle_{i}\langle\delta_{t}q_{1,k}\rangle_{i}-\langle q_{1,j}\rangle_{i}\langle\delta_{t}q_{2,k}\rangle_{i}+\langle q_{3,j}\rangle_{i}\langle\delta_{t}q_{4,k}\rangle_{i}-\langle q_{4,j}\rangle_{i}\langle\delta_{t}q_{3,k}\rangle_{i})+
\sum_{\gamma=1}^{4}(\langle q_{\gamma,j}\rangle_{i}\langle\delta_{t}p_{\gamma,k}\rangle_{i}-\langle p_{\gamma,j}\rangle_{i}\langle\delta_{t}q_{\gamma,k}\rangle_{i})\\
&+\alpha(\langle q_{2,k}\rangle_{i}\langle\delta_{t}q_{1,j}\rangle_{i}-\langle q_{1,k}\rangle_{i}\langle\delta_{t}q_{2,j}\rangle_{i}+\langle q_{3,k}\rangle_{i}\langle\delta_{t}q_{4,j}\rangle_{i}-\langle q_{4,k}\rangle_{i}\langle\delta_{t}q_{3,j}\rangle_{i})+
\sum_{\gamma=1}^{4}(\langle q_{\gamma,k}\rangle_{i}\langle\delta_{t}p_{\gamma,j}\rangle_{i}-\langle p_{\gamma,k}\rangle_{i}\langle
\delta_{t}q_{\gamma,j}\rangle_{i})).\\
\end{aligned}\end{equation*}

\section{Local energy-preserving schemes for 2D nonlinear Schr\"{o}dinger
equations}\label{LE2DS} Another PDE which we pay attention to is the
NLS with two spatial variables:
\begin{equation}\label{HDNLS}
i\psi_{t}+\alpha(\psi_{xx}+\psi_{yy})+V^\prime(|\psi|^{2},x,y)\psi=0.
\end{equation}
The symbol $'$ indicates the derivative of $V$ with respect
to the first variable. Let $\psi=p+iq$, $p$ and $q$ are real and
imaginary parts of $\psi,$ respectively. Introducing
$v=\partial_{x}p,w=\partial_{x}q,a=\partial_{y}p,b=\partial_{y}q,$
we can formulate this equation to the compact form
\eqref{HDE}, where
$$
\begin{aligned}
&M=\left(\begin{array}{cccccc}0&1&0&0&0&0\\-1&0&0&0&0&0\\0&0&0&0&0&0\\0&0&0&0&0&0\\0&0&0&0&0&0\\0&0&0&0&0&0\\\end{array}\right),\quad
K=\left(\begin{array}{cccccc}0&0&-\alpha&0&0&0\\0&0&0&-\alpha&0&0\\ \alpha&0&0&0&0&0\\0&\alpha&0&0&0&0\\0&0&0&0&0&0\\0&0&0&0&0&0\\\end{array}\right),
\quad
L=\left(\begin{array}{cccccc}0&0&0&0&-\alpha&0\\0&0&0&0&0&-\alpha\\0&0&0&0&0&0\\0&0&0&0&0&0\\ \alpha&0&0&0&0&0\\0&\alpha&0&0&0&0\\\end{array}\right),
\end{aligned}
$$

and
$$z=(p,q,v,w,a,b)^{\intercal},\quad S=\frac{1}{2}V(p^{2}+q^{2},x,y)+\frac{\alpha}{2}(v^{2}+w^{2}+a^{2}+b^{2}).$$
According to \eqref{2DECL}, the ECL of Equation \eqref{HDNLS} reads
\begin{equation}\label{ECLNLS}
    \partial_tE+\partial_xF+\partial_yG=0,
\end{equation}
where
\begin{equation*}
\begin{aligned}
&E=\frac{1}{2}V(p^{2}+q^{2},x,y)+\frac{\alpha }{2}(pv_{x}+qw_{x}+pa_{y}+qb_{y}),\\
&F=\frac{\alpha }{2}(-pv_{t}-qw_{t}+vp_{t}+wq_{t}),
\quad G=\frac{\alpha }{2}(-pa_{t}-qb_{t}+ap_{t}+bq_{t}).\\
\end{aligned}
\end{equation*}
\eqref{HDNLS} also has the local charge conservation law:
\begin{equation*}
\partial_{t}C+\partial_{x}P+\partial_{y}Q=0,
\end{equation*}
where
\begin{equation*}
C=\frac{1}{2}(p^{2}+q^{2}),P=\alpha(-vq+wp),Q=\alpha(-aq+bp).
\end{equation*}

If $V$ is independent of the variables $x,y,$ then \eqref{HDNLS} is a
multi-symplectic PDE. According to \eqref{MSCL}, the MSCL is:
\begin{equation*}
\partial_{t}(dp\wedge dq)+\partial_{x}(-\alpha dp\wedge dv-\alpha dq\wedge dw)+\partial_{y}(-\alpha dp\wedge da-\alpha dq\wedge db)=0.
\end{equation*}

All of these conservation laws lead to corresponding global
invariants under the PBC. The full discretized scheme of
\eqref{HDNLS} in vector form derived from our discrete procedure \eqref{FDS} is:
\begin{equation}\label{2DNLSM}
\left\{\begin{aligned}
&p^{\tau}=p^{0}+\Delta t\int_{0}^{1}A_{\tau,\sigma}(-(D_{x}\otimes I_{M})\alpha w^{\sigma}-(I_{N}\otimes D_{y})\alpha b^{\sigma}-
V^{'}((p^{\sigma})^{\cdot2}+(q^{\sigma})^{\cdot2},x\otimes e_{M},e_{N}\otimes y)\cdot q^{\sigma})d\sigma,\\
&q^{\tau}=q^{0}+\Delta t\int_{0}^{1}A_{\tau,\sigma}((D_{x}\otimes I_{M})\alpha v^{\sigma}+(I_{N}\otimes D_{y})\alpha a^{\sigma}+
V^{'}((p^{\sigma})^{\cdot2}+(q^{\sigma})^{\cdot2},x\otimes e_{M},e_{N}\otimes y)\cdot p^{\sigma})d\sigma,\\
&p^{1}=p^{0}+\Delta t\int_{0}^{1}(-(D_{x}\otimes I_{M})\alpha w^{\sigma}-(I_{N}\otimes D_{y})\alpha b^{\sigma}-
V^{'}((p^{\sigma})^{\cdot2}+(q^{\sigma})^{\cdot2},x\otimes e_{M},e_{N}\otimes y)\cdot q^{\sigma})d\sigma,\\
&q^{1}=q^{0}+\Delta t\int_{0}^{1}(D_{x}\otimes I_{M})\alpha v^{\sigma}+(I_{N}\otimes D_{y})\alpha a^{\sigma}+
V^{'}((p^{\sigma})^{\cdot2}+(q^{\sigma})^{\cdot2},x\otimes e_{M},e_{N}\otimes y)\cdot p^{\sigma})d\sigma,\\
&\delta_{x}p^{\sigma}=(D_{x}\otimes I_{M})p^{\sigma}=v^{\sigma},\delta_{x}q^{\sigma}=(D_{x}\otimes I_{M})q^{\sigma}=w^{\sigma},\\
&\delta_{y}p^{\sigma}=(I_{N}\otimes D_{y})p^{\sigma}=a^{\sigma},\delta_{y}q^{\sigma}=(I_{N}\otimes D_{y})q^{\sigma}=b^{\sigma},\\
\end{aligned}\right.
\end{equation}
where the entries $p_{jl},q_{jl}$ of vectors $p,q$ are arranged according to lexicographical order :
$$(j,l)\prec (k,l), \text{  when $j\prec k$},\quad (j,l)\prec (j,m), \text{  when $l\prec m$},$$
$x=(x_{0},x_{1},\ldots,x_{N-1})^{\intercal},y=(y_{0},y_{1},\ldots,y_{M-1})^{\intercal},
\ I_{N},I_{M},e_{N},e_{M}$ are $N$th and $M$th order identity
matrices, $N$ length and $M$ length identity vectors, respectively.
If the potential $V$ is a polynomial in the first variable, then the scheme
\eqref{2DNLSM} can be implemented in a similar way to \eqref{DCNLS}.

By Theorem \ref{HDECL}, \eqref{2DNLSM} preserves the discrete ECLs:
\begin{equation}\label{newECL}
\frac{E_{jl}^{1}-E_{jl}^{0}}{\Delta t}+\sum_{k=0}^{N-1}(D_{x})_{jk}\bar{F}_{jk,l}+\sum_{m=0}^{M-1}(D_{y})_{lm}\bar{G}_{j,lm}=0,
\end{equation}
for $j=0, 1, \ldots, N-1$, $l=0, 1, \ldots, M-1$, where
\begin{equation*}
\begin{aligned}
&E_{jl}^{c}=\frac{1}{2}V((p_{jl}^{c})^{2}+(q_{jl}^{c})^{2},x_{j},y_{l})+\frac{\alpha}{2}(p_{jl}^{c}\delta_{x}v_{jl}^{c}+
q_{jl}^{c}\delta_{x}w_{jl}^{c}+p_{jl}^{c}\delta_{y}a_{jl}^{c}+q_{jl}^{c}\delta_{y}b_{jl}^{c}),c=0,1,\\
&\bar{F}_{jk,l}=\frac{\alpha }{2}\sum_{i=1}^{s}b_{i}(-\langle p_{jl}\rangle_{i}\langle\delta_{t}v_{kl}\rangle_{i}
-\langle q_{jl}\rangle_{i}\langle\delta_{t}w_{kl}\rangle_{i}+\langle v_{jl}\rangle_{i}\langle\delta_{t}p_{kl}\rangle_{i}
+\langle w_{jl}\rangle_{i}\langle\delta_{t}q_{kl}\rangle_{i}\\
&+\frac{\alpha }{2}\sum_{i=1}^{s}b_{i}(-\langle p_{kl}\rangle_{i}\langle\delta_{t}v_{jl}\rangle_{i}
-\langle q_{kl}\rangle_{i}\langle\delta_{t}w_{jl}\rangle_{i}+\langle v_{kl}\rangle_{i}\langle\delta_{t}p_{jl}\rangle_{i}
+\langle w_{kl}\rangle_{i}\langle\delta_{t}q_{jl}\rangle_{i}),\\
&\bar{G}_{j,lm}=\frac{\alpha }{2}\sum_{i=1}^{s}b_{i}(-\langle p_{jl}\rangle_{i}\langle\delta_{t}a_{jm}\rangle_{i}
-\langle q_{jl}\rangle_{i}\langle\delta_{t}b_{jm}\rangle_{i}+\langle a_{jl}\rangle_{i}\langle\delta_{t}p_{jm}\rangle_{i}
+\langle b_{jl}\rangle_{i}\langle\delta_{t}q_{jm}\rangle_{i})\\
&+\frac{\alpha }{2}\sum_{i=1}^{s}b_{i}(-\langle p_{jm}\rangle_{i}\langle\delta_{t}a_{jl}\rangle_{i}
-\langle q_{jm}\rangle_{i}\langle\delta_{t}b_{jl}\rangle_{i}+\langle a_{jm}\rangle_{i}\langle\delta_{t}p_{jl}\rangle_{i}
+\langle b_{jm}\rangle_{i}\langle\delta_{t}q_{jl}\rangle_{i}).
\end{aligned}
\end{equation*}
However, the expressions of $E_{jl}^{c}$, $\bar{F}_{jk,l}$, $\bar{G}_{j,lm}$ are lengthy and difficult to be calculated.
We thus rewrite them as:
\begin{equation}\label{II}
\begin{aligned}
&E_{jl}^{c}=\frac{1}{2}V((p_{jl}^{c})^{2}+(q_{jl}^{c})^{2},x_{j},y_{l})-\frac{\alpha}{2}((v_{jl}^{c})^{2}+(w_{jl}^{c})^{2}
+(a_{jl}^{c})^{2}+(b_{jl}^{c})^{2})+\tilde{E}_{jl}^{c},c=0,1,\\
&\bar{F}_{jk,l}=\alpha\sum_{i=1}^{s}b_{i}(\langle v_{jl}\rangle_{i}\langle\delta_{t}p_{kl}\rangle_{i}
+\langle w_{jl}\rangle_{i}\langle\delta_{t}q_{kl}\rangle_{i}+\langle v_{kl}\rangle_{i}\langle\delta_{t}p_{jl}\rangle_{i}
+\langle w_{kl}\rangle_{i}\langle\delta_{t}q_{jl}\rangle_{i})+\tilde{F}_{jk,l},\\
&\bar{G}_{j,lm}=\alpha\sum_{i=1}^{s}b_{i}(\langle a_{jl}\rangle_{i}\langle\delta_{t}p_{jm}\rangle_{i}
+\langle b_{jl}\rangle_{i}\langle\delta_{t}q_{jm}\rangle_{i}+\langle a_{jm}\rangle_{i}\langle\delta_{t}p_{jl}\rangle_{i}
+\langle b_{jm}\rangle_{i}\langle\delta_{t}q_{jl}\rangle_{i})+\tilde{G}_{j,lm},\\
\end{aligned}
\end{equation}
where $\tilde{E}_{jl}^{c},\ \tilde{F}_{jk,l}\ \tilde{G}_{j,lm}$ are
the corresponding residuals. Taking derivatives with respect to $\tau$ on both sides of
\begin{equation*}
    \delta_xp_{jl}^\tau=v_{jl}^\tau=v_{jl}^0+\Delta t\int_0^1A_{\tau,\sigma}\delta_tv_{jl}^\sigma d\sigma
\end{equation*}
and setting $\tau=c_1,\ldots,c_s$, we have
\begin{equation*}
    \langle\delta_x\delta_tp_{jl}\rangle_i=\langle\delta_tv_{jl}\rangle_i,
\end{equation*}
for $i=1,\ldots,s.$ By using this law of commutation and following the standard proof procedure of Theorem \ref{DECL}, the term involving $v_{jl}$ can be eliminated from $\tilde{E}_{jl}^{c},\ \tilde{F}_{jk,l}\ \tilde{G}_{j,lm}$. The terms involving $w_{jl}$, $a_{jl}$, $b_{jl}$ can be dealt with in the same way. 

Therefore we arrive at 
\begin{equation}\label{new}
    \frac{\tilde{E}_{jl}^{1}-\tilde{E}_{jl}^{0}}{\Delta t}+\sum_{k=0}^{N-1}(D_{x})_{jk}\tilde{F}_{jk,l}+\sum_{m=0}^{M-1}(D_{y})_{lm}\tilde{G}_{j,lm}=0.
\end{equation}
Subtracting \eqref{new} from \eqref{newECL}, we obtain the new discree ECLs of \eqref{2DNLSM}
\begin{equation}\label{2DNLSECL}
\frac{E_{jl}^{1}-E_{jl}^{0}}{\Delta t}+\sum_{k=0}^{N-1}(D_{x})_{jk}\bar{F}_{jk,l}+\sum_{m=0}^{M-1}(D_{y})_{lm}\bar{G}_{j,lm}=0,
\end{equation}
for $j=0,1,\ldots,N-1,\quad l=0,1,\ldots,M-1,$ where
\begin{equation*}
\begin{aligned}
&E_{jl}=\frac{1}{2}V(p_{jl}^{2}+q_{jl}^{2},x_{j},y_{l})-\frac{\alpha}{2}(v_{jl}^{2}+w_{jl}^{2}+a_{jl}^{2}+b_{jl}^{2}),\\
&\bar{F}_{jk,l}=\alpha\sum_{i=1}^{s}b_{i}(\langle v_{jl}\rangle_{i}\langle\delta_{t}p_{kl}\rangle_{i}
+\langle w_{jl}\rangle_{i}\langle\delta_{t}q_{kl}\rangle_{i}+\langle v_{kl}\rangle_{i}\langle\delta_{t}p_{jl}\rangle_{i}
+\langle w_{kl}\rangle_{i}\langle\delta_{t}q_{jl}\rangle_{i}),\\
&\bar{G}_{j,lm}=\alpha\sum_{i=1}^{s}b_{i}(\langle a_{jl}\rangle_{i}\langle\delta_{t}p_{jm}\rangle_{i}
+\langle b_{jl}\rangle_{i}\langle\delta_{t}q_{jm}\rangle_{i}+\langle a_{jm}\rangle_{i}\langle\delta_{t}p_{jl}\rangle_{i}
+\langle b_{jm}\rangle_{i}\langle\delta_{t}q_{jl}\rangle_{i}).\\
\end{aligned}
\end{equation*}
\eqref{2DNLSECL} can be thought of as a  discrete
version of
$$\partial_{t}(\frac{1}{2}V(p^{2}+q^{2},x,y))+\partial_{x}(vp_{t}+wq_{t})+\partial_{y}(ap_{t}+bq_{t})=0,$$
which is a more common ECL equation \eqref{HDNLS} than \eqref{ECLNLS}.

\eqref{2DNLSECL} involve less discrete derivatives than
\eqref{newECL}, thus  can be easily calculated.

\section{Numerical experiments for coupled nonlinear Schr\"{o}dingers
equations}\label{Numer_Experi} If we choose the two-point Gauss-Legendre
quadrature formula:
\begin{equation*}
\begin{aligned}
&b_{1}=\frac{1}{2},b_{2}=\frac{1}{2},\\
&c_{1}=\frac{1}{2}-\frac{\sqrt{3}}{6},c_{2}=\frac{1}{2}+\frac{\sqrt{3}}{6}\\
\end{aligned}
\end{equation*}
for the CRK method, then
\begin{equation*}
A_{\tau,\sigma}=\tau((4-3\tau)-6(1-\tau)\sigma).
\end{equation*}
This CRK method is of order four by \eqref{Order}. In this section, we use it for the temporal discretization
while the spatial direction is discretized by the
pseudospectral method. The corresponding local energy-preserving method for the CNLS is denoted by ET4.

Throughout the experiments in this section we always take the periodic
boundary condition
$u(x_{0},t)=u(x_{0}+L,t),v(x_{0},t)=v(x_{0}+L,t)$ and set the initial
time $t_{0}=0$. Besides the discrete global energy which has been
{mentioned in \eqref{GEL}, we define} these discrete global
quantities as follows:

1. The discrete global charges of $u$ and $v$ at time $n\Delta t$:
\begin{equation*}
\left\{
\begin{aligned}
&CH_{U}^{n}=\Delta x\sum_{j=0}^{N-1}((q_{1,j}^{n})^2+(q_{2,j}^{n})^2),\\
&CH_{V}^{n}=\Delta x\sum_{j=0}^{N-1}((q_{3,j}^{n})^2+(q_{4,j}^{n})^2).\\
\end{aligned}\right.
\end{equation*}

2. The discrete global momentum at time $n\Delta t$:
\begin{equation*}
I^{n}=\Delta x\sum_{j=0}^{N-1}(q_{2,j}^{n}p_{1,j}^{n}-q_{1,j}^{n}p_{2,j}^{n}+q_{4,j}^{n}p_{3,j}^{n}-q_{3,j}^{n}p_{4,j}^{n}).
\end{equation*}

The (relative) global energy error (GEE$^{n}$), global momentum
error (GIE$^{n}$), global charge errors of $u$ (GCE$_{U}^{n}$) and $v$
(GCE$_{V}^{n}$) at time $n\Delta t$ will be calculated by the
following formulas :

\begin{equation*}
\begin{aligned}
&GGE^{n}=\frac{E^{n}-E^{0}}{|E^{0}|},
GIE^{n}=\frac{I^{n}-I^{0}}{|I^{0}|},\\
&GCE_{U}^{n}=\frac{CH_{U}^{n}-CH_{U}^{0}}{|CH_{U}^{0}|},
GCE_{V}^{n}=\frac{CH_{V}^{n}-CH_{V}^{0}}{|CH_{V}^{0}|},\\
\end{aligned}
\end{equation*}
respectively.

\begin{myexp}\label{5.1}
We first consider to set the constants
$\alpha,\beta=0.$ Then the CNLS decompose into two independent
NLSs:
\begin{equation}
\left\{
\begin{aligned}
&iu_{t}+\frac{1}{2}u_{xx}+|u|^{2}u=0,\\
&iv_{t}+\frac{1}{2}v_{xx}+|v|^{2}v=0.\\
\end{aligned}\right.
\end{equation}
Given the initial condition:
\begin{equation*}
\left\{
\begin{aligned}
&u(x,0)=sech(x),\\
&v(x,0)=sech(x)exp(i\frac{x}{\sqrt{10}}),\\
\end{aligned}\right.
\end{equation*}
the analytic expressions of $u$ and $v$ are :
\begin{equation}\label{ES}
\left\{
\begin{aligned}
&u(x,t)=sech(x)exp(i\frac{t}{2}),\\
&v(x,t)=sech(x-\frac{t}{\sqrt{10}})exp(i(\frac{x}{\sqrt{10}}+\frac{9}{20}t)).\\
\end{aligned}\right.
\end{equation}

In this experiment, we compute the difference between the numerical solution and the exact solution
of $u$. Since $u$ decays exponentially away from the point $(0,t)$,
we can take the boundary condition $u(-30,0)=u(30,0),v(-30,0)=v(30,0)$
with little loss of accuracy on $u$. We also
compare our local energy-preserving method ET4 with a classical
multi-symplectic scheme (MST4) which is obtained by concatenating the two-point
Gauss-Legendre symplectic Runge--Kutta method in time and the
pseudospectral method in space.  Note that ET4 and MST4 are of the same order. Let $N=300,\Delta t=0.4,\ 0.8$ and
set  $\varepsilon=10^{-14}$ as the error tolerance for iteration
solutions. The numerical results over the time interval $[0,1200]$, which is about $100$ multiples of the period of $u$, are plotted in
Figs. \ref{EI},\ldots,\ref{shape}.

\begin{figure}[ptb]
\centering
\begin{tabular}[c]{cccc}%
  % Requires \usepackage{graphicx}
  \subfigure[Global energy and momentum errors]{\includegraphics[width=8cm,height=6cm]{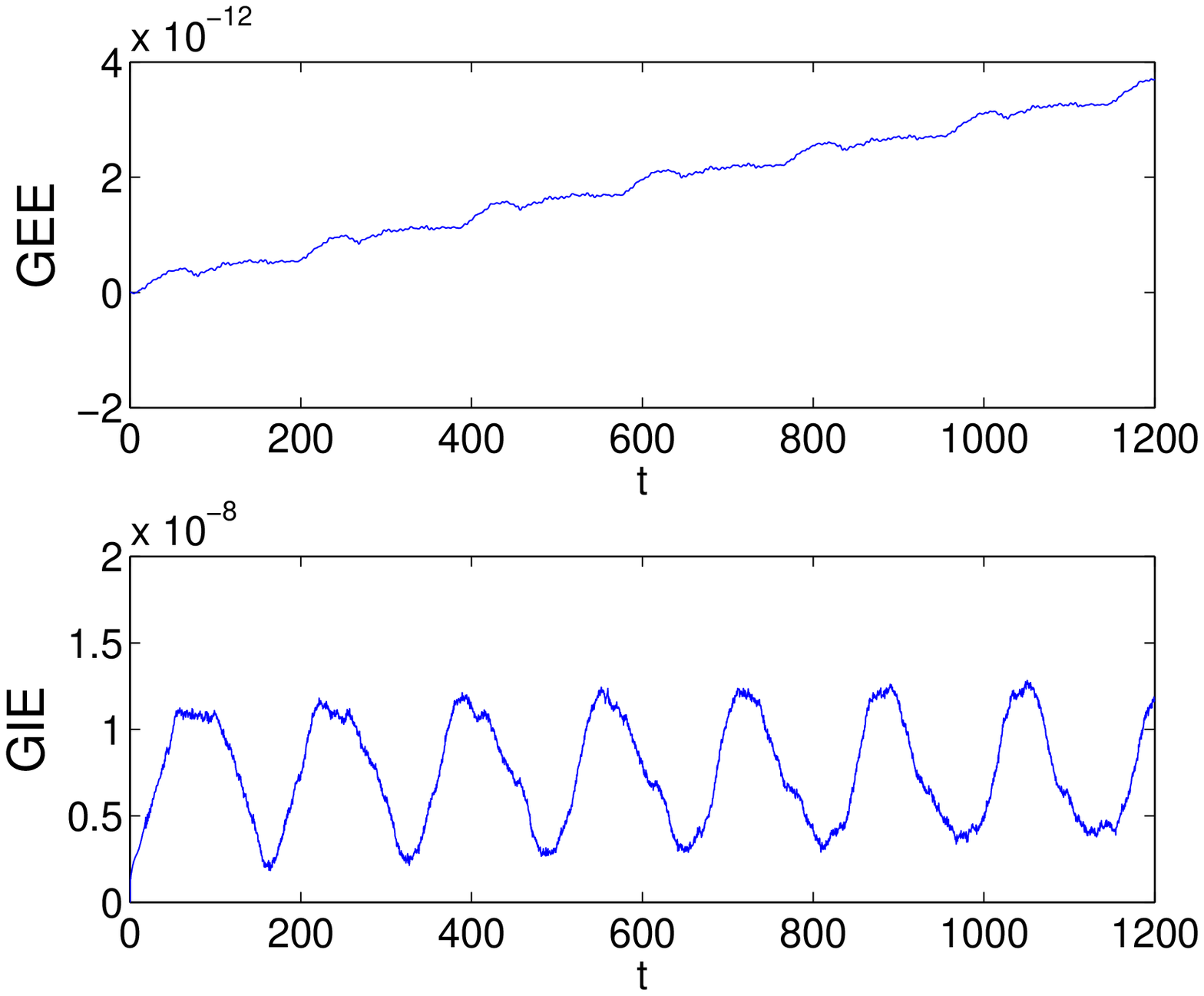}}
  \subfigure[Global charge errors]{\includegraphics[width=8cm,height=6cm]{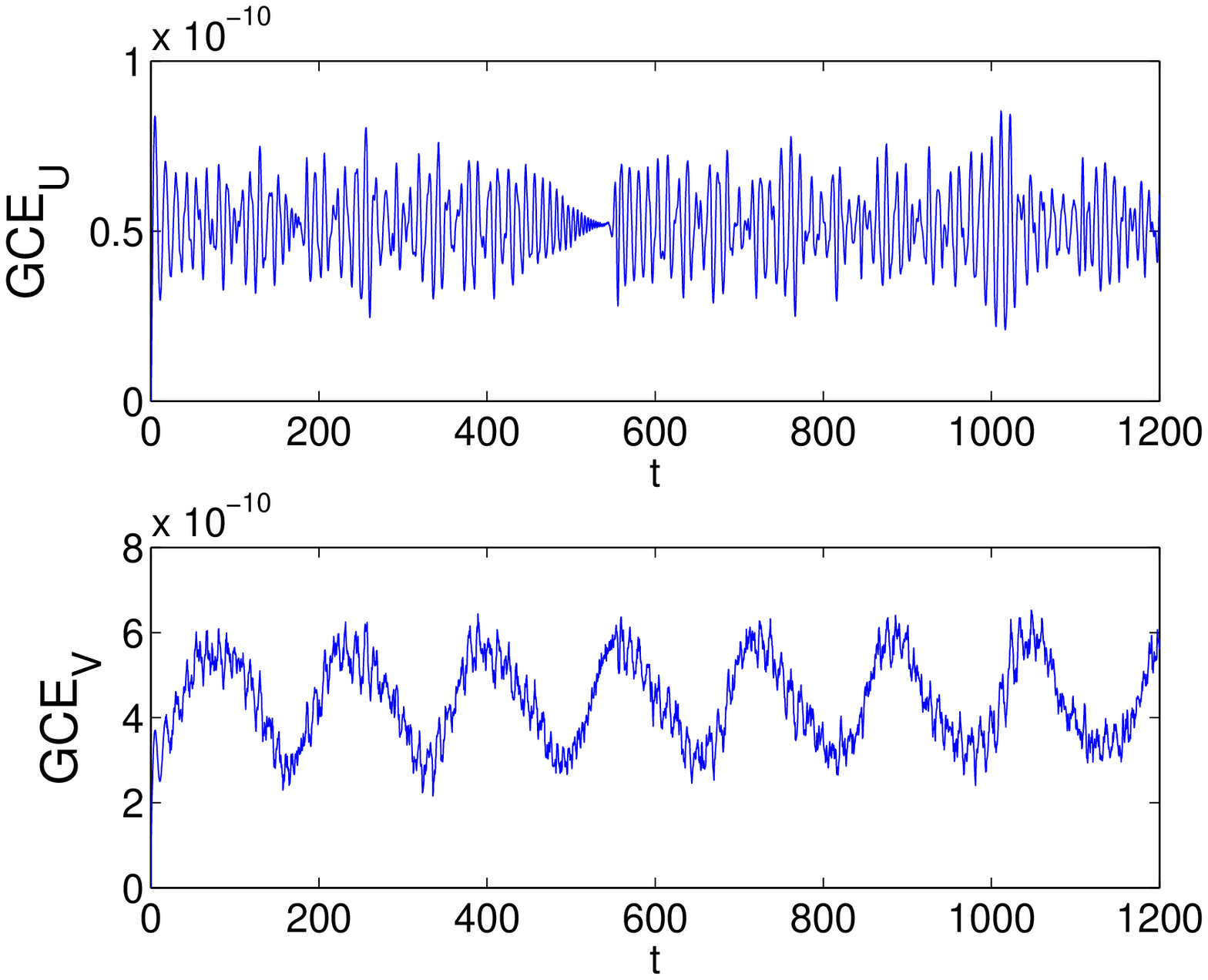}}
\end{tabular}
\caption{Errors obtained by ET4, $\Delta t=0.4$.}
\label{EI}
\end{figure}

\begin{figure}[ptb]
\centering
\begin{tabular}[c]{cccc}%
  % Requires \usepackage{graphicx}
  \subfigure[Global energy and momentum errors]{\includegraphics[width=8cm,height=6cm]{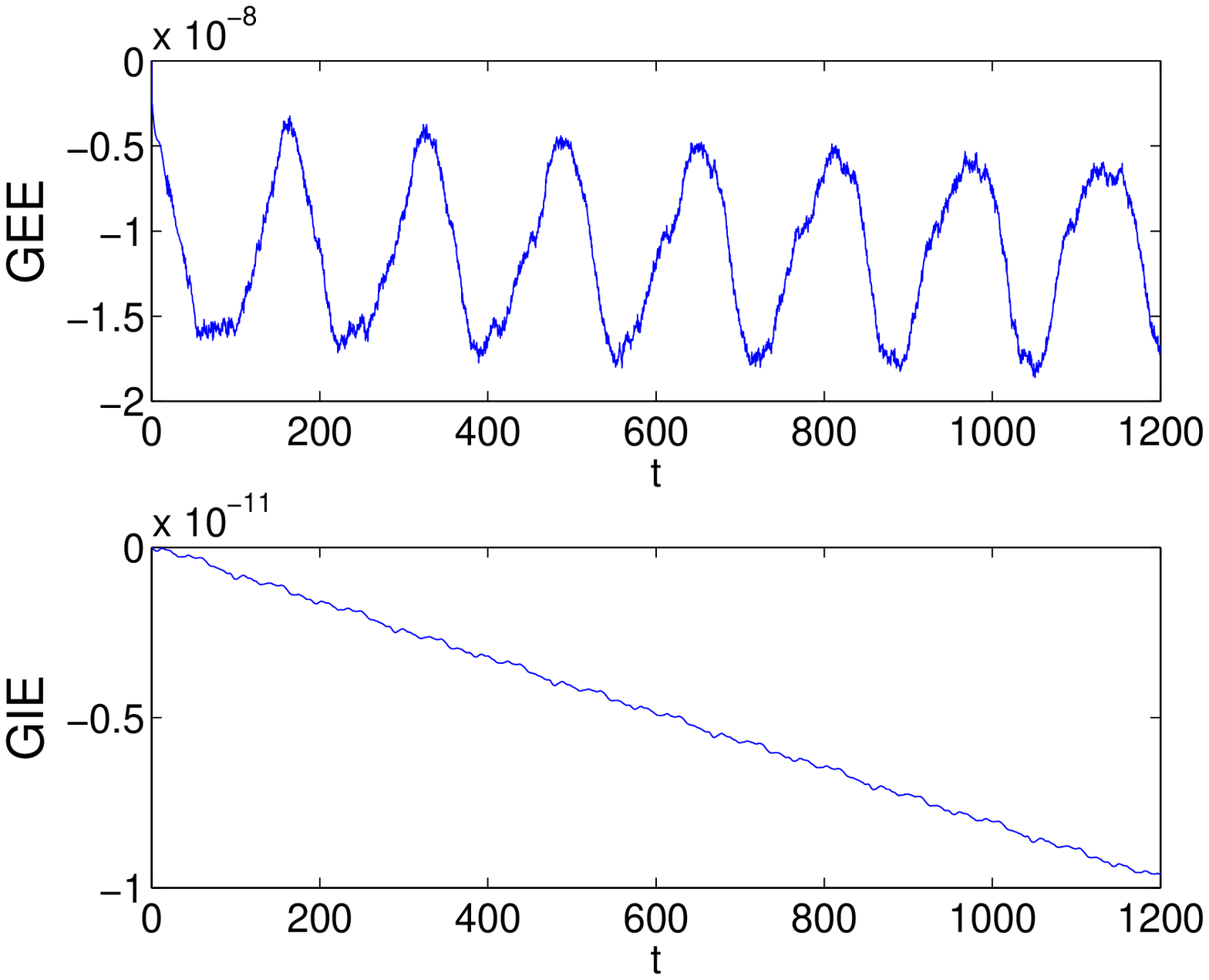}}
  \subfigure[Global charge errors]{\includegraphics[width=8cm,height=6cm]{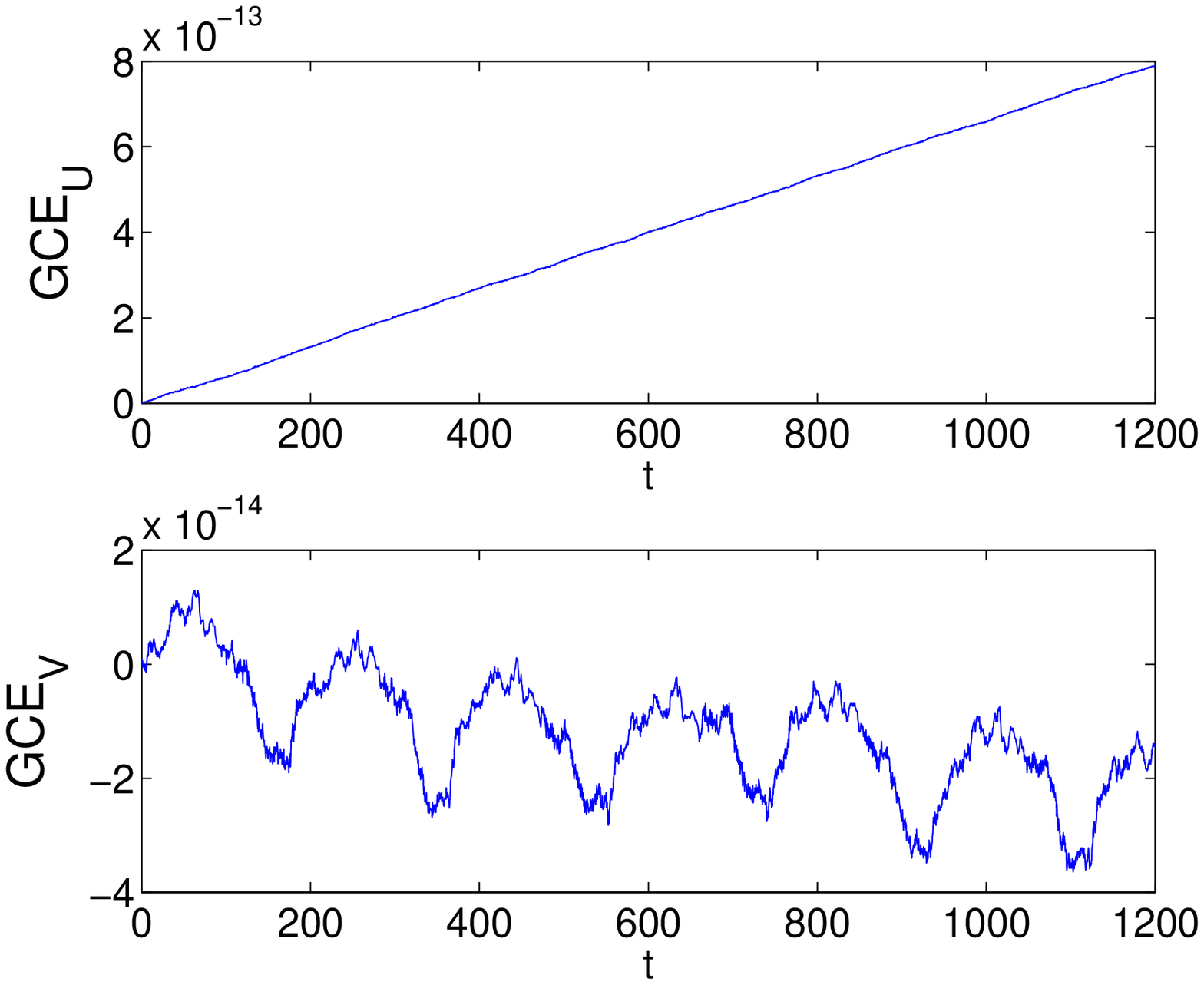}}
\end{tabular}
\caption{Errors obtained by MST4, $\Delta t$=0.4.}
\label{EI2}
\end{figure}

\begin{figure}[ptb]
\centering
\begin{tabular}[c]{cccc}%
  % Requires \usepackage{graphicx}
  \subfigure[Global energy and momentum errors]{\includegraphics[width=8cm,height=6cm]{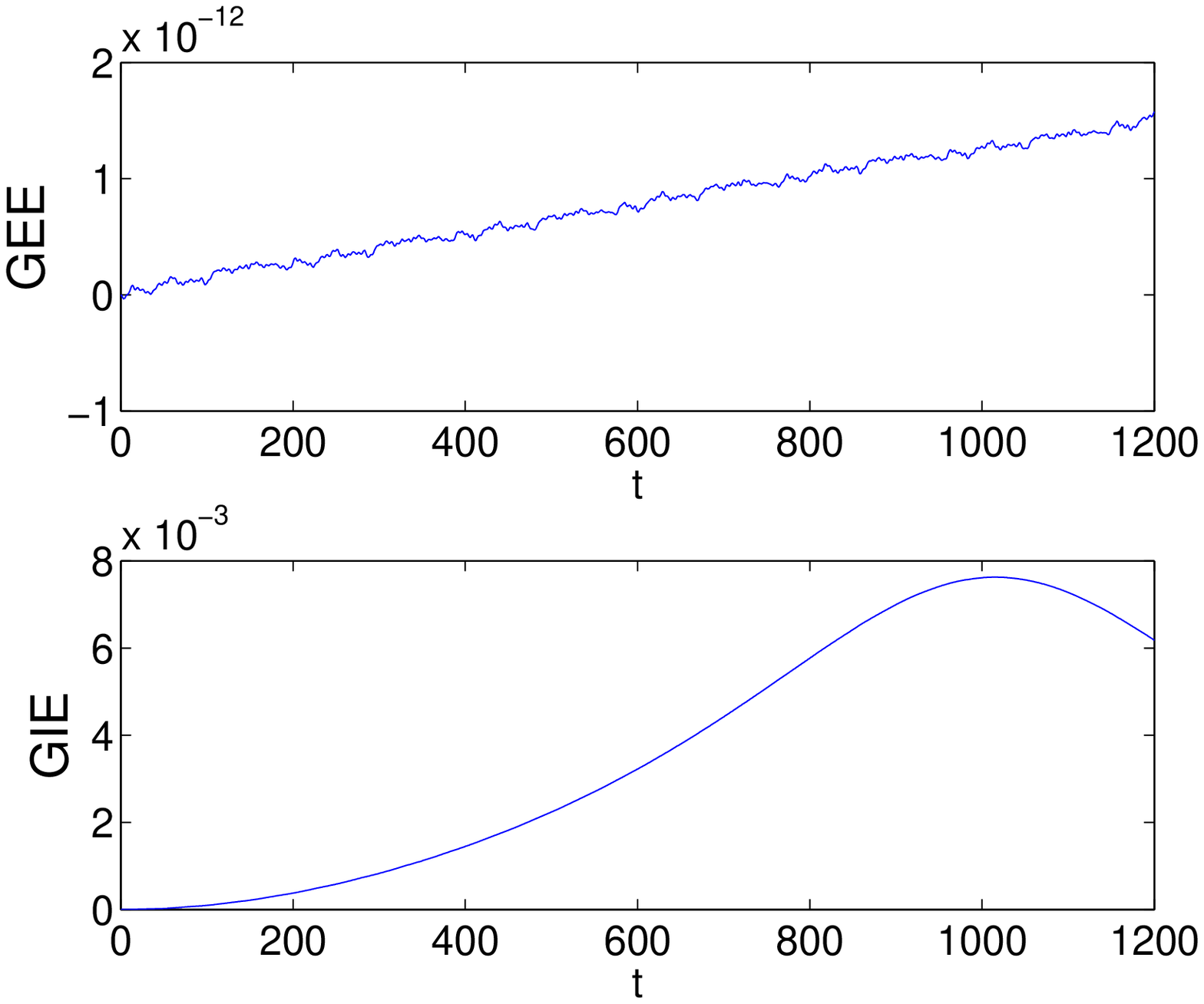}}
  \subfigure[Global charge errors]{\includegraphics[width=8cm,height=6cm]{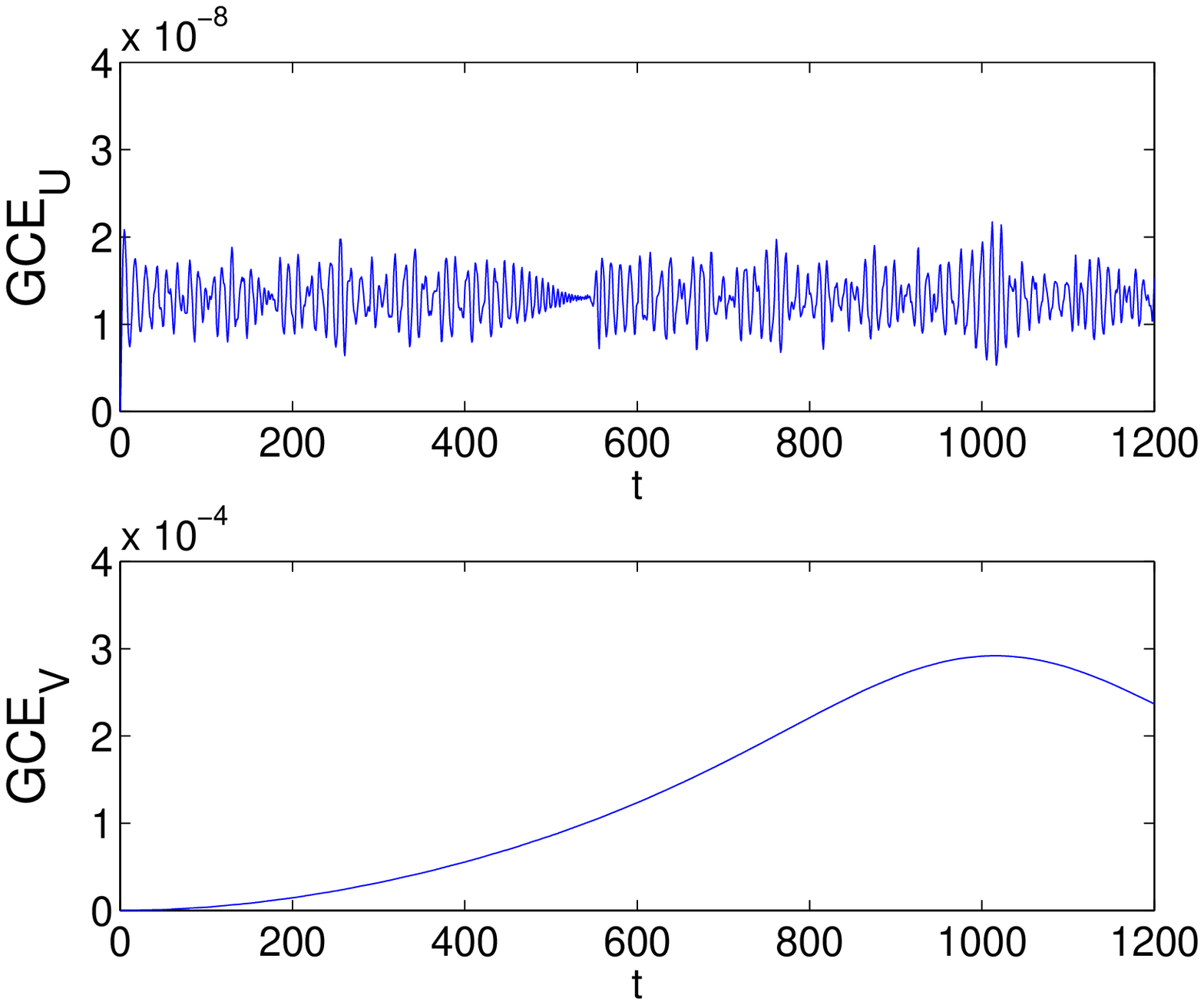}}
\end{tabular}
\caption{Errors obtained by ET4, $\Delta t$=0.8.}
\label{EI3}
\end{figure}

\begin{figure}[ptb]
\centering
\begin{tabular}[c]{cccc}%
  % Requires \usepackage{graphicx}
  \subfigure[Global energy and momentum errors]{\includegraphics[width=8cm,height=6cm]{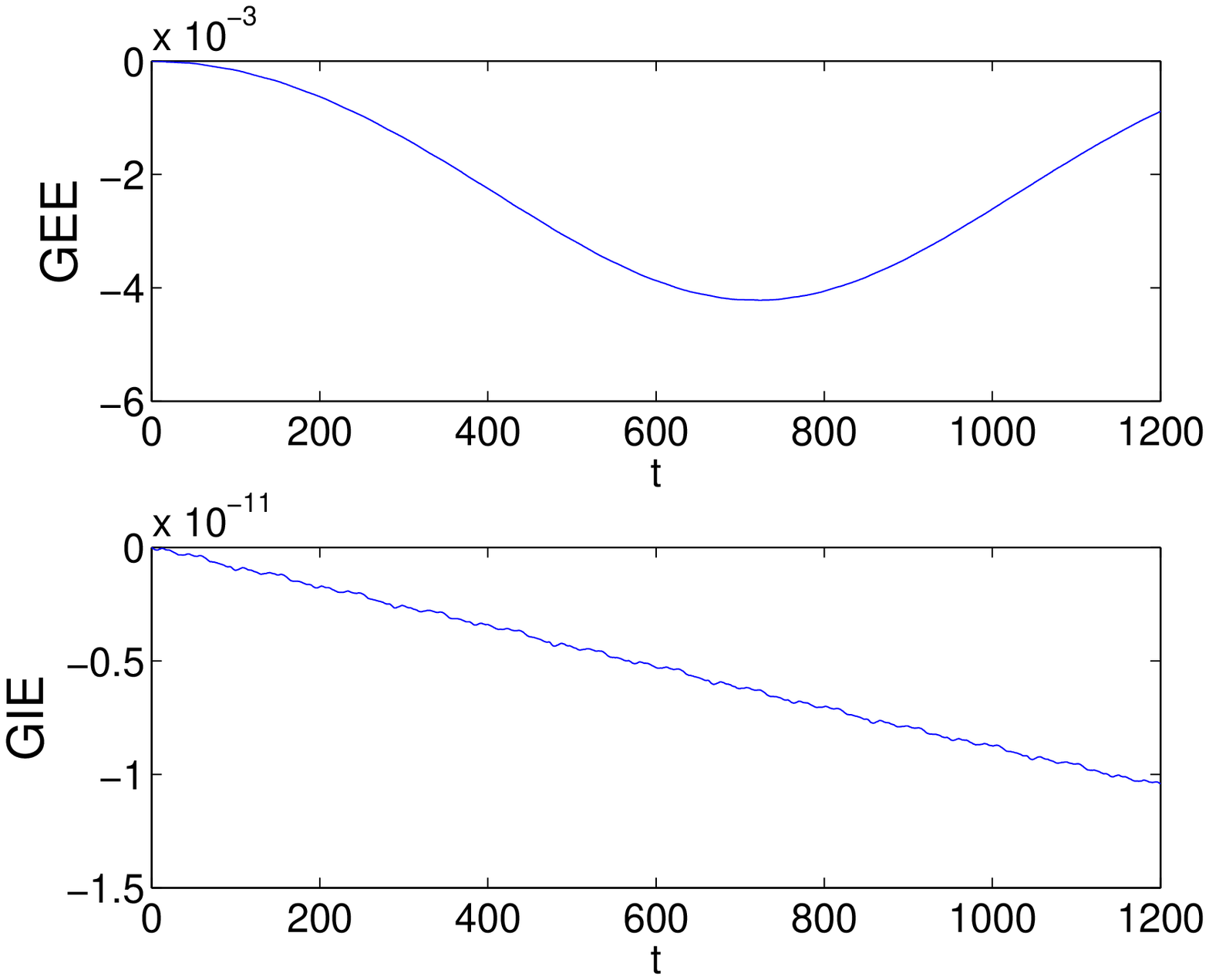}}
  \subfigure[Global charge errors]{\includegraphics[width=8cm,height=6cm]{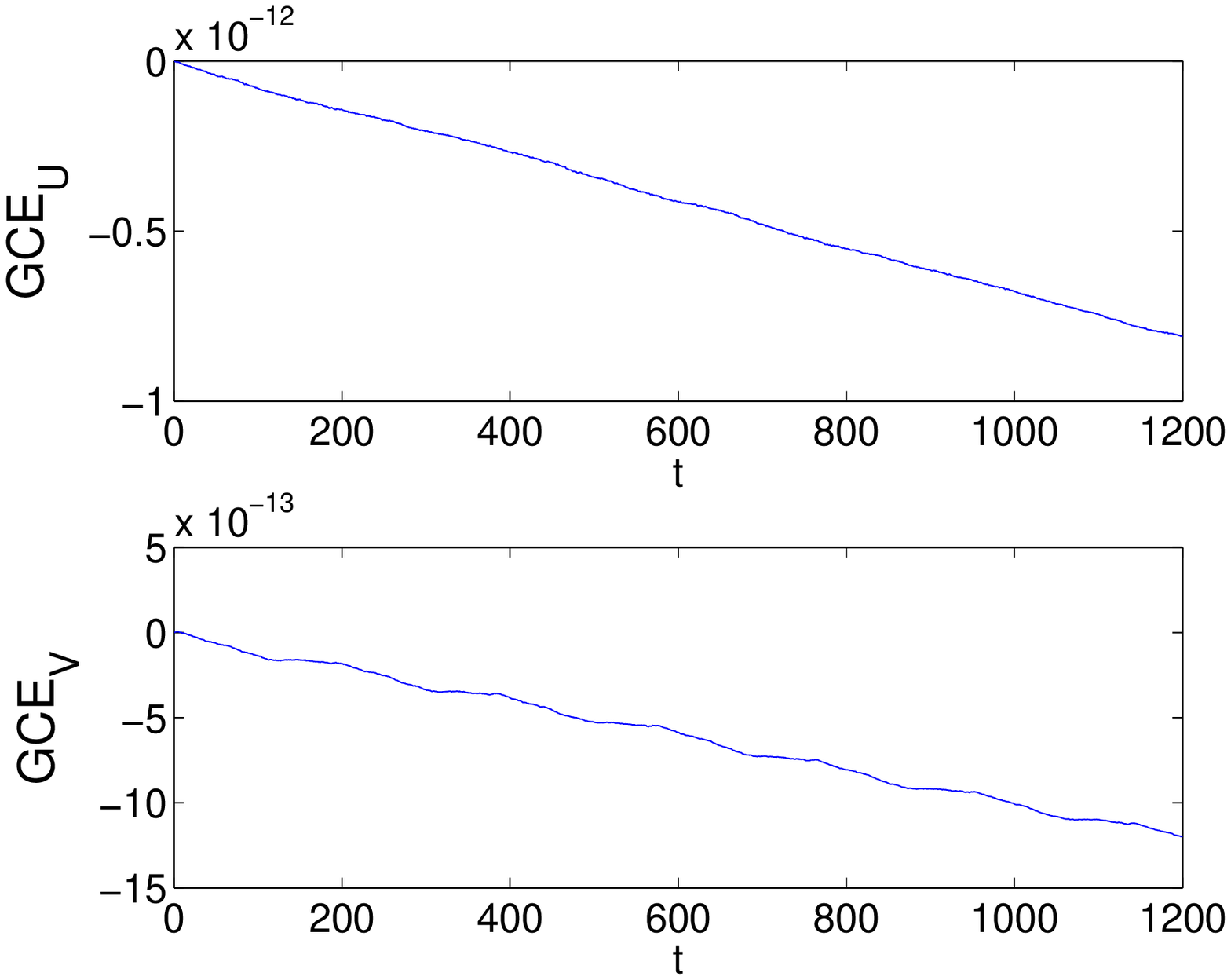}}
\end{tabular}
\caption{Errors obtained by MST4, $\Delta t$=0.8.}
\label{EI4}
\end{figure}

\newpage
\begin{figure}[ptb]
\centering
\begin{tabular}[c]{cccc}%
  % Requires \usepackage{graphicx}
  \subfigure[Maximum global errors, $\Delta t=0.4.$]{\includegraphics[width=8cm,height=4cm]{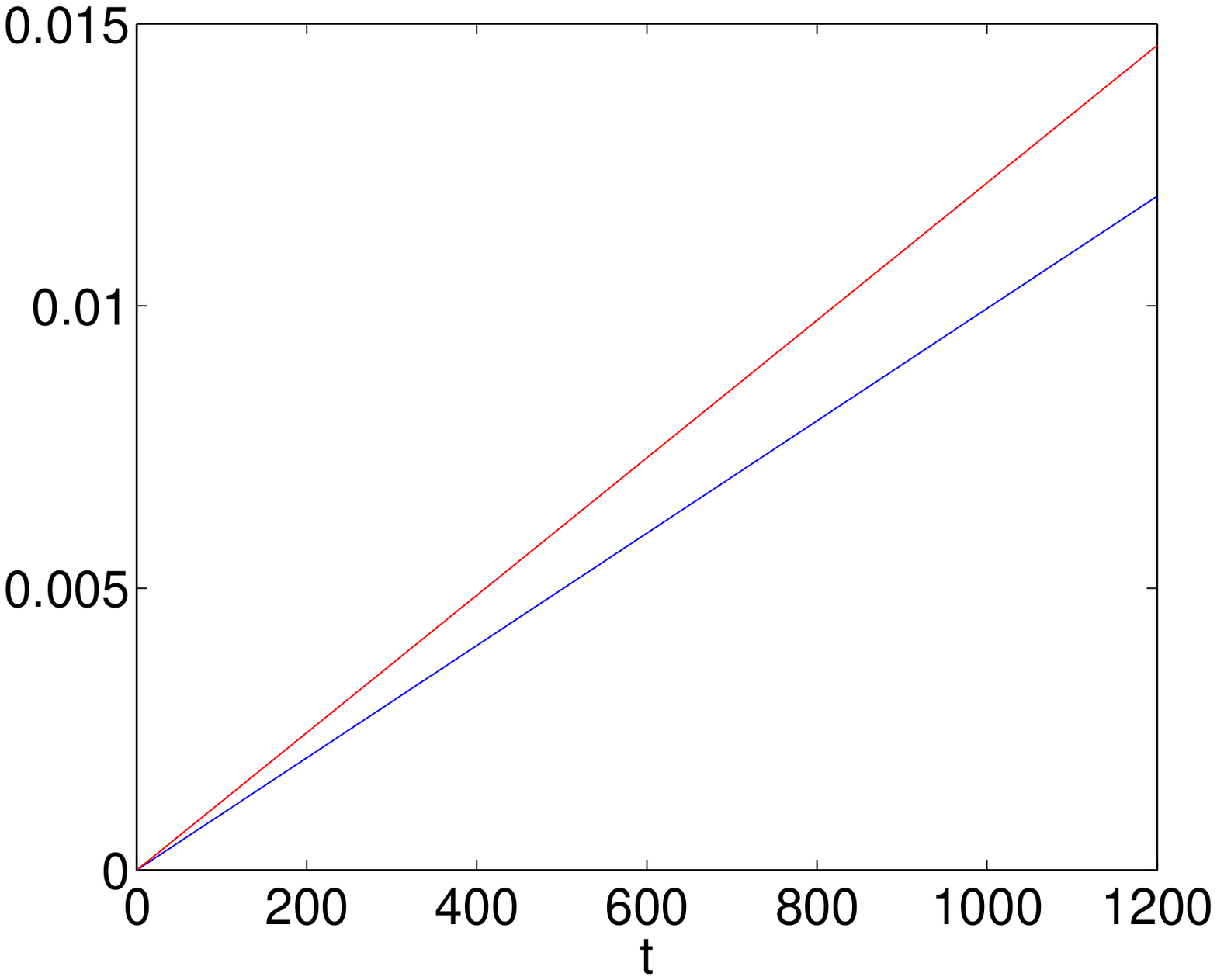}}
  \subfigure[Maximum global errors, $\Delta t=0.8.$]{\includegraphics[width=8cm,height=4cm]{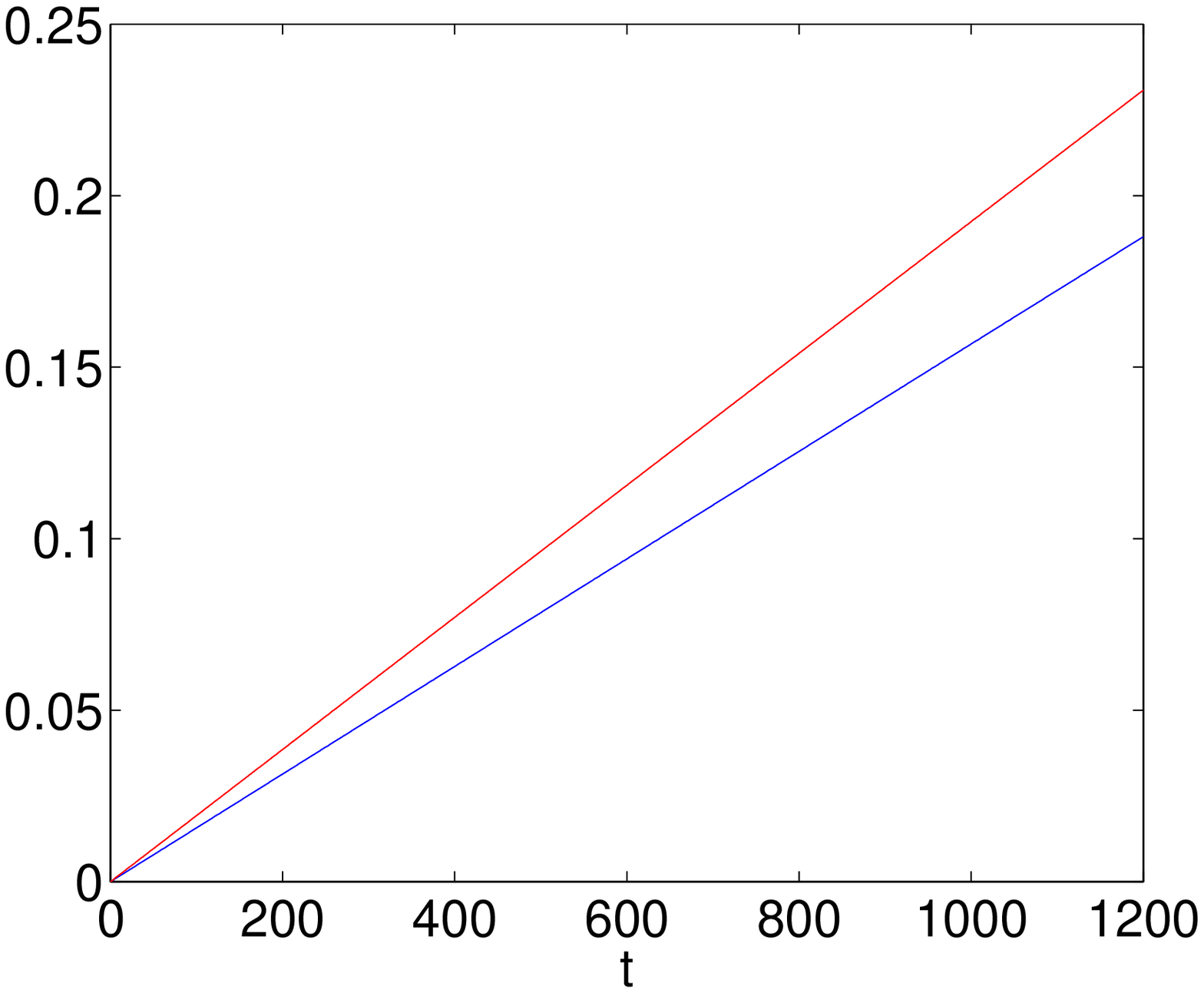}}
\end{tabular}
\caption{Maximum global errors of ET4 (left) and MST4 (right) . The
blue and red curves are the errors of ET4 and MST4 respectively.}
\label{GE}
\end{figure}

\begin{figure}[ptb]
\centering
\begin{tabular}[c]{cccc}%
  % Requires \usepackage{graphicx}
  \subfigure{\includegraphics[width=8cm,height=6cm]{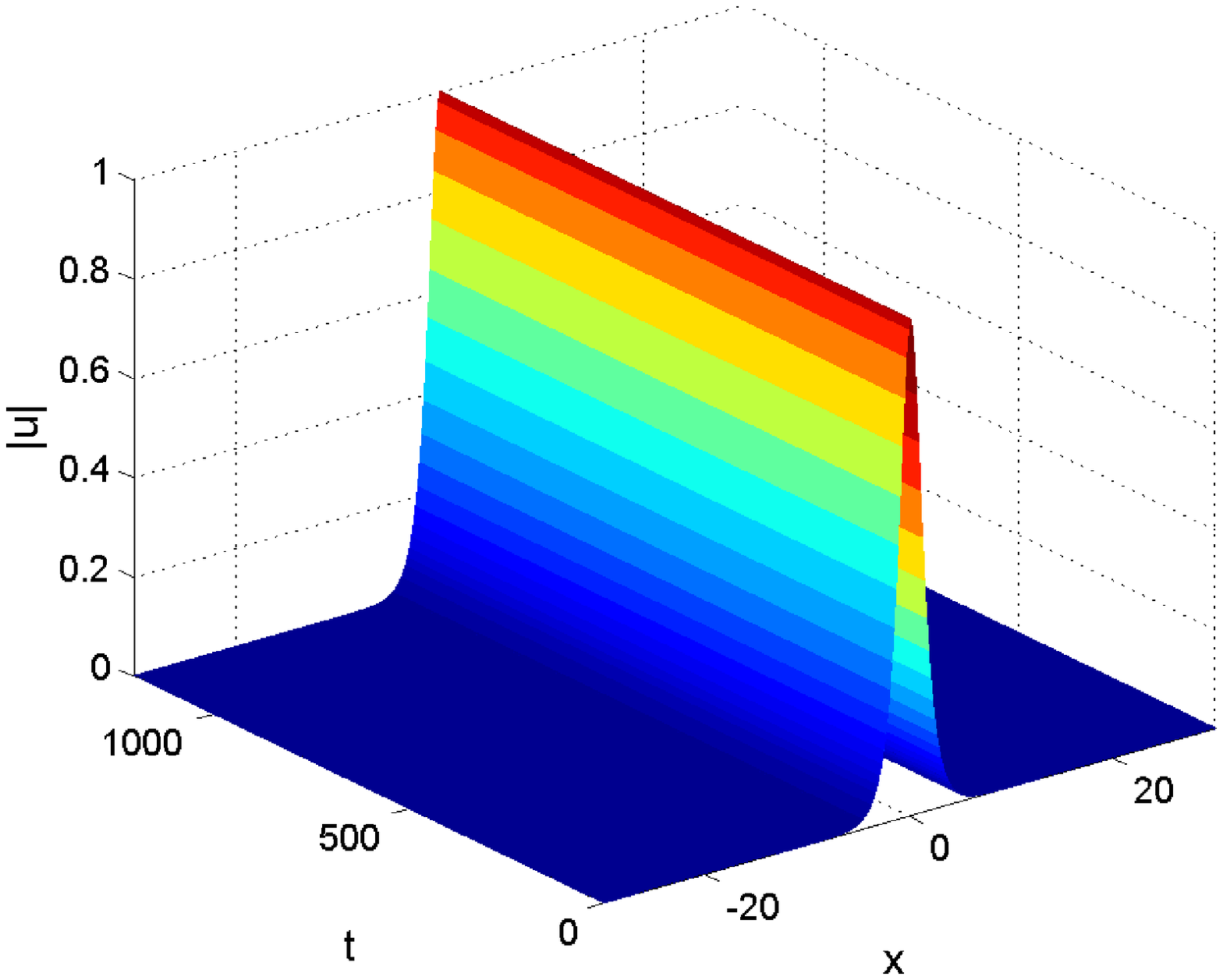}}
  \subfigure{\includegraphics[width=8cm,height=6cm]{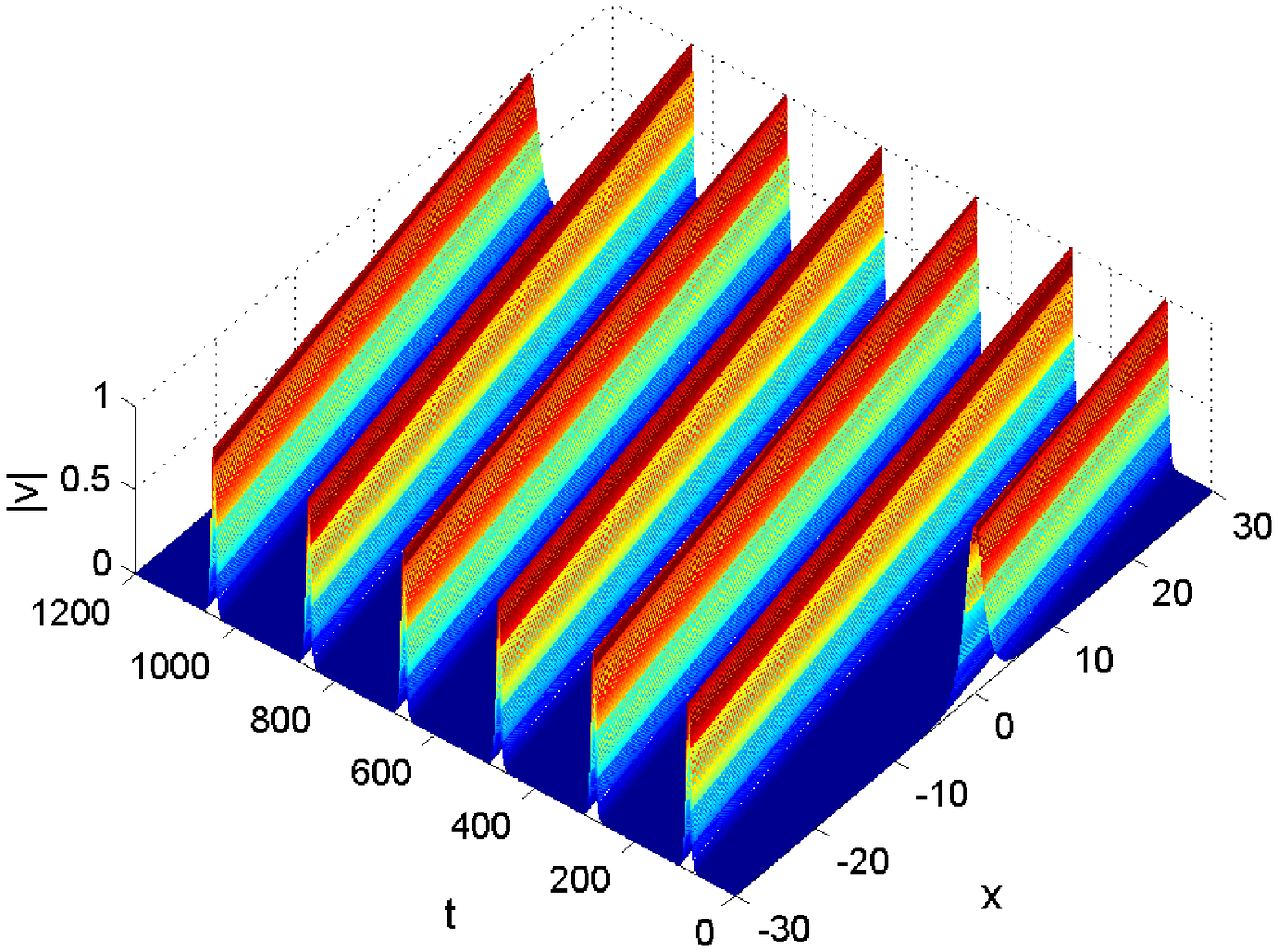}}
\end{tabular}
\caption{Numerical shapes of $u$ (left) and $v$ (right), obtained by ET4.}
\label{shape}
\end{figure}

Figs. \ref{EI}, \ref{EI3} illustrate that ET4 conserves the discrete global energy exactly
(regardless of round-off errors). Although ET4 cannot preserve
discrete global charges, its global charge errors show
reasonable oscillation in magnitude $10^{-10}$ ($\Delta t=0.4$) and
$10^{-4}$ ($\Delta t=0.8$), respectively. We attribute this behaviour to the conjugate-symplecticity
of the CRK method.

On the contrary, Figs. \ref{EI2}, \ref{EI4} show that MST4 conserves
global charges exactly (regardless of round-off errors) while its
global energy errors oscillates in magnitude $10^{-8}$ ($\Delta t=0.4$) and
$10^{-3}$ ($\Delta t=0.8$). This is a
character of symplectic integrators.

According to Figs. \ref{EI},\ldots,\ref{EI4}, MST4 preserves the discrete global momentum better than ET4 in this experiment.

It can be observed from  \eqref{ES} that the amplitudes of
$u$ and $v$ are both $1$. Fig. \ref{GE} shows that ET4 and MST4
both have excellent long-term behaviours. The relative maximum
global errors do not exceed 1.5\% ($\Delta t=0.4$) and 25\% ($\Delta t=0.8$) over the time interval [0,1200] .

Here we  point out that ET4 and ST4 have the same
iteration cost with the same $\Delta t$ and
$\varepsilon$. In the case $\Delta t=0.4$, both of them
need $19$ iterations per step. This phenomenon also
occurs in the following experiments.
\end{myexp}

\begin{myexp}
We now start to simulate the collision of
double solitons with the initial condition:
\begin{equation*}
\left\{
\begin{aligned}
&u(x,0)=\sum_{j=1}^{2}\sqrt{\frac{2a_{j}}{1+\beta}}sech(\sqrt{2a_{j}}(x-x_{j}))exp(i(\gamma_{j}-\alpha)(x-x_{j}),\\
&v(x,0)=\sum_{j=1}^{2}\sqrt{\frac{2a_{j}}{1+\beta}}sech(\sqrt{2a_{j}}(x-x_{j}))exp(i(\gamma_{j}+\alpha)(x-x_{j})).\\
\end{aligned}\right.
\end{equation*}
This is an initial condition resulting in a collision of two
separate single solitons. Here we choose
$x_{0}=0,L=100,\alpha=0.5,\beta=\frac{2}{3},a_{1}=1,a_{2}=0.8,\gamma_{1}=1.5,
\gamma_{2}=-1.5,x_{1}=20,x_{2}=80.$ Take the temporal stepsize $\Delta
t=0.2$ and spatial grid number $N=450$. The numerical results are
shown in Figs. \ref{EI2S}, \ref{shape2S}.

\begin{figure}[ptb]
\centering
\begin{tabular}[c]{cccc}%
  % Requires \usepackage{graphicx}
  \subfigure[Global energy (upper) and momentum (lower) errors ]{\includegraphics[width=8cm,height=6cm]{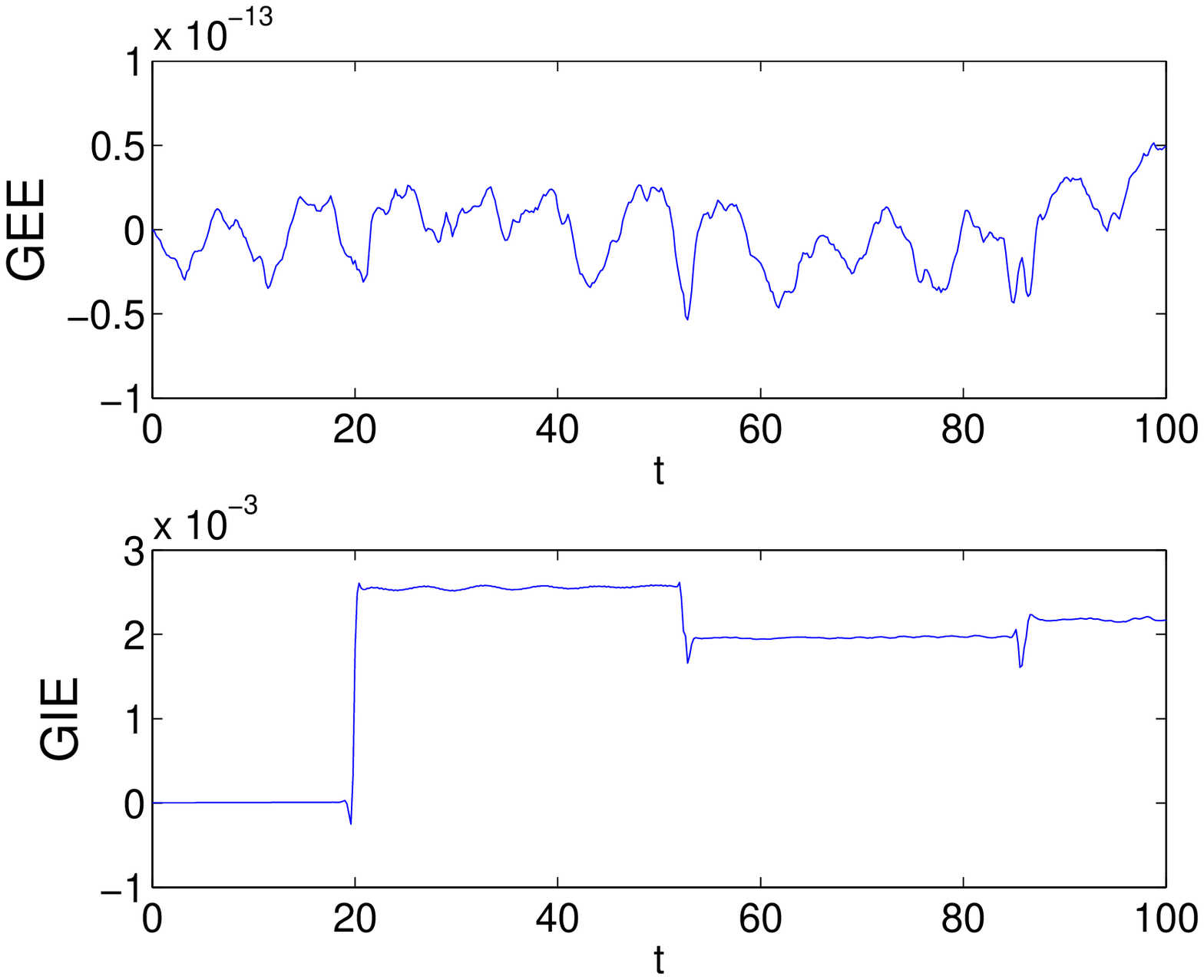}}
  \subfigure[Global charge errors of $u$ (upper) and $v$ (lower)]{\includegraphics[width=8cm,height=6cm]{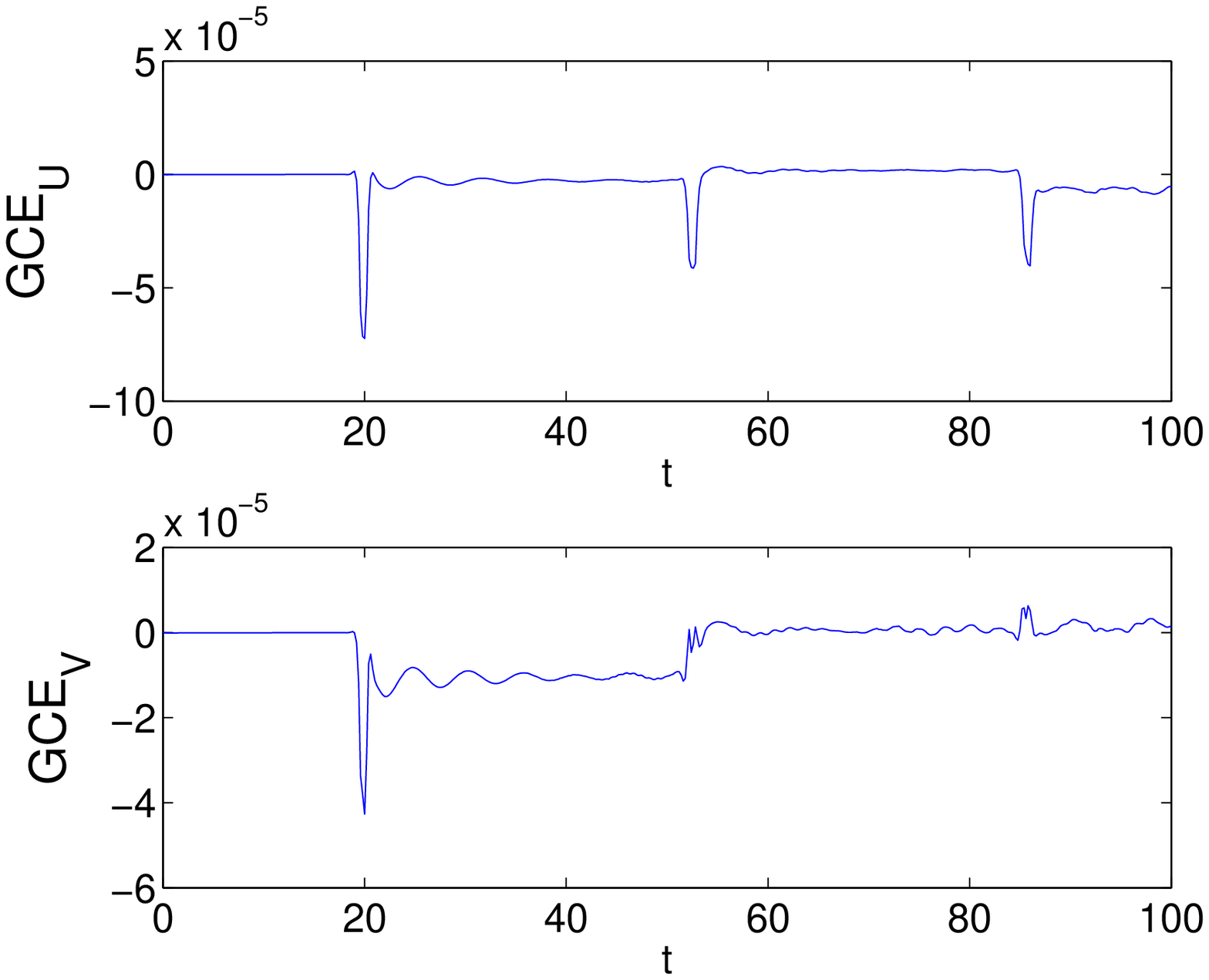}}
\end{tabular}
\caption{Errors obtained by ET4, $\Delta t=0.2,\ N=450$.}
\label{EI2S}
\end{figure}

\begin{figure}[ptb]
\centering
\begin{tabular}[c]{cccc}%
  % Requires \usepackage{graphicx}
  \subfigure[The shape of bisoliton $u$]{\includegraphics[width=8cm,height=6cm]{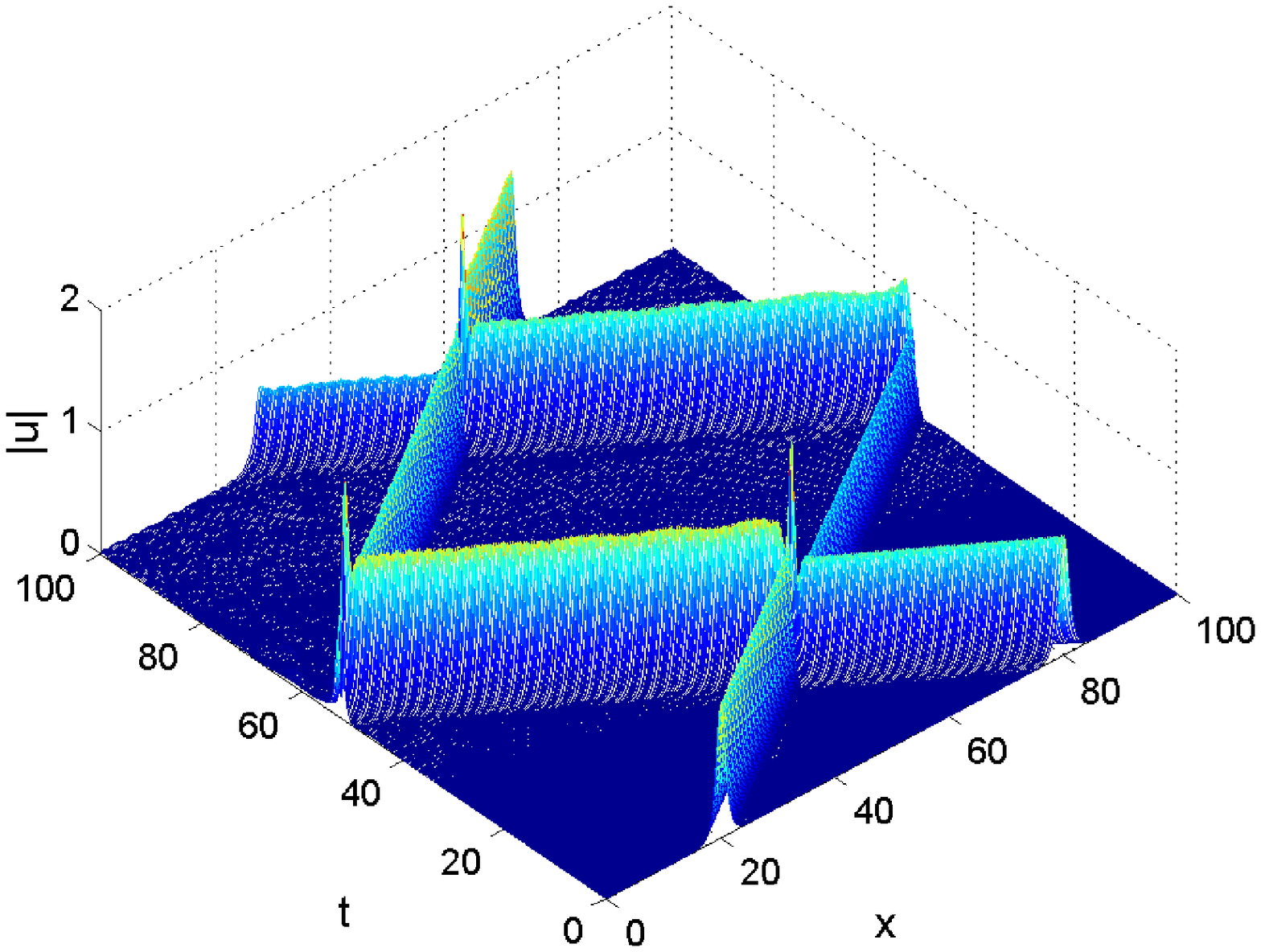}}
  \subfigure[The shape of bisoliton $v$]{\includegraphics[width=8cm,height=6cm]{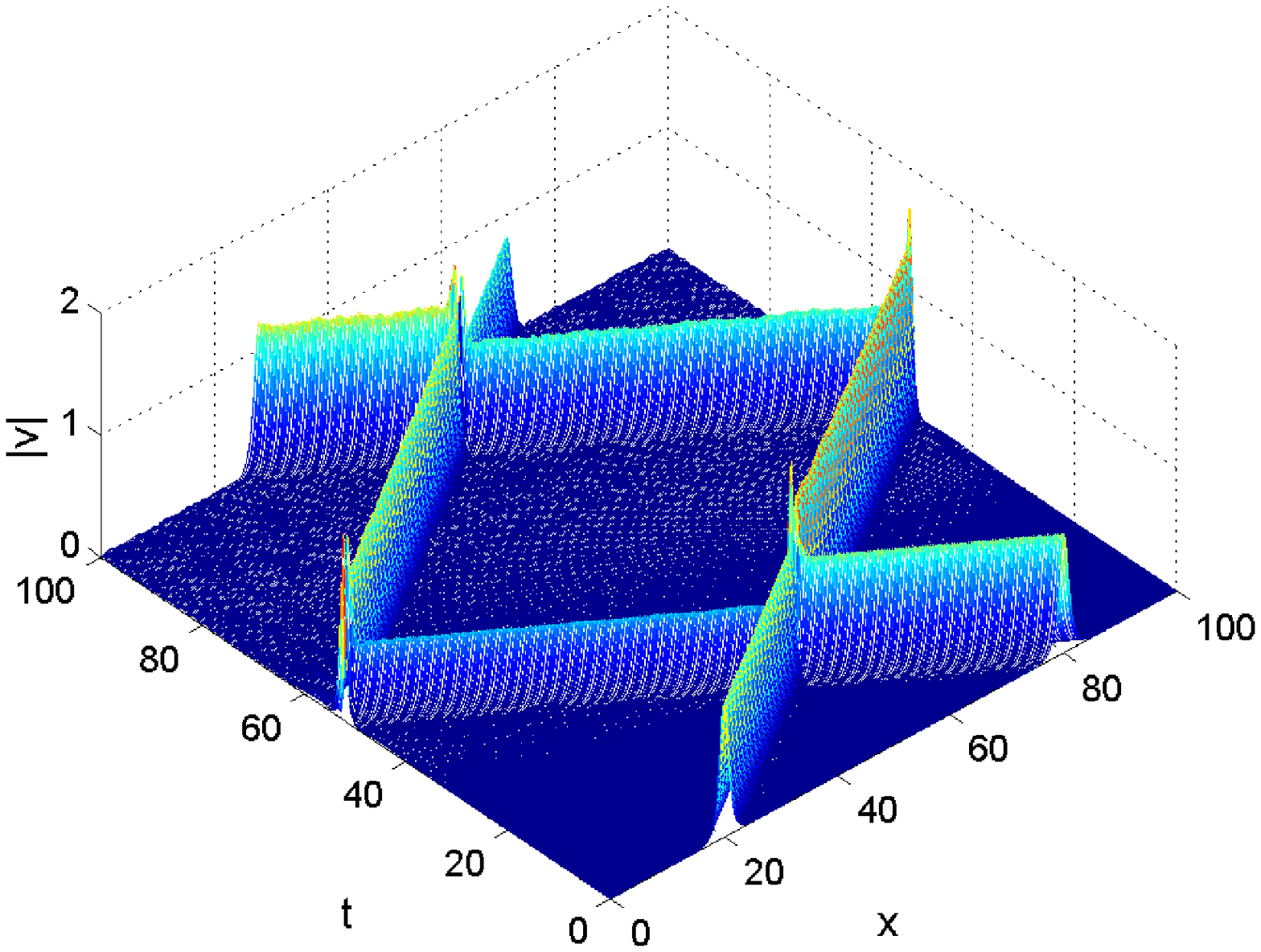}}
\end{tabular}
\caption{Numerical shapes of $u$ (left) and $v$ (right), obtained from ET4.}
\label{shape2S}
\end{figure}
Obviously, ET4 successfully simulates the collision of two solitons and the effects of boundaries on bisolitons. It preserves exactly the discrete
energy and conserves the discrete charges and momentum very well.
\end{myexp}

\begin{myexp}
\label{propo} The last experiment on the CNLS is the simulation of
the interaction among triple solitons with the initial
condition:
\begin{equation*}
\left\{
\begin{aligned}
&u(x,0)=\sum_{j=1}^{3}\sqrt{\frac{2a_{j}}{1+\beta}}sech(\sqrt{2a_{j}}(x-x_{j}))exp(i(\gamma_{j}-\alpha)(x-x_{j}),\\
&v(x,0)=\sum_{j=1}^{3}\sqrt{\frac{2a_{j}}{1+\beta}}sech(\sqrt{2a_{j}}(x-x_{j}))exp(i(\gamma_{j}+\alpha)(x-x_{j})).\\
\end{aligned}\right.
\end{equation*}
Here we also test another scheme associated with ET4. The only
difference between it and ET4 is that we evaluate the nonlinear
integrals in ET4 not by symbol calculation, but by the high-order GL
quadrature formula. In the case of ET4, the polynomials
are of degrees $6$, so we can calculate
them exactly by a $4$-point GL formula. To illustrate
the alternative scheme, we evaluate the nonlinear integrals by a
$3$-point GL formula:
\begin{equation*}
\begin{aligned}
&b_{1}=\frac{5}{18},b_{2}=\frac{4}{9},b_{3}=\frac{5}{18},\\
&c_{1}=\frac{1}{2}-\frac{\sqrt{15}}{10},c_{2}=\frac{1}{2},c_{3}=\frac{1}{2}+\frac{\sqrt{15}}{10}.\\
\end{aligned}
\end{equation*}
For example, the first nonlinear integral of \eqref{DCNLS} is
approximated by
\begin{equation*}
\begin{aligned}
&\int_{0}^{1}(((q_{1}^{\sigma})^{\cdot2}+(q_{2}^{\sigma})^{\cdot2})+\beta((q_{3}^{\sigma})^{\cdot2}+(q_{4}^{\sigma})^{\cdot2}))\cdot q_{2}^{\sigma}d\sigma\\
&\approx\sum_{i=1}^{3}b_{i}(((q_{1}^{c_{i}})^{\cdot2}+(q_{2}^{c_{i}})^{\cdot2})+\beta((q_{3}^{c_{i}})^{\cdot2}+(q_{4}^{c_{i}})^{\cdot2}))\cdot q_{2}^{c_{i}}.\\
\end{aligned}
\end{equation*}
For convenience, we denote the scheme by ET4GL6. Setting $\Delta
t=0.2,N=360,x_{0}=0,L=80,\alpha=0.5,\beta=\frac{2}{3},\gamma_{1}=1.5,\gamma_{2}=0.1,\gamma_{3}=-1.2,a_{1}=0.75,a_{2}=1,a_{3}=0.5,x_{1}=20,x_{2}=40,x_{3}=60,$
we compute it over the time interval [0,100]. Numerical results are
presented in Figs. \ref{EI3ET4}, \ldots, \ref{shape3S}. The
behaviours of ET4, ET4GL6 are very similar in conserving momentum.
Unsurprisingly, ET4 and ST4 preserve exactly the discrete global
energy and charges, respectively. However, ET4GL6 can conserve the
discrete energy in magnitude $10^{-6}$ , while ST4 only preserves the
energy in magnitude $10^{-4}$. So if we give more weight on the discrete
energy, ET4GL6 is a favourable scheme. In fact, when the nonlinear
integrals cannot be calculated exactly or have to be
integrated in very complicated forms, ETGL6 is a reasonable
alternative scheme.

\begin{figure}[ptb]
\centering
\begin{tabular}[c]{cccc}%
  % Requires \usepackage{graphicx}
  \subfigure[Global energy (upper) and momentum (lower) errors ]{\includegraphics[width=8cm,height=6cm]{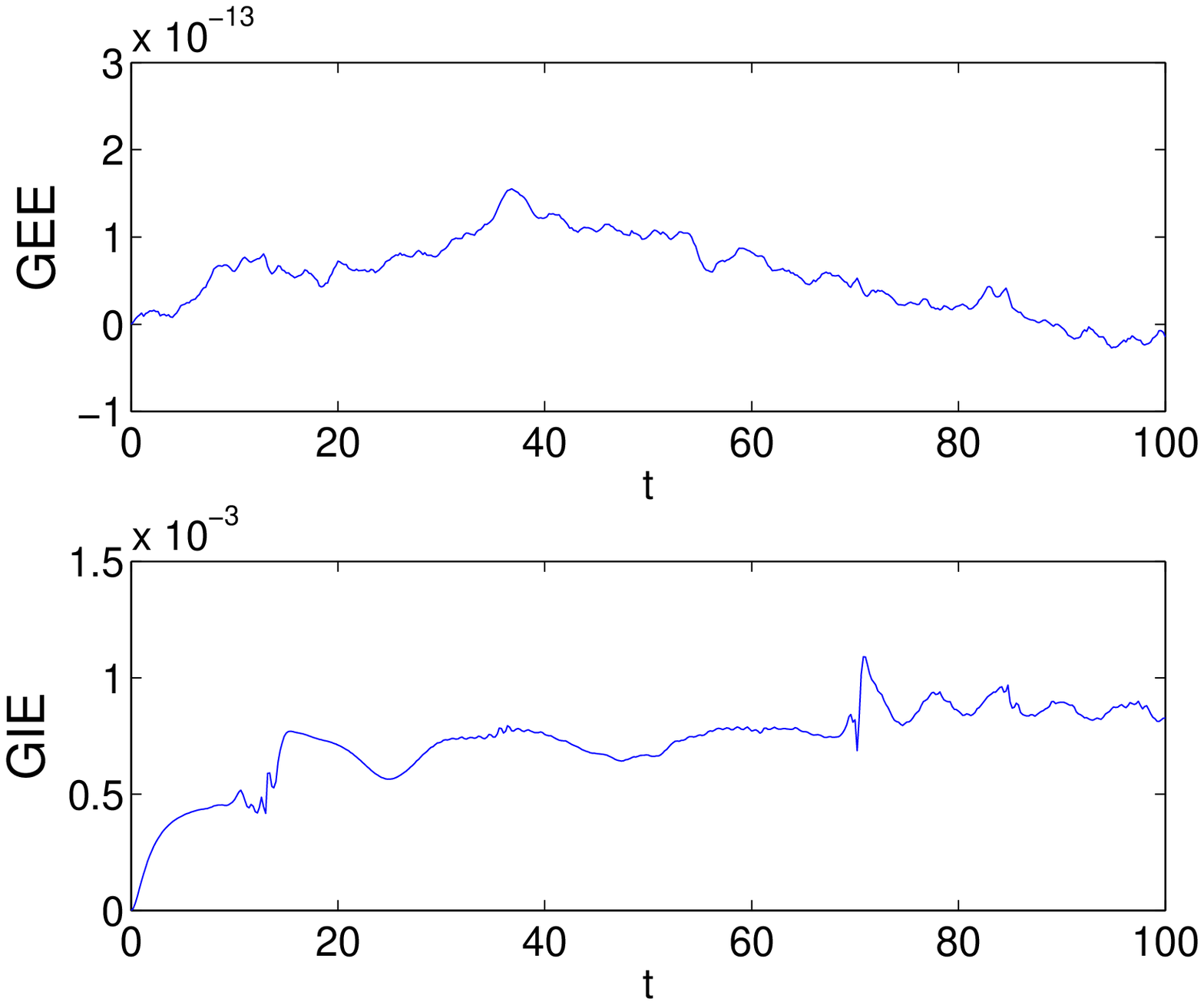}}
  \subfigure[Global charge errors of $u$ (upper) and $v$ (lower)]{\includegraphics[width=8cm,height=6cm]{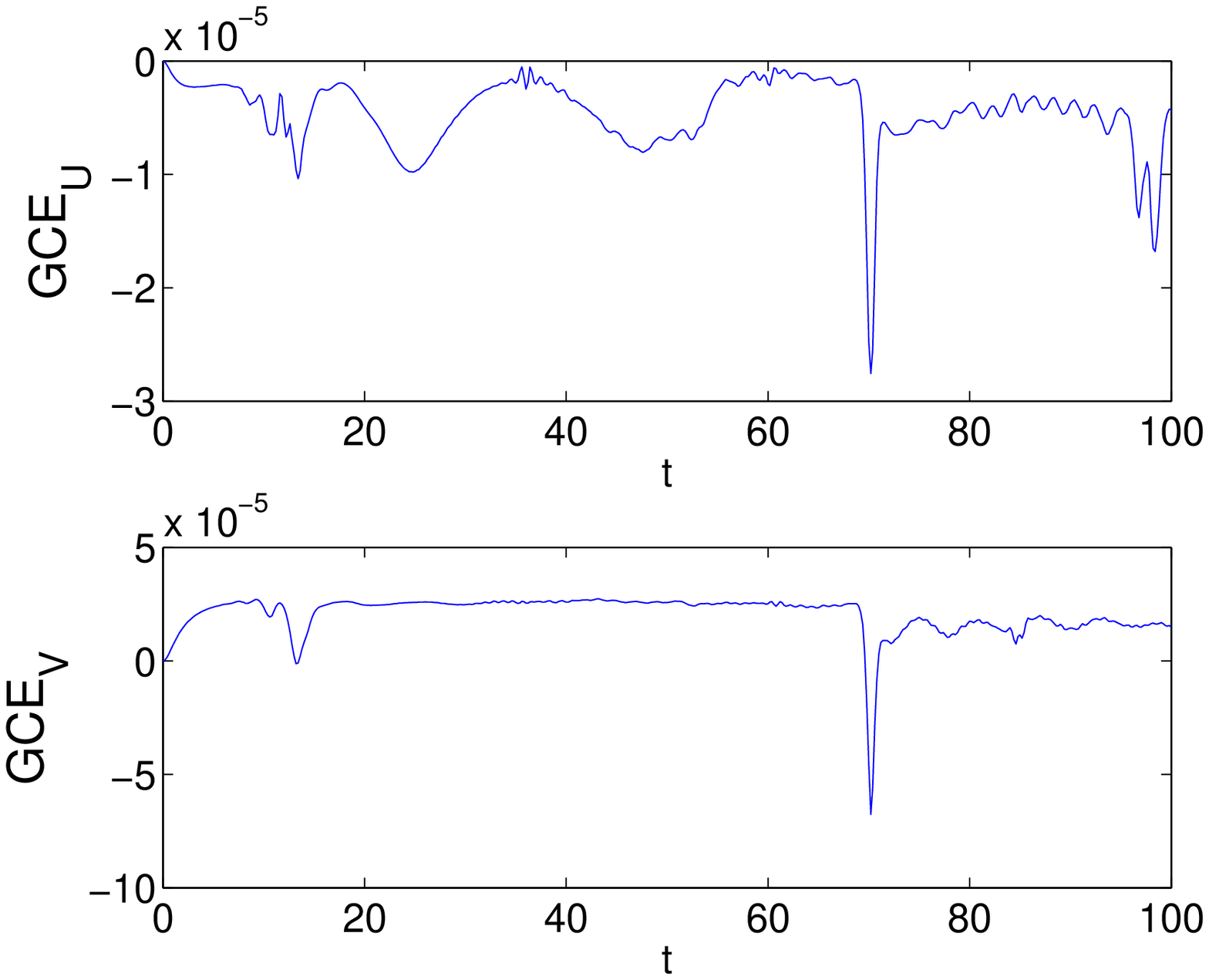}}
\end{tabular}
\caption{Errors obtained by ET4, $\Delta t=0.2,\ N=360$.}
\label{EI3ET4}
\end{figure}

\begin{figure}[ptb]
\centering
\begin{tabular}[c]{cccc}%
  % Requires \usepackage{graphicx}
  \subfigure[Global energy (upper) and momentum (lower) errors ]{\includegraphics[width=8cm,height=6cm]{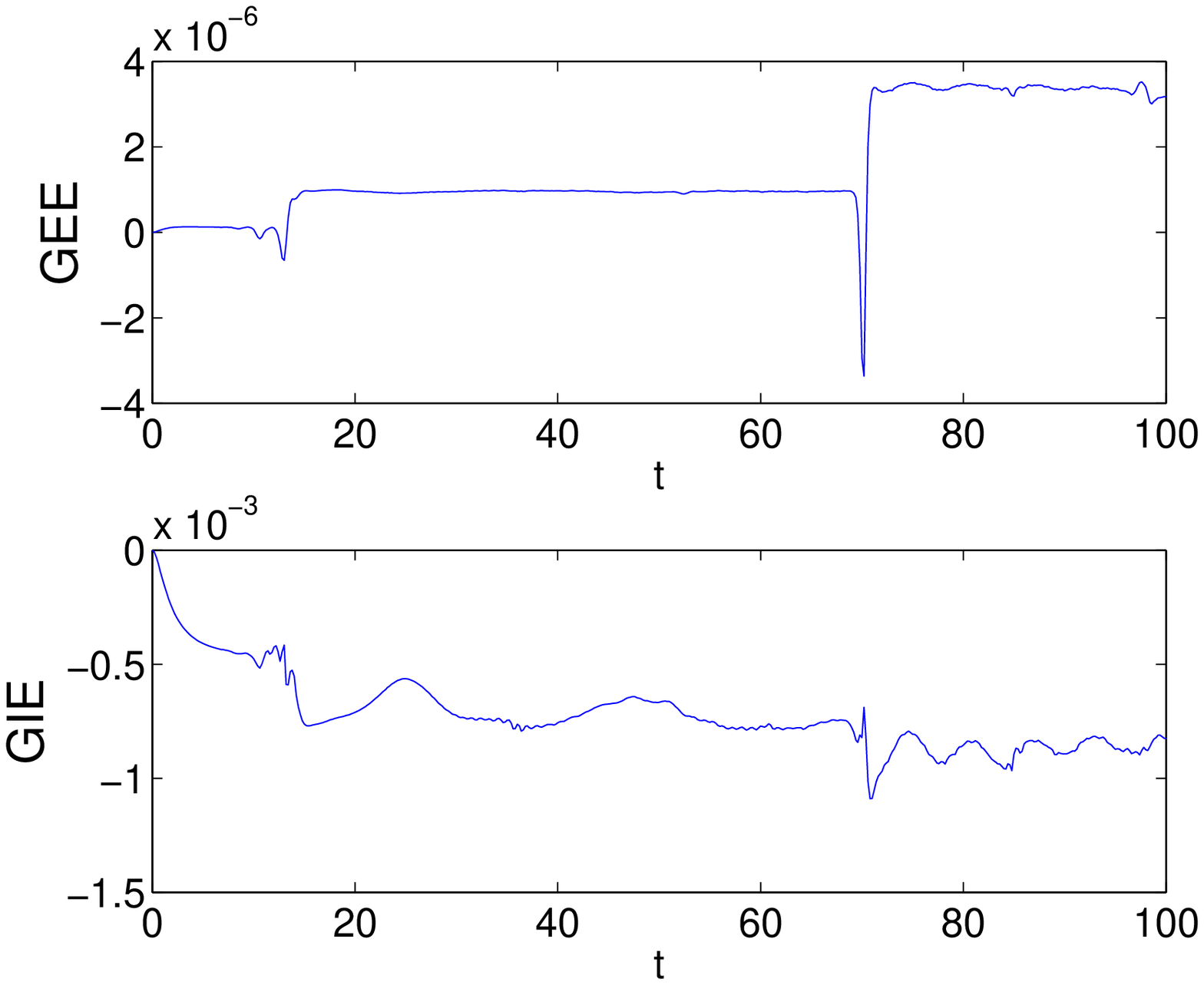}}
  \subfigure[Global charge errors of $u$ (upper) and $v$ (lower)]{\includegraphics[width=8cm,height=6cm]{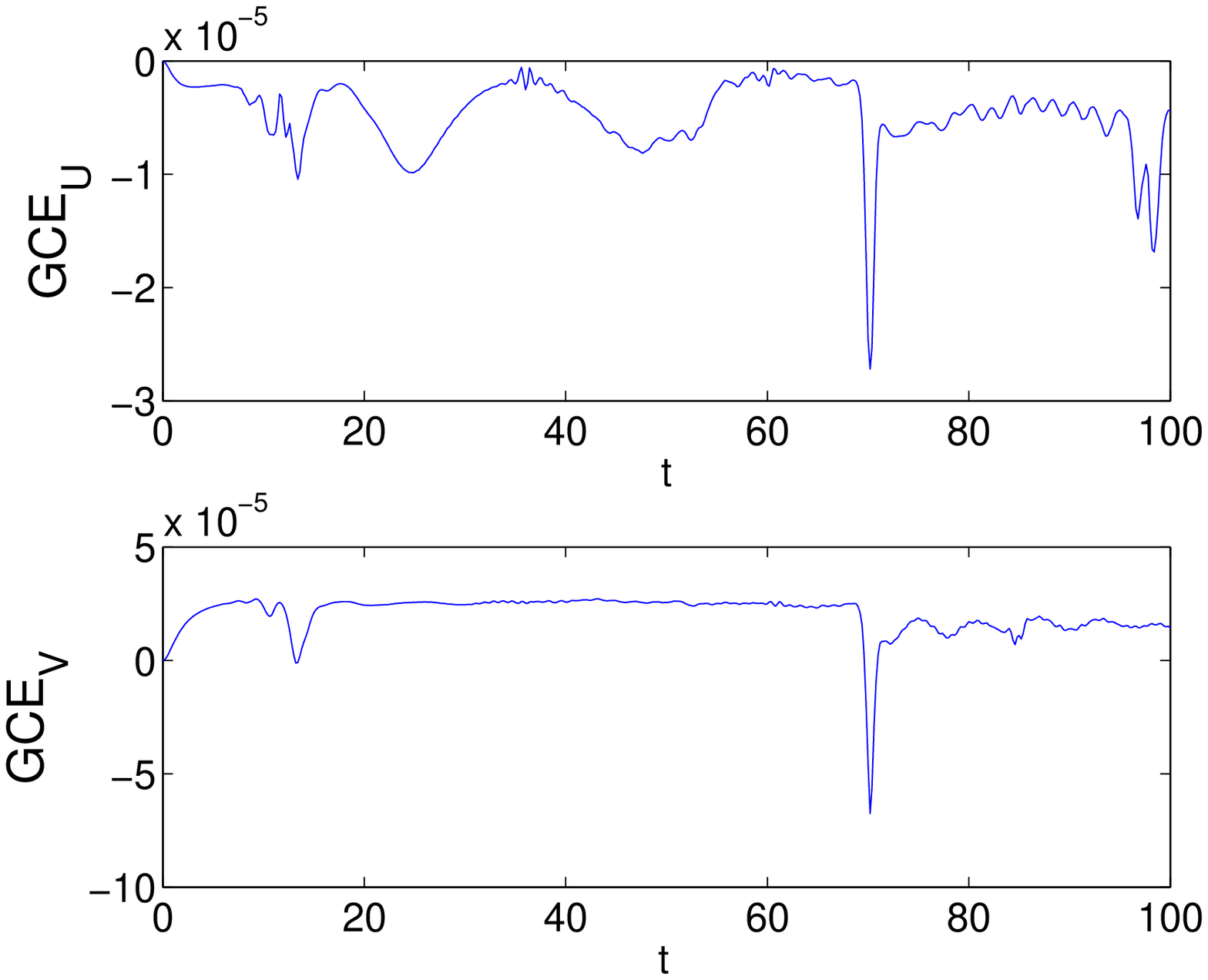}}
\end{tabular}
\caption{Errors obtained by ET4GL6, $\Delta t=0.2,\ N=360$.}
\label{EI3GL6}
\end{figure}

\begin{figure}[ptb]
\centering
\begin{tabular}[c]{cccc}%
  % Requires \usepackage{graphicx}
  \subfigure[Global energy (upper) and momentum (lower) errors ]{\includegraphics[width=8cm,height=6cm]{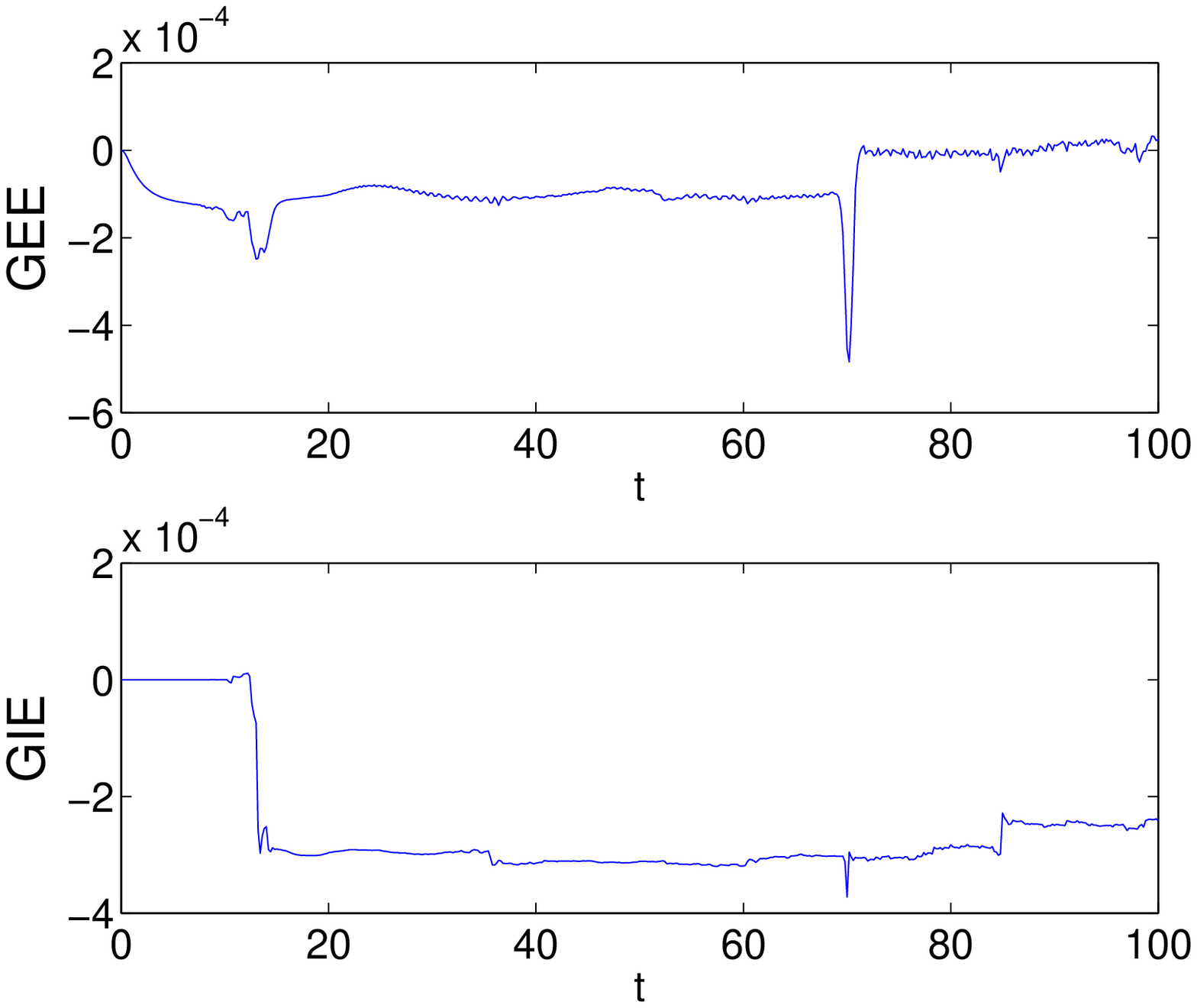}}
  \subfigure[Global charge errors of $u$ (upper) and $v$ (lower)]{\includegraphics[width=8cm,height=6cm]{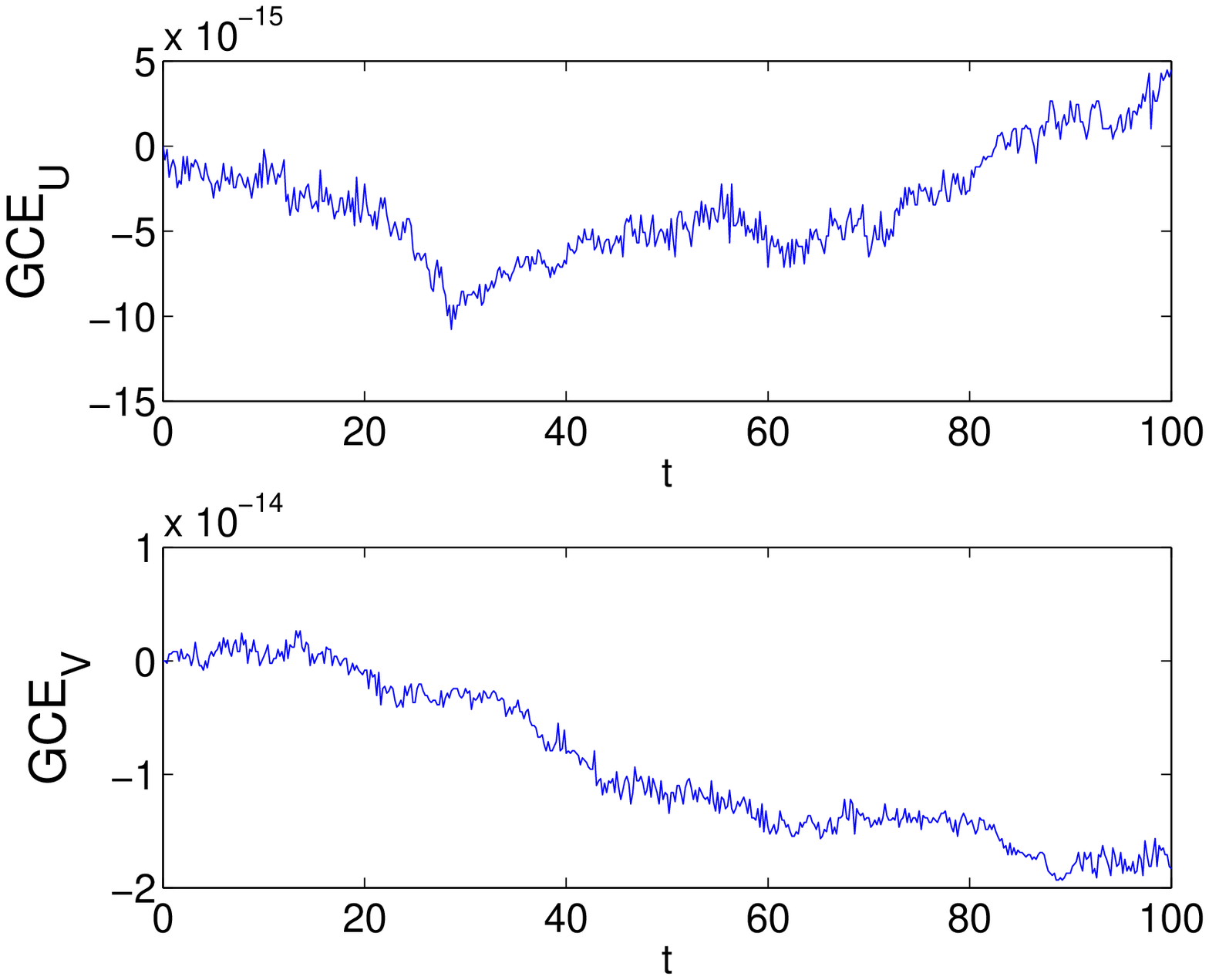}}
\end{tabular}
\caption{Errors obtained by ST4,  $\Delta t=0.2,\ N=360$.}
\label{EIGL3S}
\end{figure}

\begin{figure}[ptb]
\centering
\begin{tabular}[c]{cccc}%
  % Requires \usepackage{graphicx}
  \subfigure[]{\includegraphics[width=8cm,height=7cm]{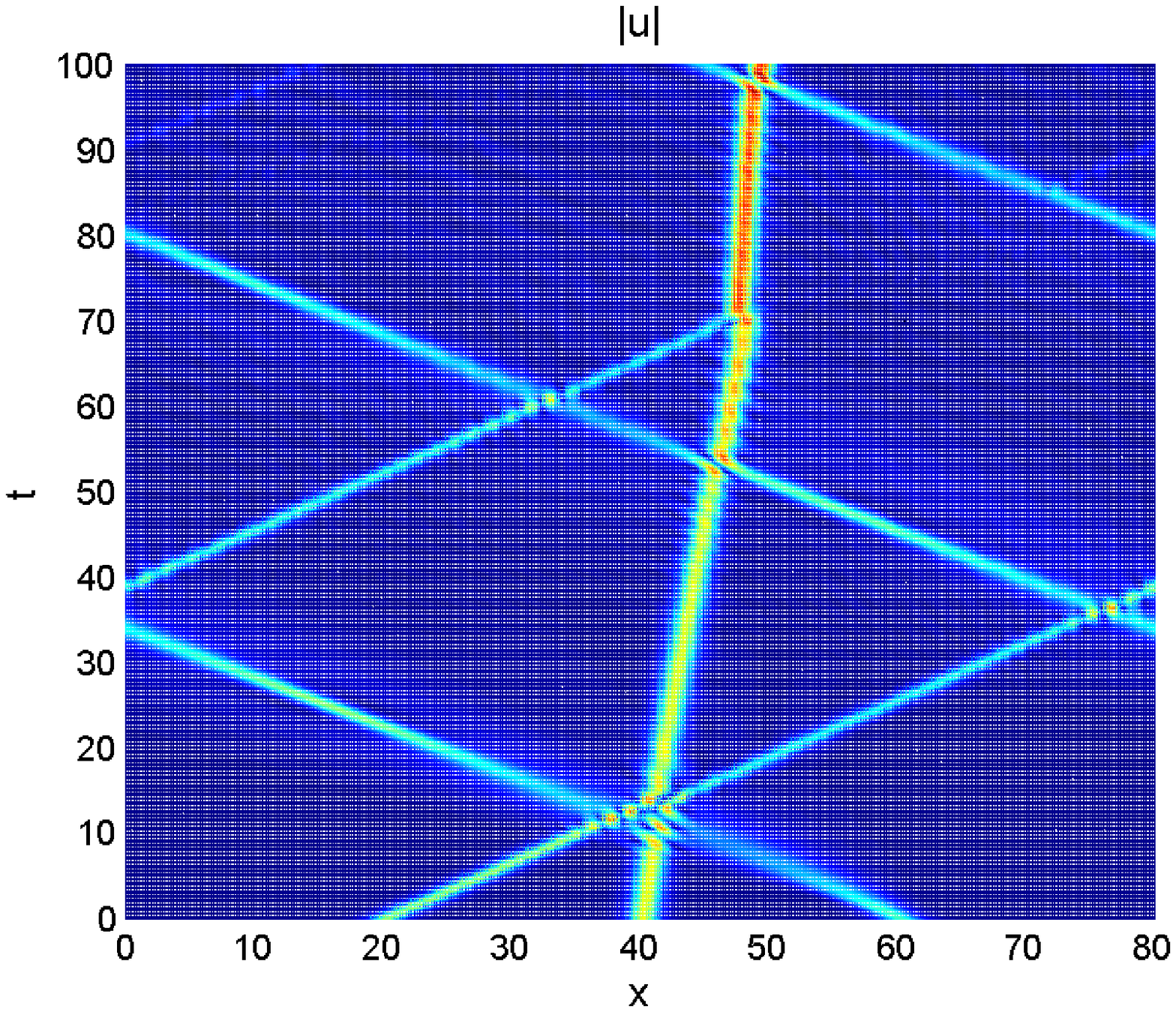}}
  \subfigure[]{\includegraphics[width=8cm,height=7cm]{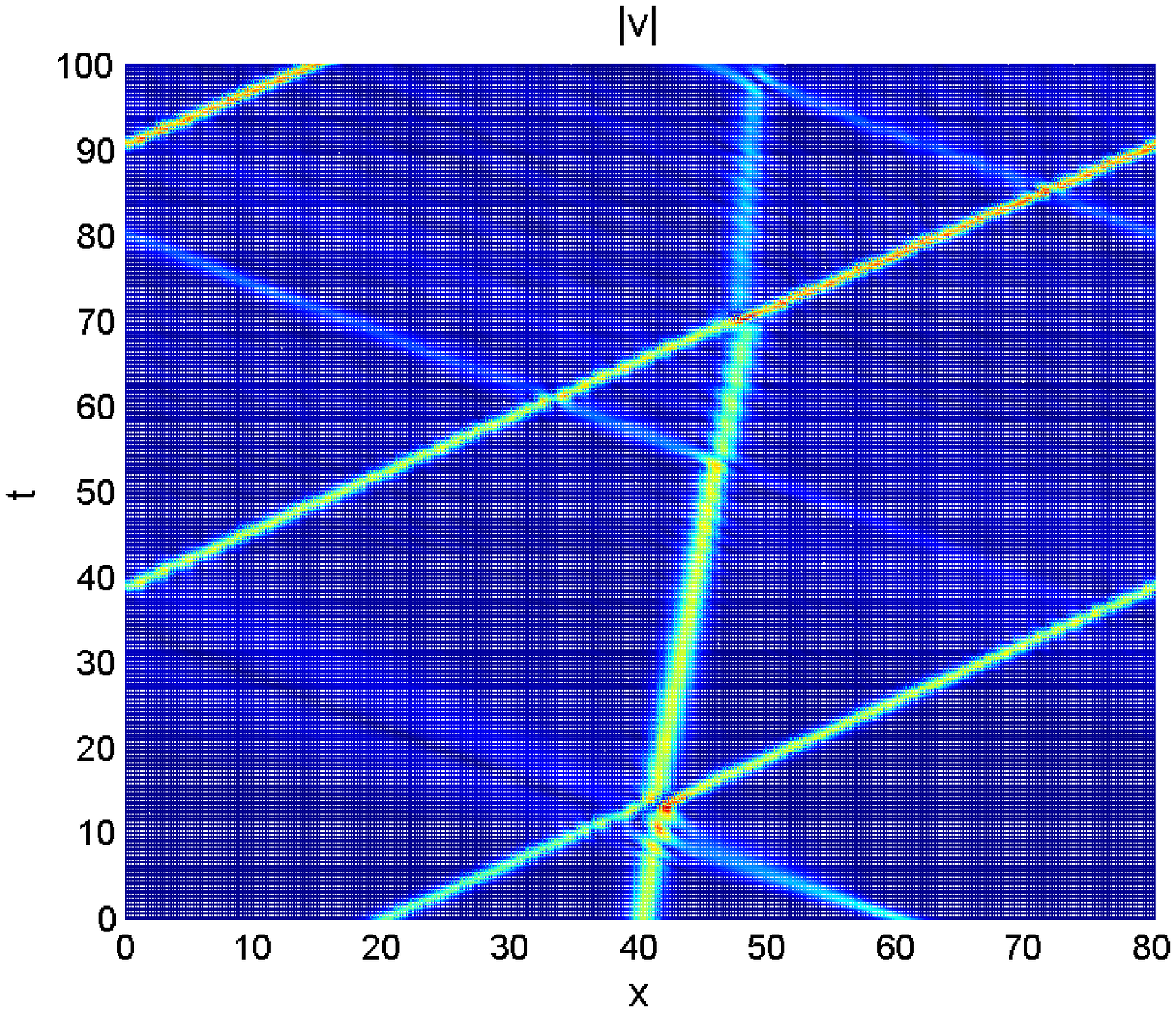}}
\end{tabular}
\caption{Numerical solitons of $u,v$, obtained by ET4.}
\label{shape3S}
\end{figure}
\end{myexp}

\section{Numerical experiments for 2D nonlinear Schr\"{o}dinger equations}\label{Numer_Experi2}
In this section, we apply the CRK method of second-order (i.e. average
vector field method) to t-direction and the pseudospectral method to $x$
and $y$ directions. This scheme is denoted by ET2. To illustrate our
method, we will compare it with another prominent traditional scheme
which is obtained by the implicit midpoint temporal discretization
and the pseudospectral spatial discretization(ST2). If \eqref{HDNLS} is
linear, our scheme ET2 is the same as ST2. Hence we will not give
numerical examples of 2D linear Schr\"{o}dinger equations.

The boundary condition is always taken
to be periodic:
\begin{equation}
u(x_{l},y,t)=u(x_{r},y,t),u(x,y_{l},t)=u(x,y_{r},t).
\end{equation}
And the grid numbers of $x$ and $y$ directions are denoted by $N$
and $M$, respectively.

The discrete global charge $CH$ will still be taken into account:
\begin{equation*}
CH^{n}=\Delta x\Delta
y\sum_{j=0}^{N-1}\sum_{l=0}^{M-1}((p_{jl}^{n})^2+(q_{jl}^{n})^2),
\end{equation*}
where
$$CH^{n}\approx\int_{x_{l}}^{x_{r}}\int_{y_{l}}^{y_{r}}(p(x,y,n\Delta t)^{2}+q(x,y,n\Delta t)^{2})dxdy.$$

Besides, the residuals in the ECL \eqref{2DNLSECL} are defined as:
\begin{equation*}
R_{jl}^{n}=\frac{E_{jl}^{n+1}-E_{jl}^{n}}{\Delta t}+\sum_{k=0}^{N-1}(D_{x})_{jk}\bar{F}_{jk,l}+\sum_{m=0}^{M-1}(D_{y})_{lm}\bar{G}_{j,lm},
\end{equation*}
for $j=0,1,\ldots,N-1,\quad l=0,1,\ldots,M-1.$

In this section, we calculate $\mathbf{R}^{n}$ : the residual with the maximum absolute value at the time level $n\Delta t$.

\begin{myexp}\label{7.1}
Let
$\alpha=\frac{1}{2},V(\xi,x,y)=V_{1}(x,y)\xi+\frac{1}{2}\beta\xi^{2},$
then \eqref{HDNLS} becomes the Gross--Pitaevskii (GP) equation:
\begin{equation}\label{GP}
i\psi_{t}+\frac{1}{2}(\psi_{xx}+\psi_{yy})+V_{1}(x,y)\psi+\beta|\psi|^{2}\psi=0.
\end{equation}
This equation is an important mean-field model for the dynamics of a
dilute gas Bose-Einstein condensate (BEC) (see, e.g.
\cite{Deco2001}). The parameter $\beta$ determines whether
\eqref{GP} is attractive ($\beta>0$) or repulsive ($\beta<0$).

Note that equation \eqref{GP} is no longer multi-symplectic, the scheme ST2 is only symplectic in time.
We first consider the attractive case $\beta=1$. The external potential $V_{1}$ is:
\begin{equation*}
V_{1}(x,y)=-\frac{1}{2}(x^{2}+y^{2})-2exp(-(x^{2}+y^{2})).
\end{equation*}
The initial condition is given by:
\begin{equation*}
\psi(x,y,0)=\sqrt{2}exp(-\frac{1}{2}(x^{2}+y^{2})).
\end{equation*}
This IVP has the exact solution (see, e.g.
\cite{Antar2013}):
\begin{equation*}
\psi(x,y,t)=\sqrt{2}exp(-\frac{1}{2}(x^{2}+y^{2}))exp(-it).
\end{equation*}
For the same reason in the experiment \ref{5.1}, we set the spatial
domain as $x_{l}=-6,x_{r}=6,y_{l}=-6,y_{r}=6.$ The temporal stepzie
is chosen as $\Delta t=0.15,0.1,0.05$, respectively. Fixing the
number of spatial grids $N=M=42,$ {we compute the numerical solution
over the time interval $[0,45]$}. The numerical results of ET2 and
ST2 are shown in Figs. \ref{GME1}, \ldots, \ref{GECL}.
\begin{figure}[ptb]
\centering
\begin{tabular}[c]{cccc}%
  % Requires \usepackage{graphicx}
  \subfigure[ $\Delta t=0.10$]{\includegraphics[width=5cm,height=4cm]{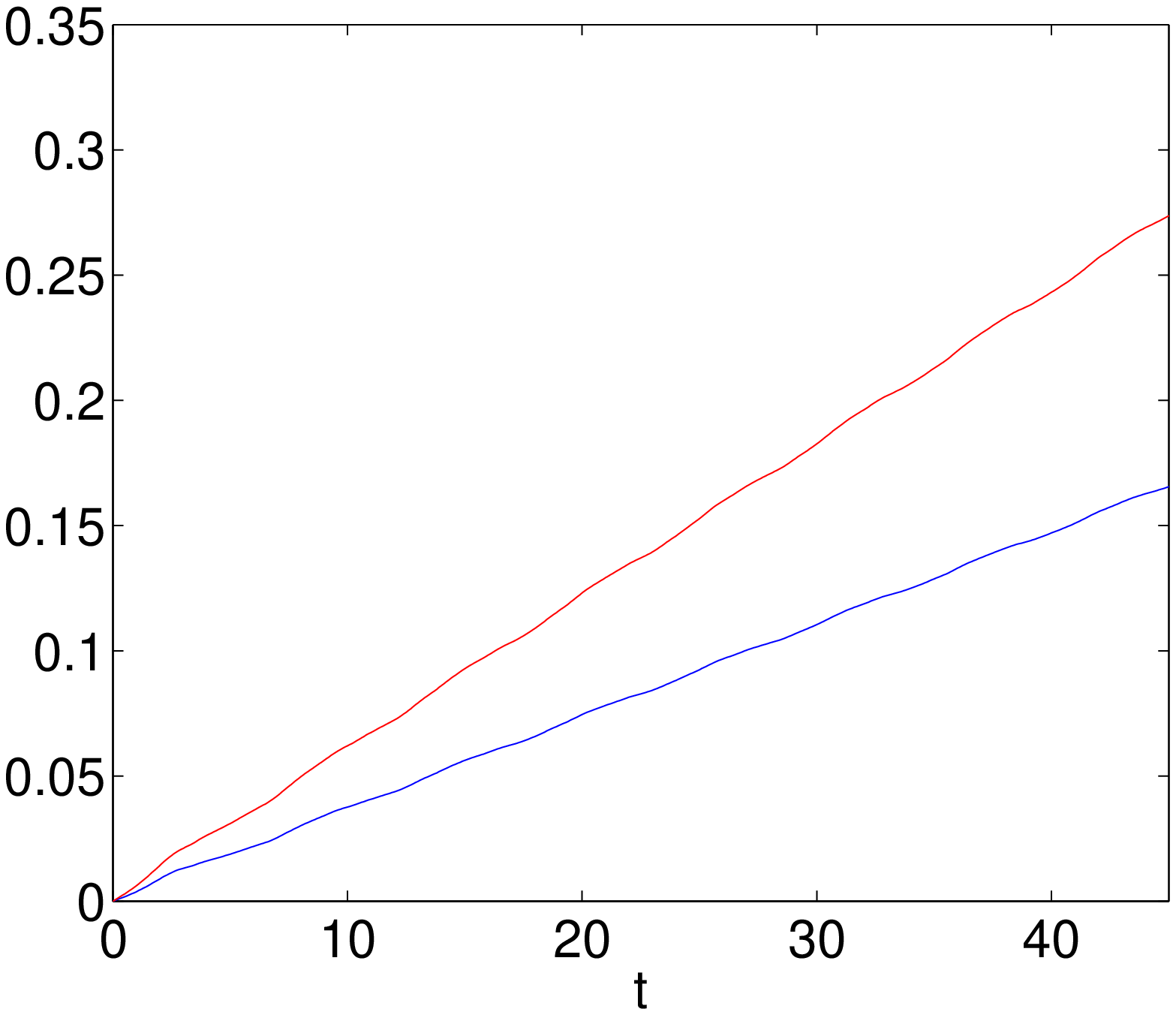}}
  \subfigure[ $\Delta t=0.08$]{\includegraphics[width=5cm,height=4cm]{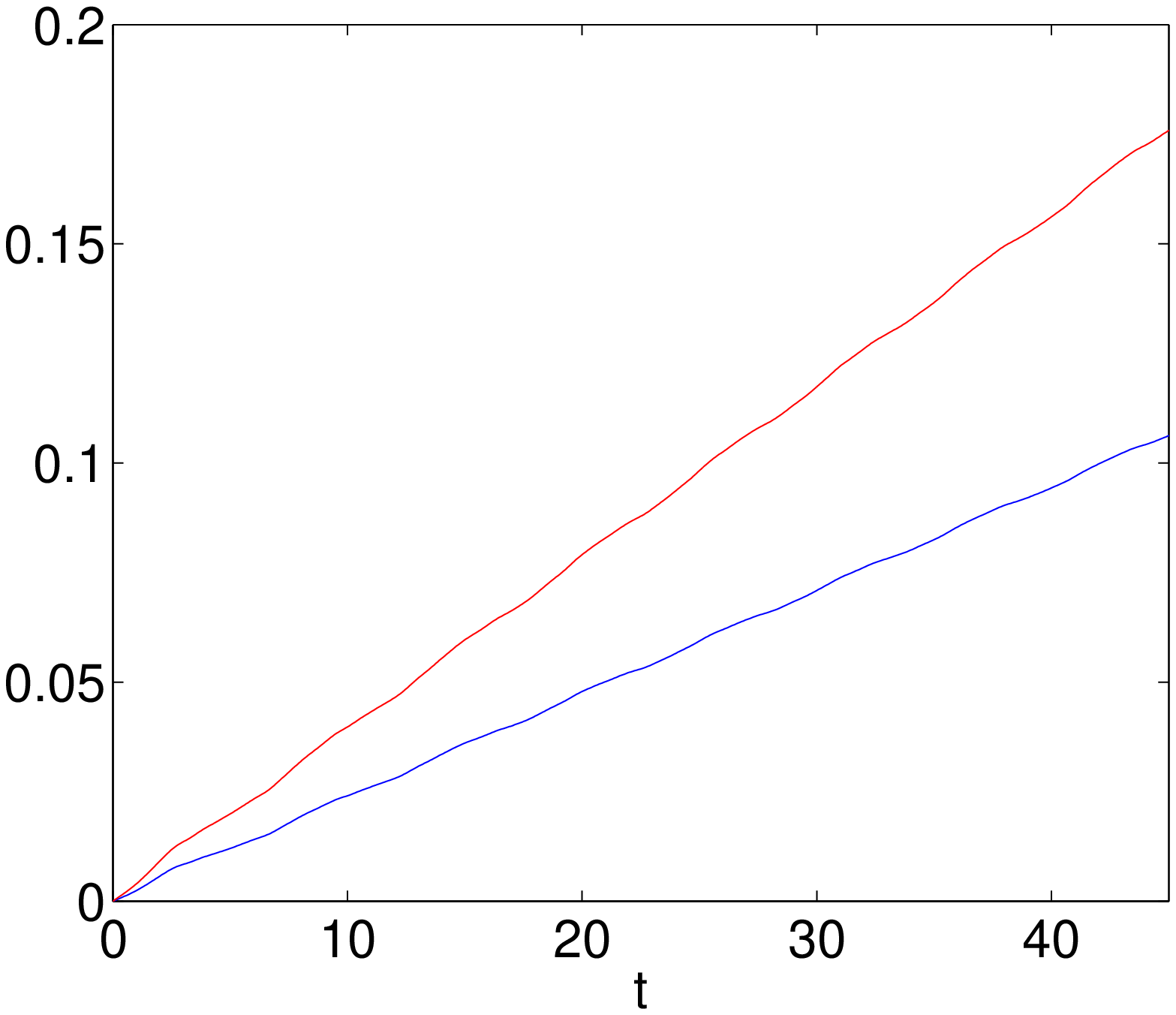}}
  \subfigure[ $\Delta t=0.05$]{\includegraphics[width=5cm,height=4cm]{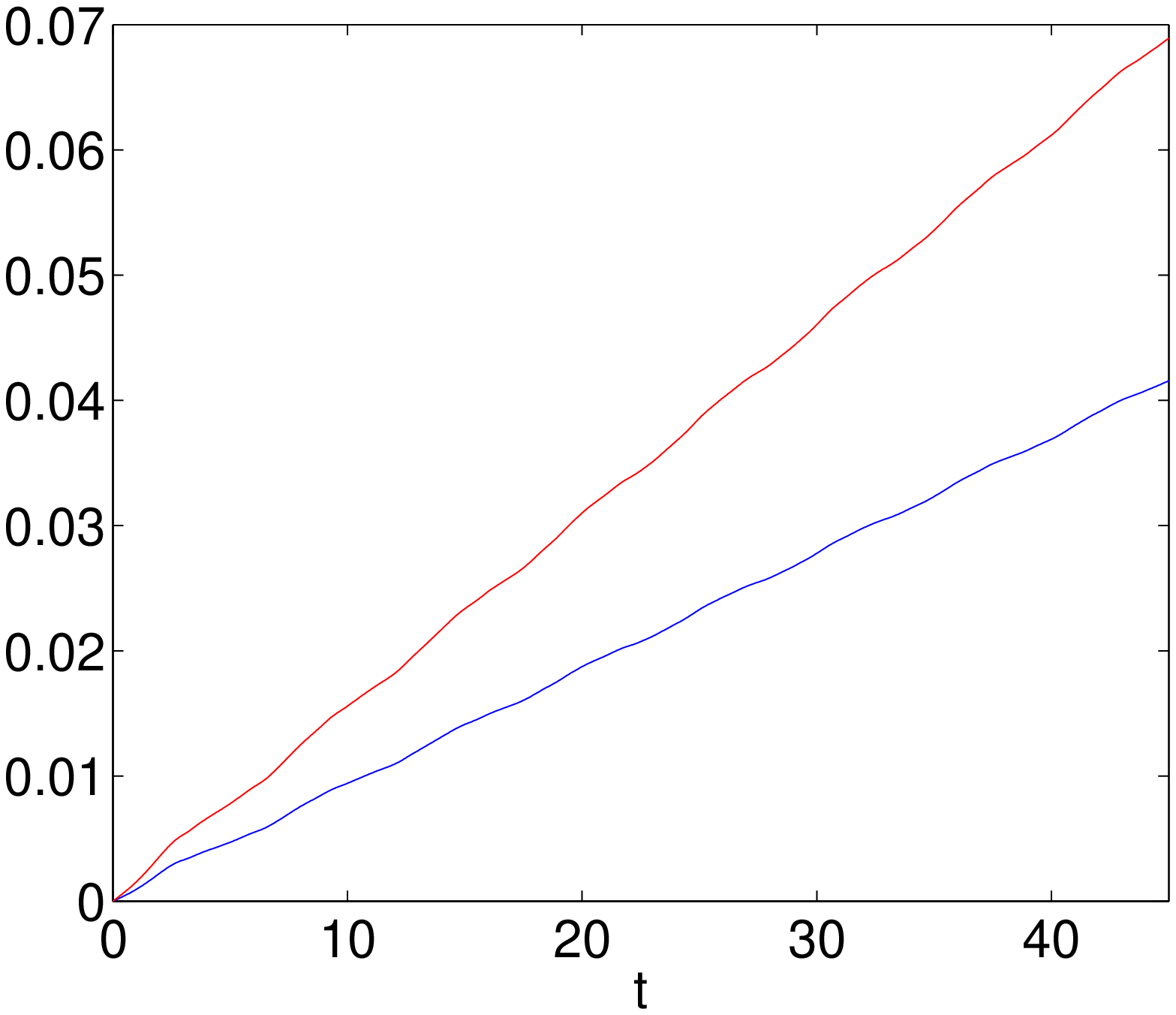}}
\end{tabular}
\caption{Maximum global errors. The blue curves are {\color{red}the}
results of ET2, the red curves are {\color{red}the} results of ST2.}
\label{GME1}
\end{figure}
\begin{figure}[ptb]
\centering
\begin{tabular}[c]{cccc}%
  % Requires \usepackage{graphicx}
  \subfigure[]{\includegraphics[width=8cm,height=4cm]{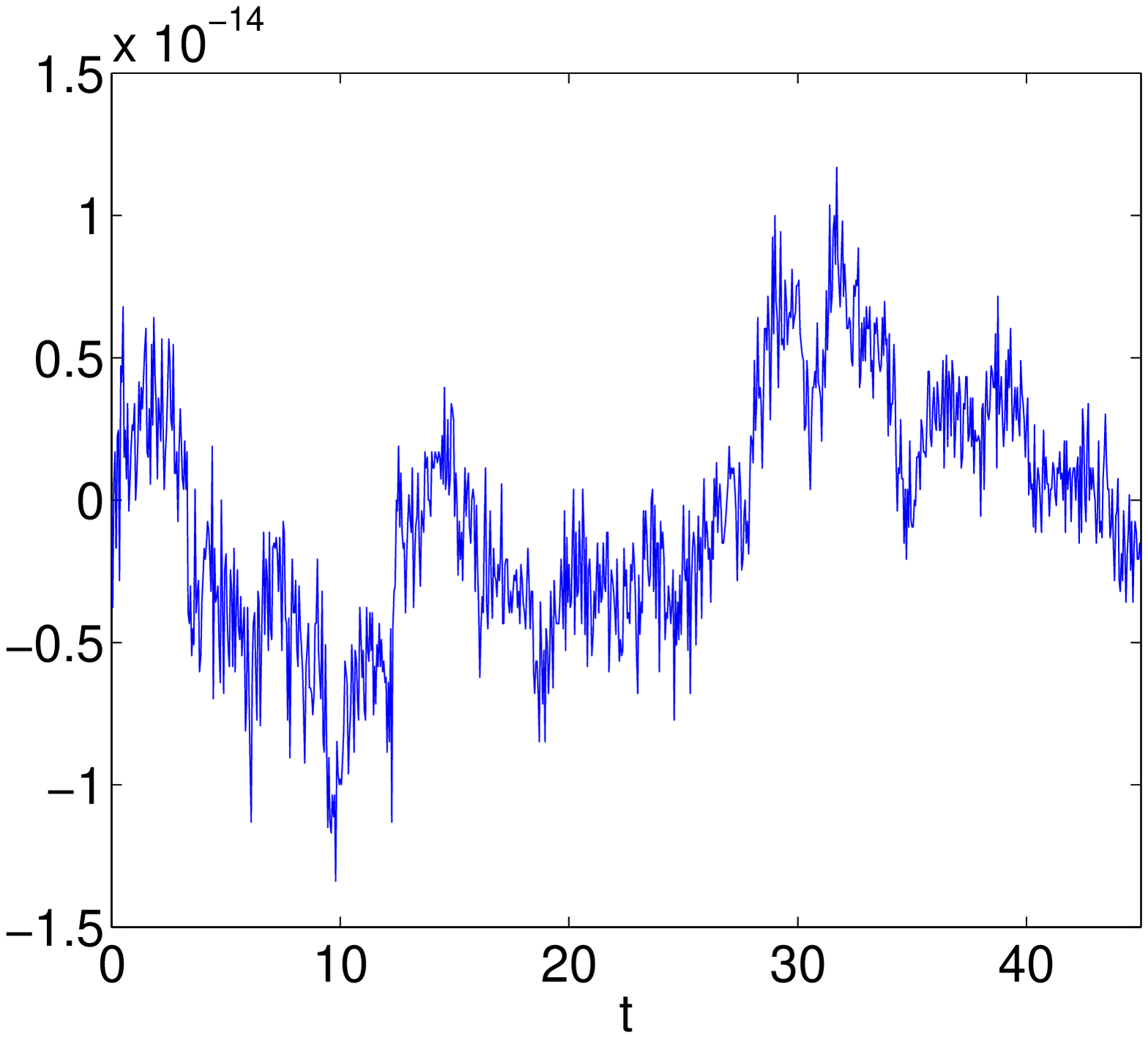}}
  \subfigure[]{\includegraphics[width=8cm,height=4cm]{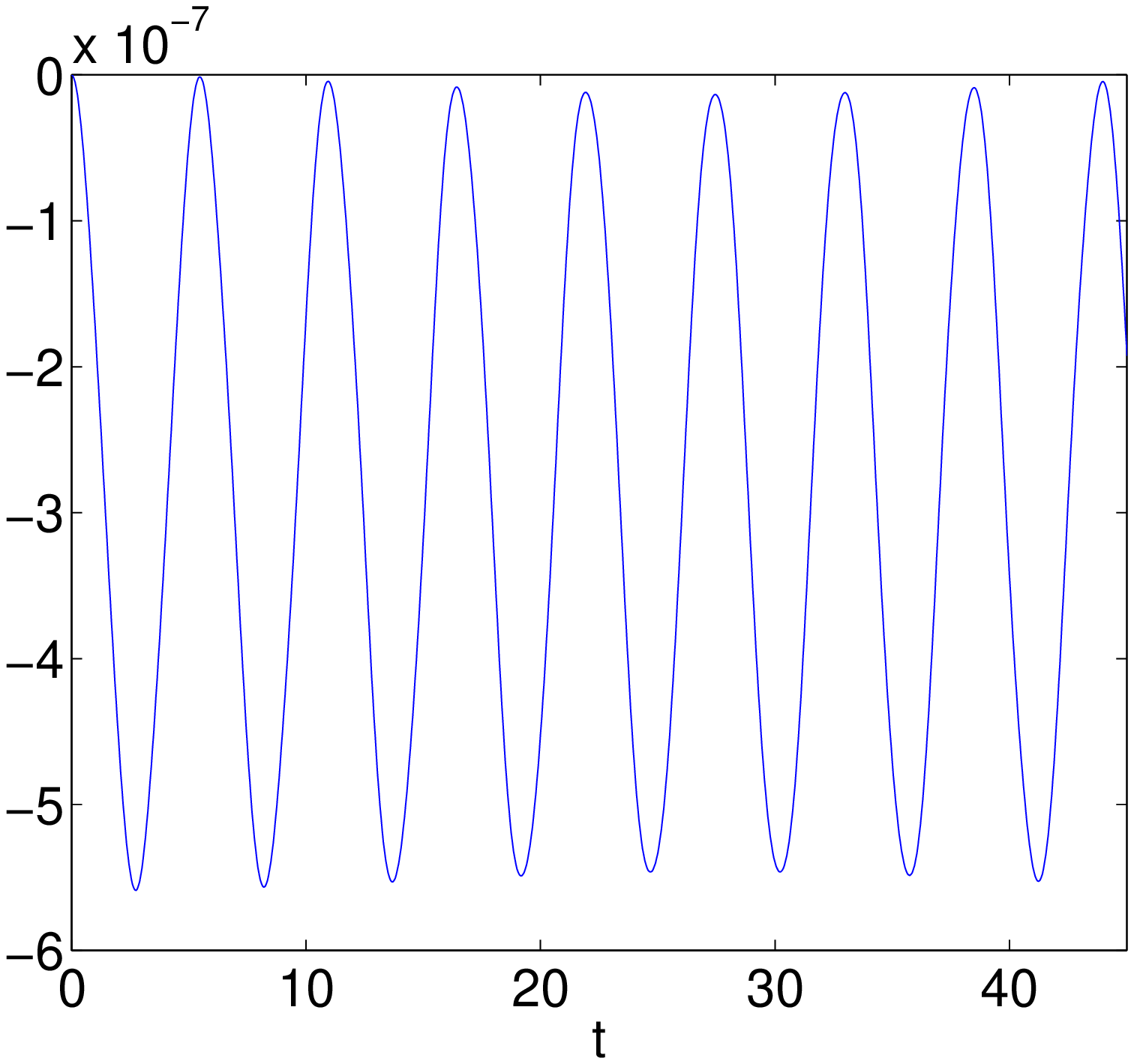}}
\end{tabular}
\caption{Global energy errors of ET2 (left) and ST2 (right), $\Delta
t=0.05$.}
\label{GEE}
\end{figure}
\begin{figure}[ptb]
\centering
\begin{tabular}[c]{cccc}%
  % Requires \usepackage{graphicx}
  \subfigure[]{\includegraphics[width=8cm,height=4cm]{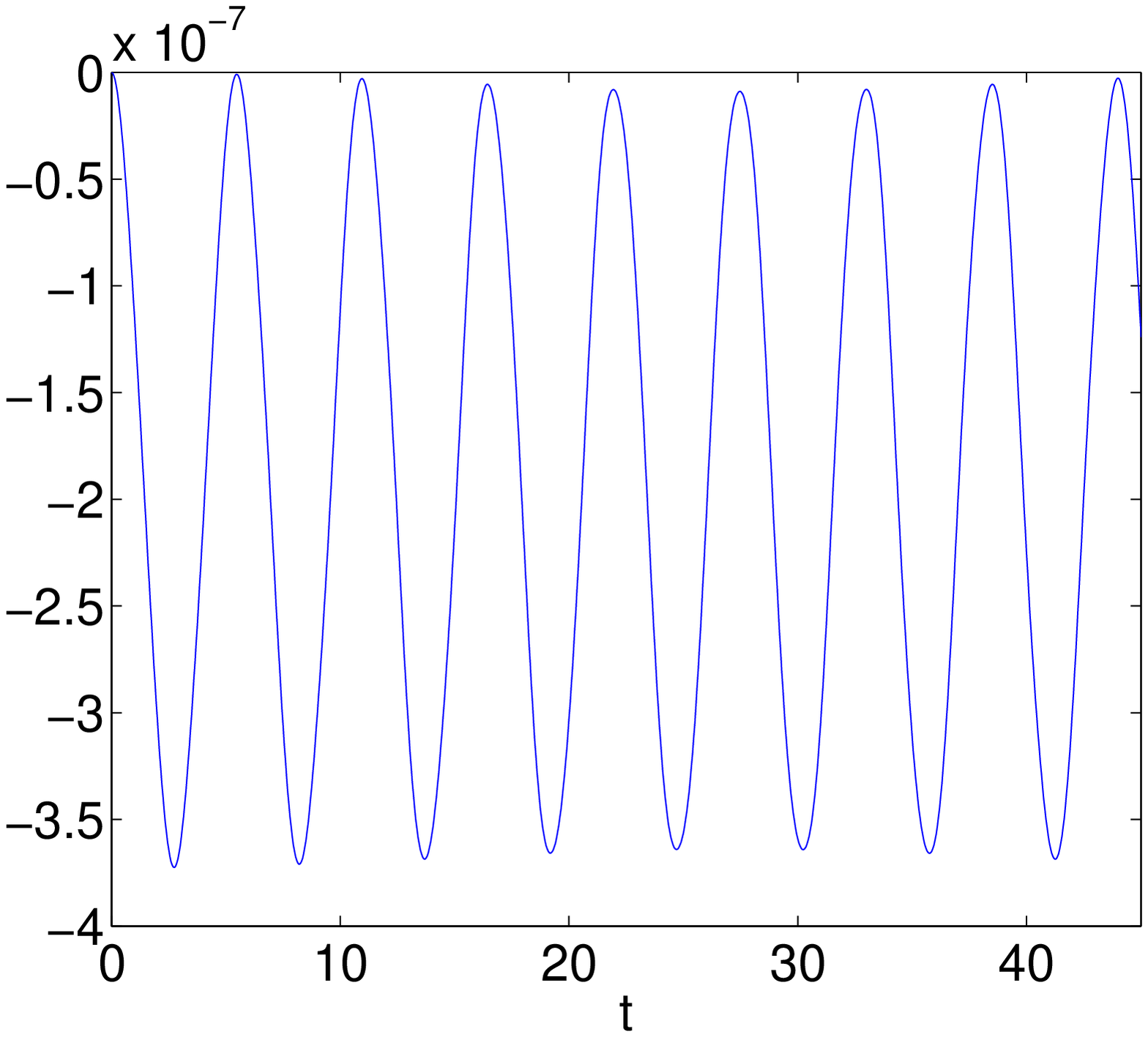}}
  \subfigure[]{\includegraphics[width=8cm,height=4cm]{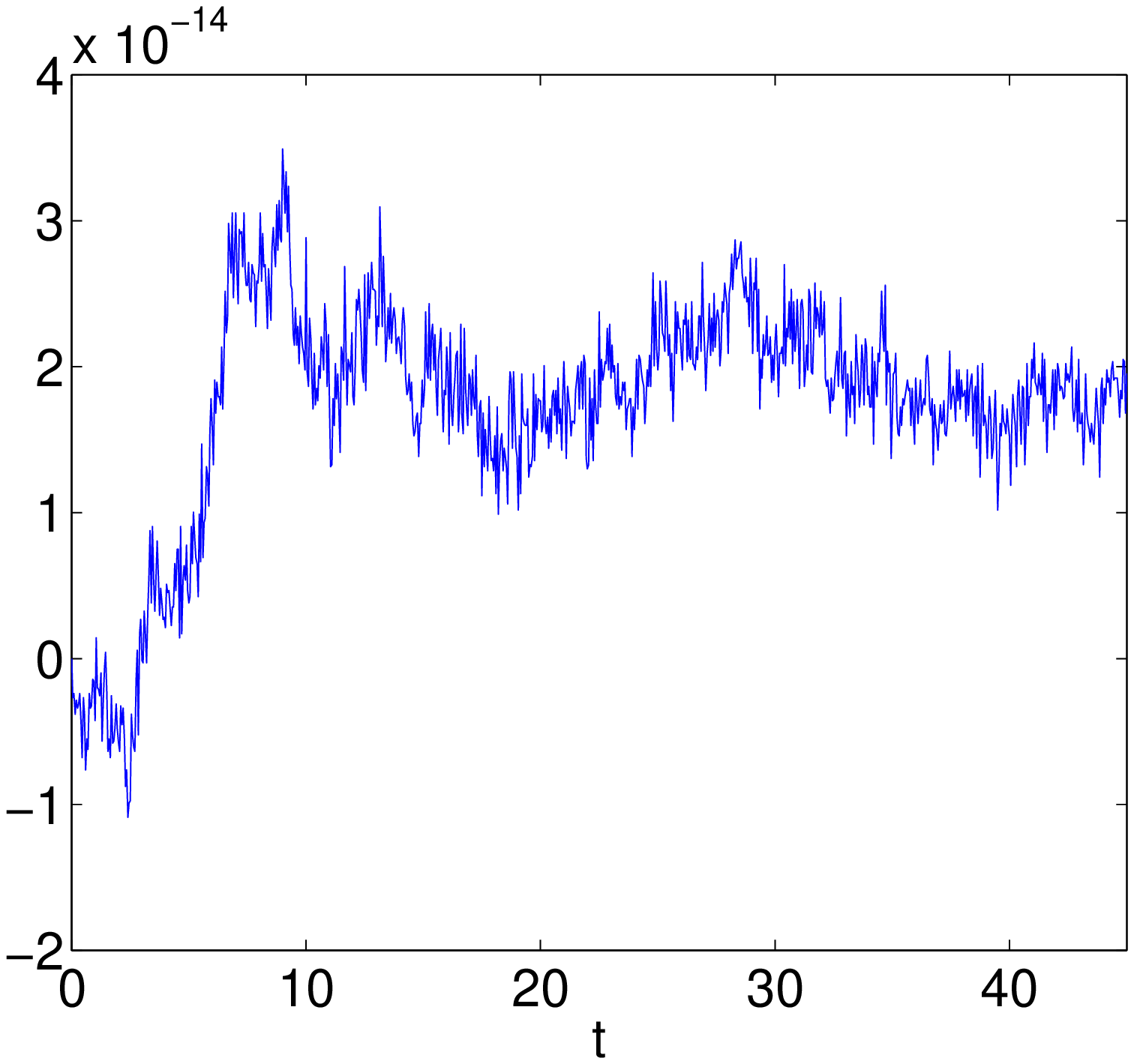}}
\end{tabular}
\caption{Global charge errors of ET2 (left) and ST2 (right), $\Delta
t=0.05$.}
\label{GCE}
\end{figure}
\begin{figure}[ptb]
\centering
\begin{tabular}[c]{cccc}%
  % Requires \usepackage{graphicx}
  \subfigure[]{\includegraphics[width=8cm,height=4cm]{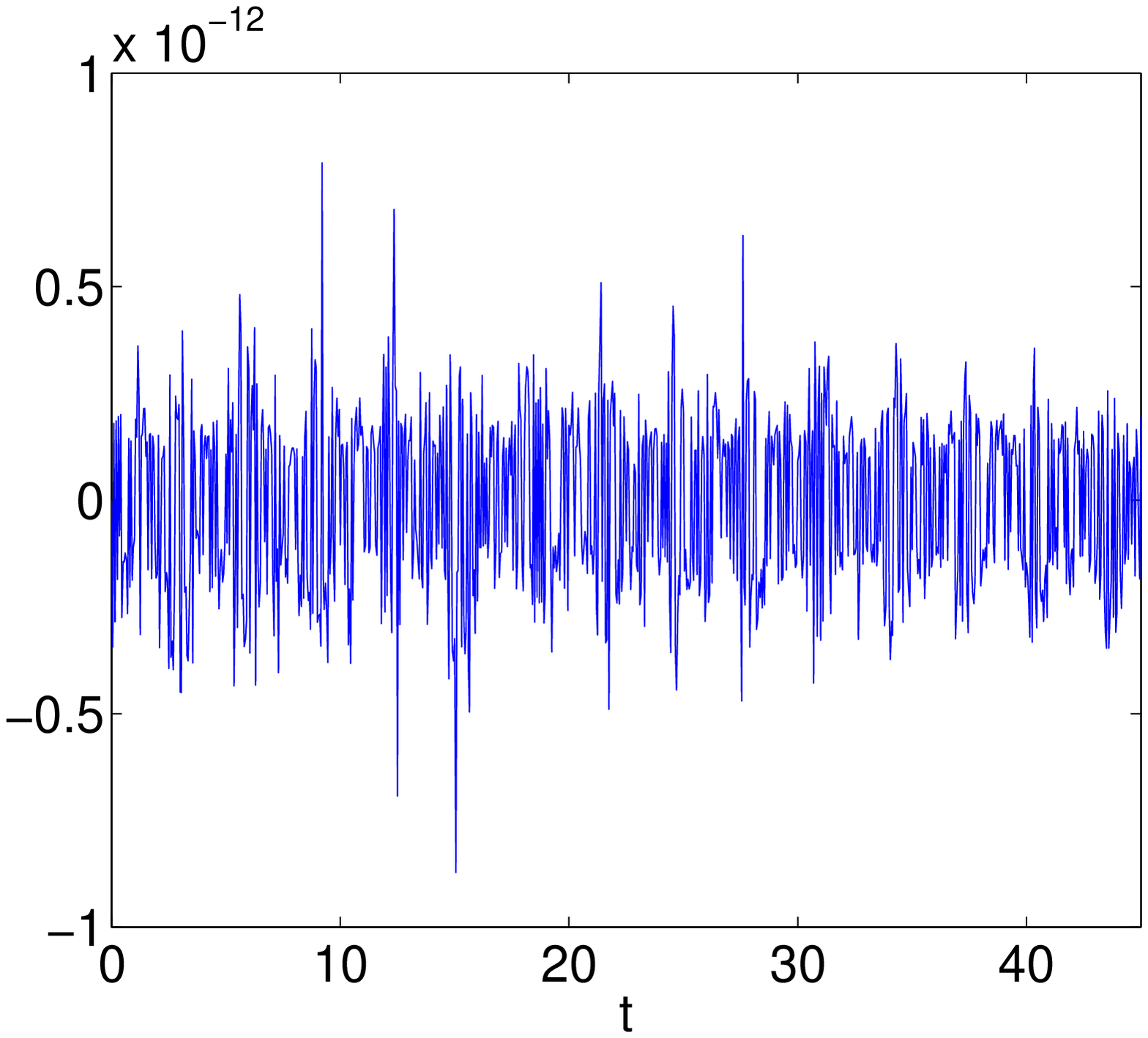}}
  \subfigure[]{\includegraphics[width=8cm,height=4cm]{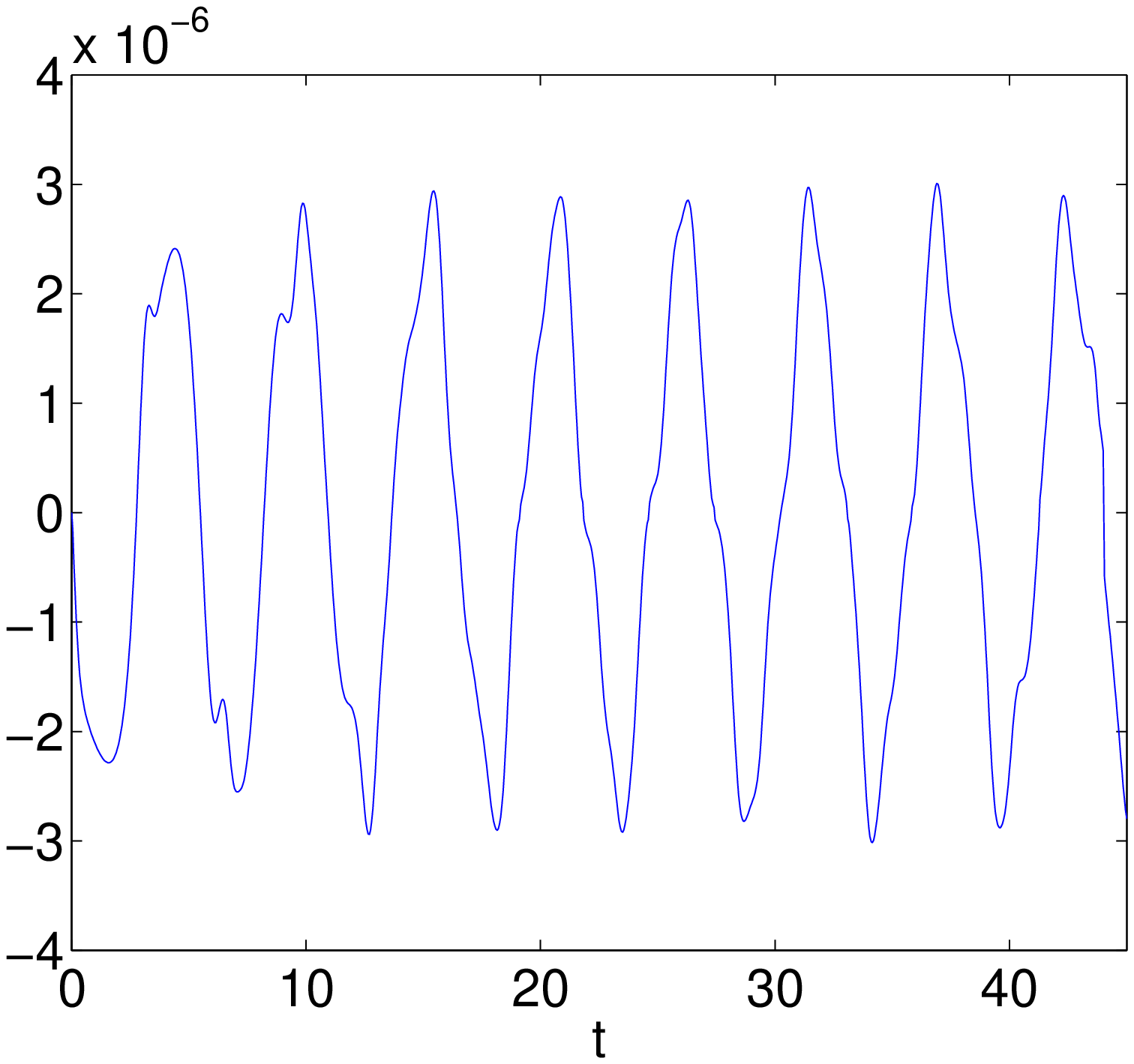}}
\end{tabular}
\caption{Maximum residuals ($\mathbf{R}$) of ET2 (left) and ST2 (right) in the ECL, $\Delta
t=0.05$.}
\label{GECL}
\end{figure}

\begin{figure}[ptb]\label{generalKGE}
\centering
\begin{tabular}[c]{cccc}%
  % Requires \usepackage{graphicx}
  \subfigure[]{\includegraphics[width=8cm,height=6cm]{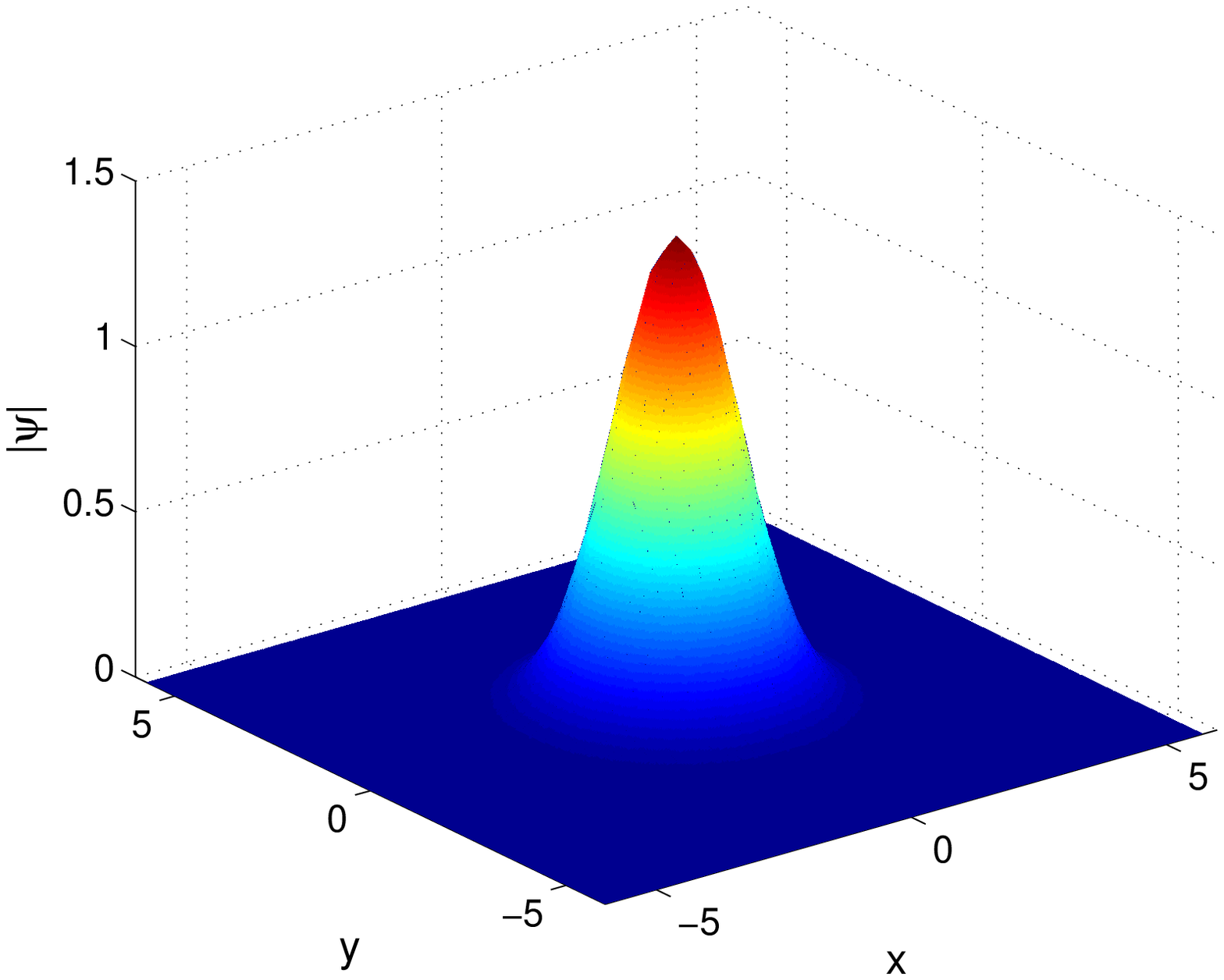}}
  \subfigure[]{\includegraphics[width=8cm,height=6cm]{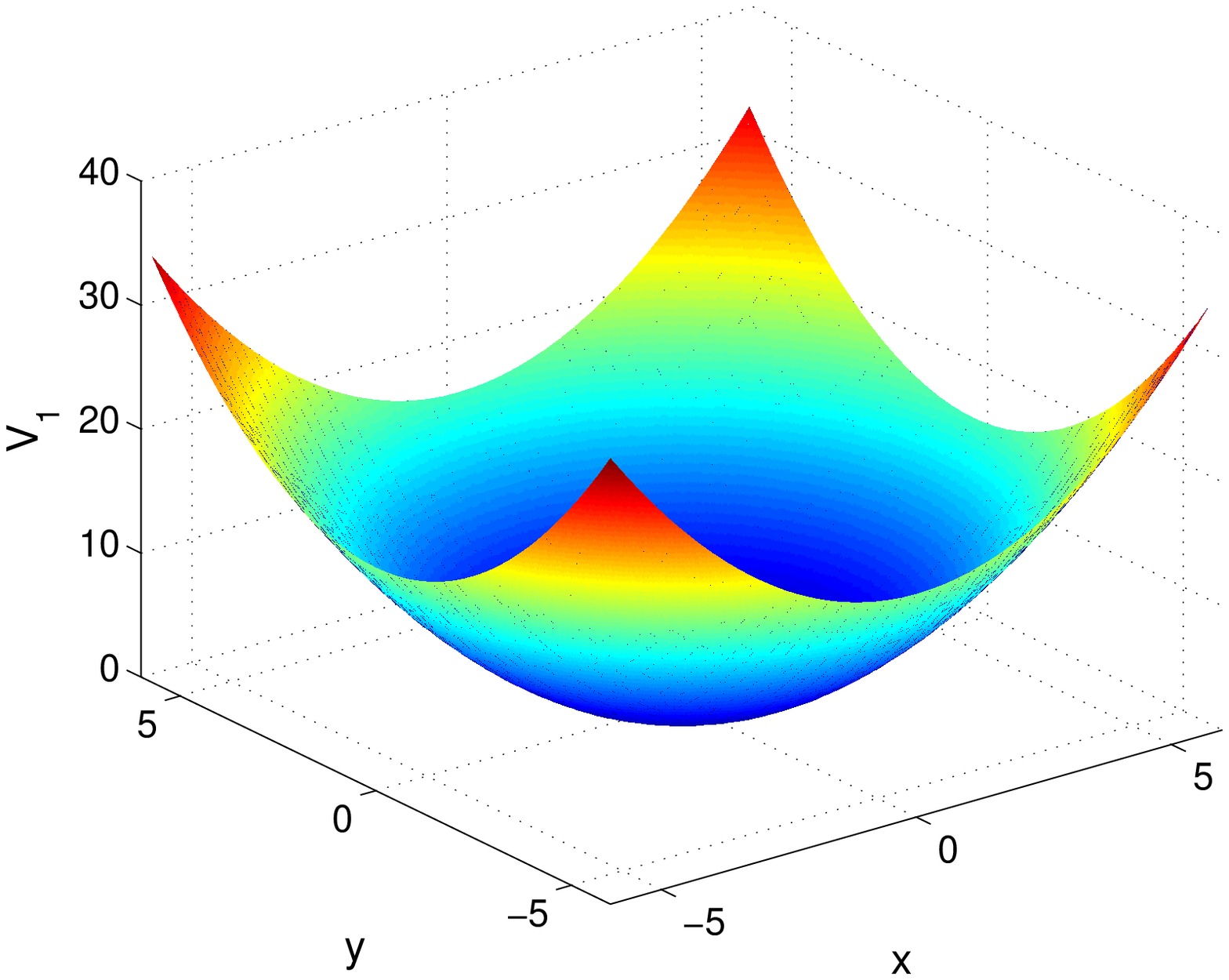}}
\end{tabular}
\caption{Shapes of the solution (left) and the potential $V_{1}$
(right).}
\label{2Dshape}
\end{figure}
From the results, we can see that ET2 conserves both the global energy and the ECL exactly while its global charge errors
oscillates in magnitude $10^{-7}.$ On the other hand, ST2 preserves the global charge accurately while its global energy errors
oscillates in magnitude $10^{-7}$ and its maximum residuals in the ECL oscillates in magnitude $10^{-6}.$ However, the maximum
global errors of ST2 are twice as large as that of ET2 under the three different $\Delta t$.
\end{myexp}

\begin{myexp}\label{7.2}
Let $V_{1}(x,y)=-\frac{1}{2}(x^{2}+y^{2}),\beta=-2$.
Given the initial condition
\begin{equation*}
\psi(x,y,0)=\frac{1}{\sqrt{\pi}}exp(-\frac{1}{2}(x^{2}+y^{2})),
\end{equation*}
we now consider the repulsive GP equation in space
$[-8,8]\times[-8,8]$ (see \cite{Kong2011}). Let $N=M=36, \Delta
t=0.1, $ we compute the numerical solution over the time interval
$[0,200].$ The results are plotted in Figs.
\ref{2DNLSR}, \ref{2DNLSR2}. Obviously, ET2 still show the eminent
long-term behaviour dealing with high dimensional problems.
\begin{figure}[ptb]
\centering
\begin{tabular}[c]{cccc}%
  % Requires \usepackage{graphicx}
  \subfigure[Errors in global energy]{\includegraphics[width=5cm,height=4cm]{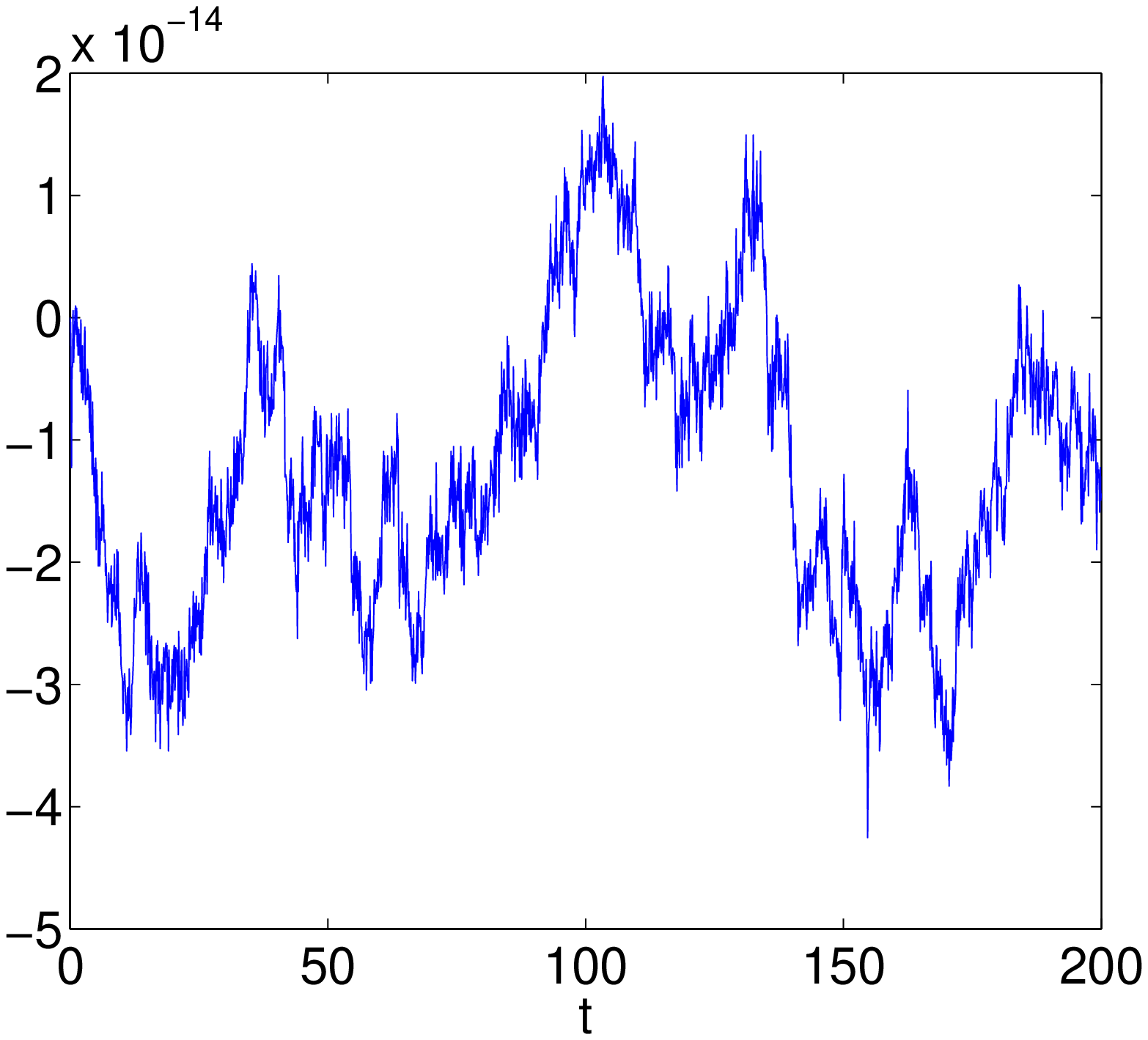}}
  \subfigure[Maximum residuals ($\mathbf{R}$) in the ECL]{\includegraphics[width=5cm,height=4cm]{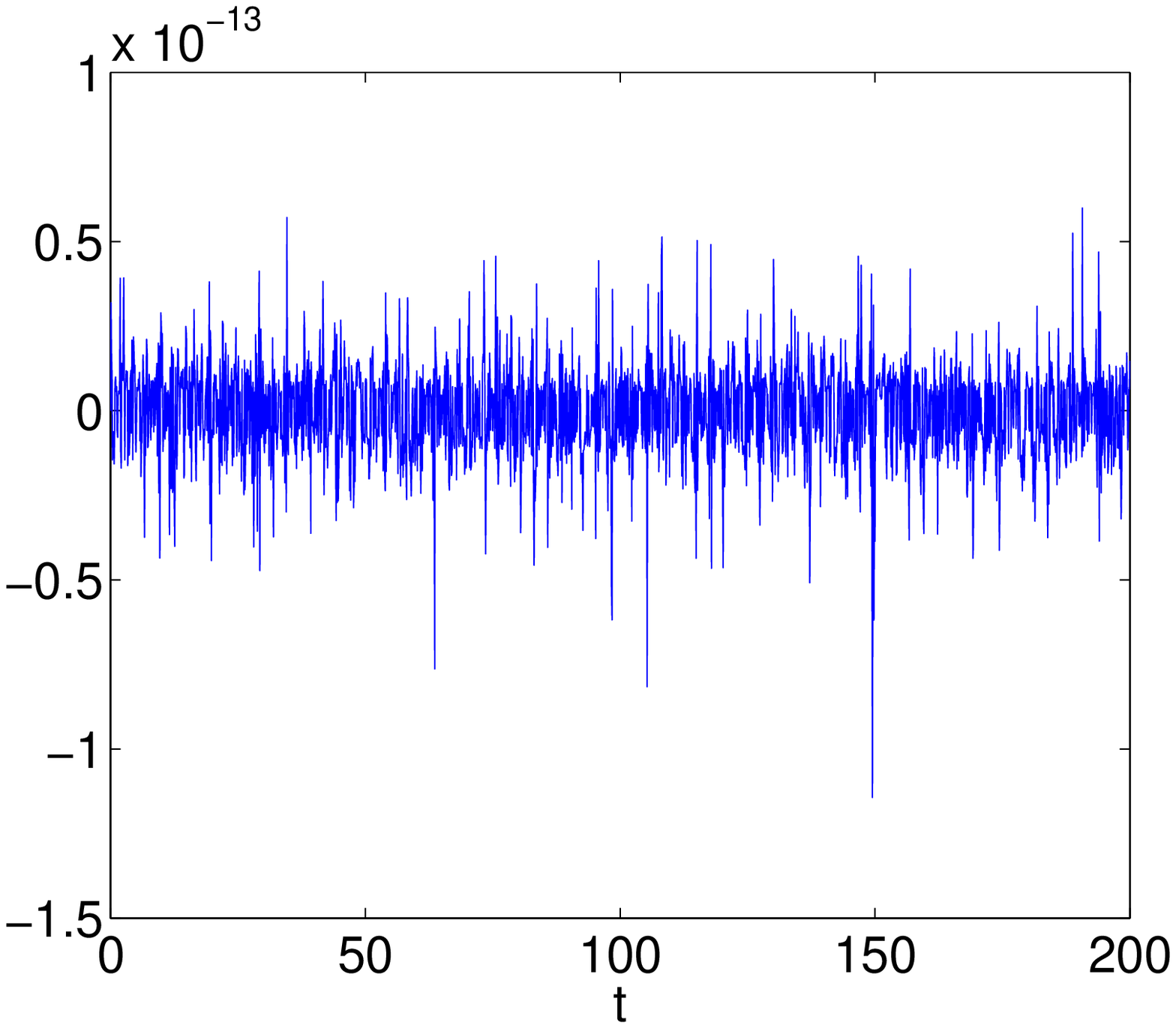}}
  \subfigure[Errors in global charge]{\includegraphics[width=5cm,height=4cm]{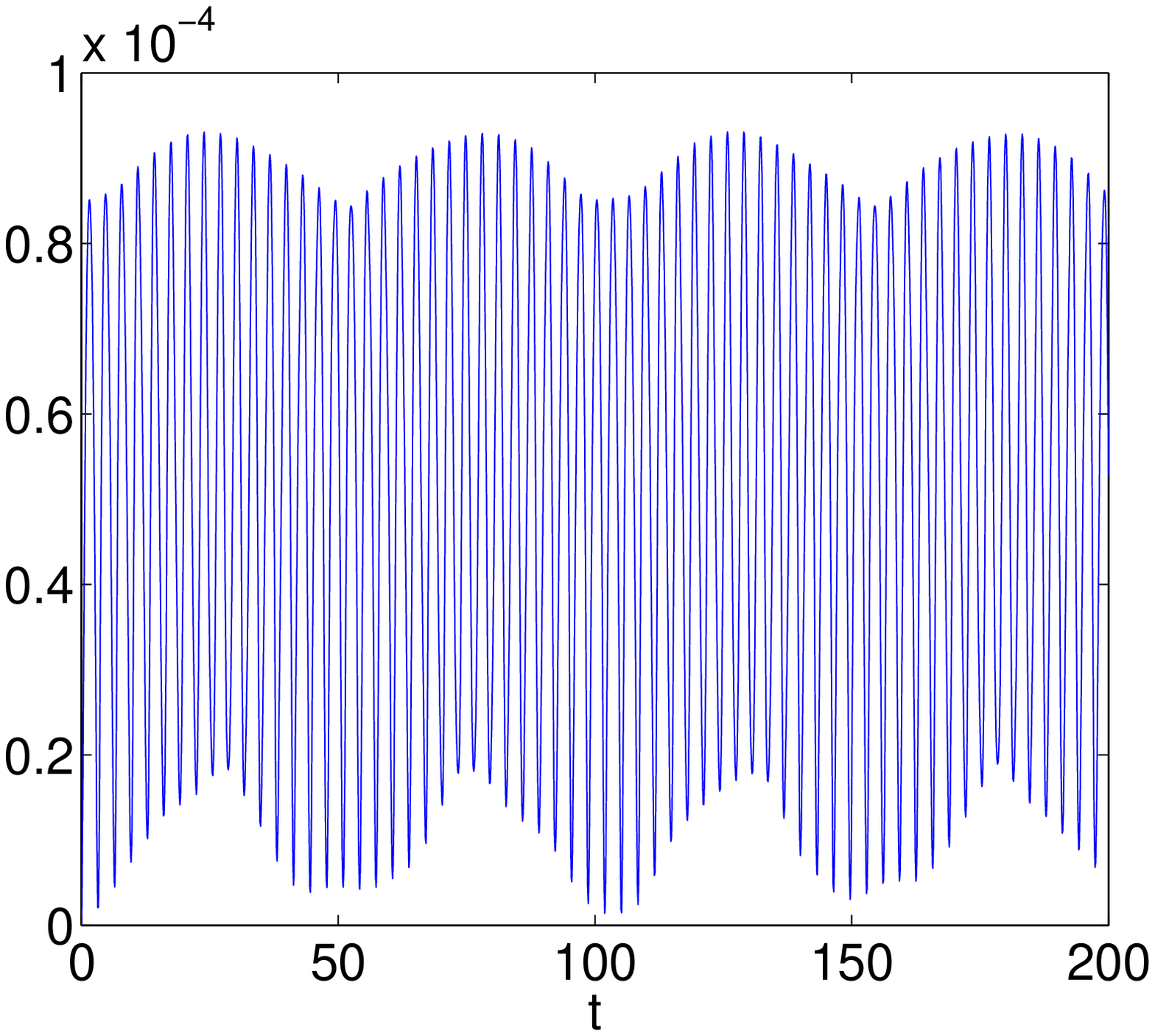}}
\end{tabular}
\caption{Errors obtained by ET2.}
\label{2DNLSR}
\end{figure}
\begin{figure}[ptb]
\centering
\begin{tabular}[c]{cccc}%
  % Requires \usepackage{graphicx}
  \subfigure[The shape of $\psi$ at $t=100$]{\includegraphics[width=8cm,height=6cm]{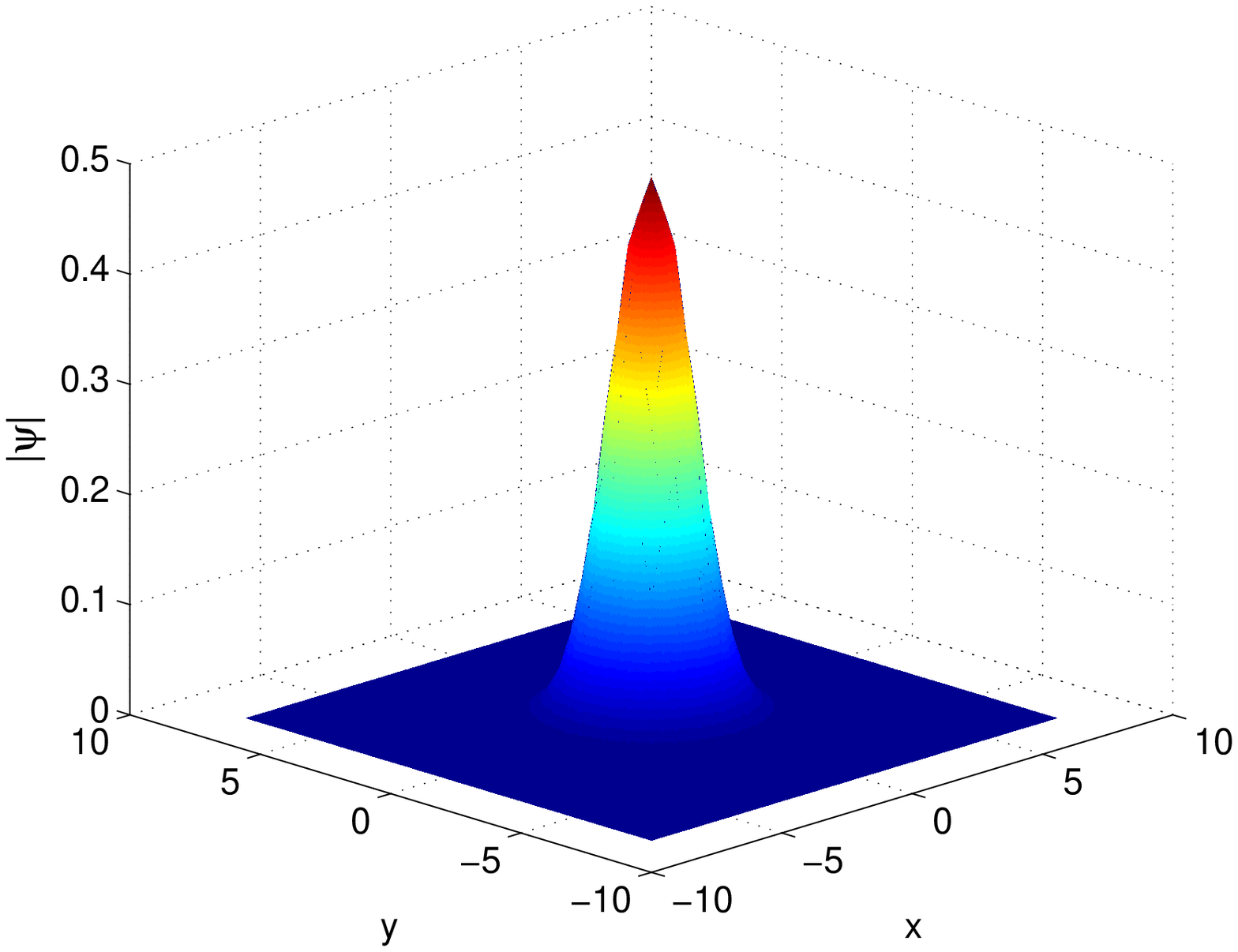}}
  \subfigure[The shape of $\psi$ at $t=200$]{\includegraphics[width=8cm,height=6cm]{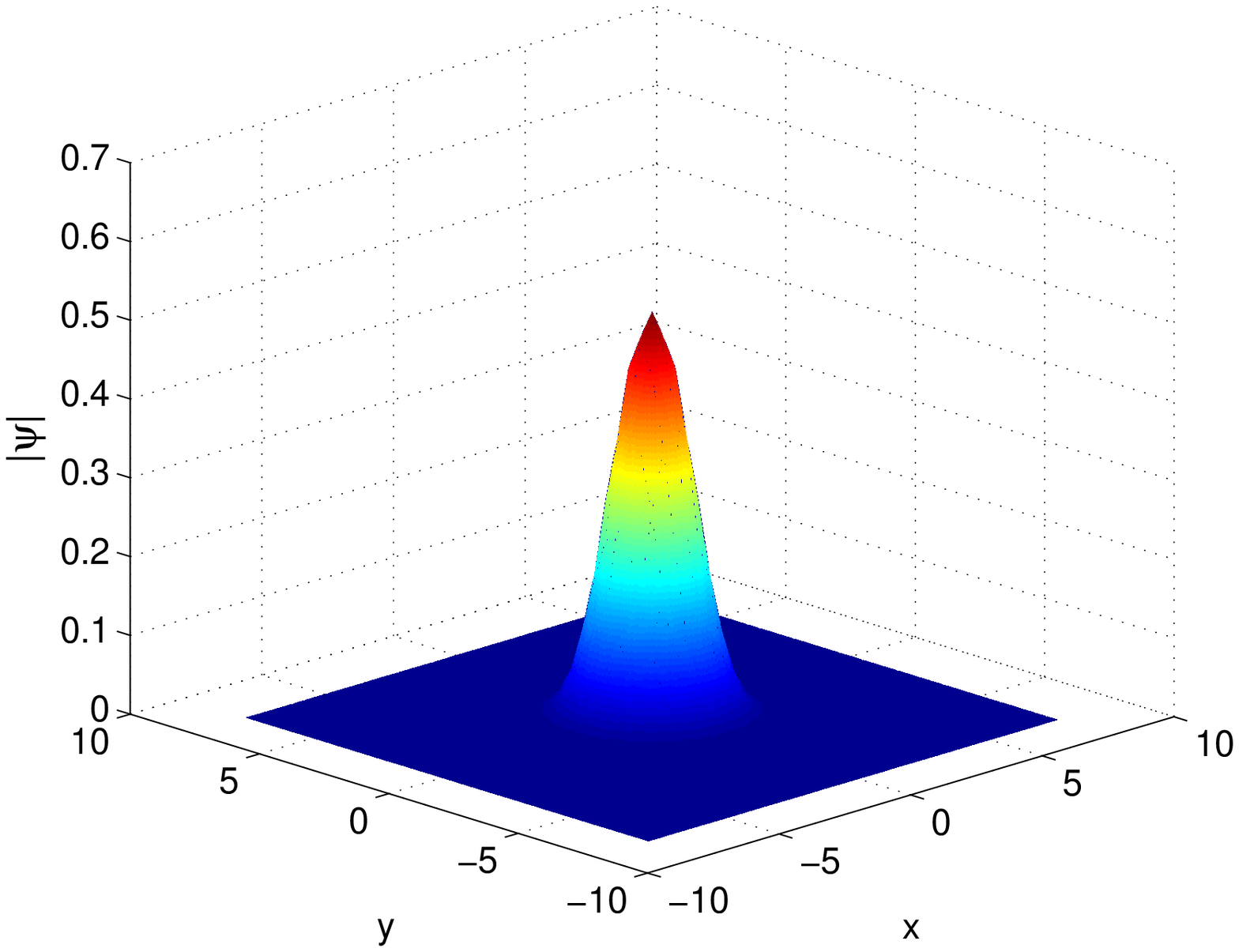}}
\end{tabular}
\caption{The numerical shapes of $\psi$.}
\label{2DNLSR2}
\end{figure}
\end{myexp}

\begin{myexp}
We then consider the 2DNLS with quintic nonlinearity:
\begin{equation}\label{7.3}
i\psi_{t}+\psi_{xx}+\psi_{yy}+V_{1}(x,y)\psi+|\psi|^{4}\psi=0,
\end{equation}
where
\begin{equation*}
V_{1}(x,y)=-A^{4}(4A^{4}(x^{2}+y^{2})-exp(-A^{4}(x^{2}+y^{2})))
\end{equation*}
is an external field, and $A$ is a constant. Its
potential is:
\begin{equation*}
V(\xi,x,y)=V_{1}(x,y)\xi+\frac{1}{3}\xi^{3}.
\end{equation*}
This equation admits the solution:
\begin{equation*}
\psi(x,y,t)=Aexp(-\frac{1}{4}A^{4}(x^{2}+y^{2}))exp(-iA^{4}t).
\end{equation*}
Its period is $\frac{2\pi}{A^{4}}$. Set
$A=1.5,x_{l}=-4,x_{r}=4,y_{l}=-4,y_{r}=4,\Delta t=0.01, N=M=42.$ We
integrate \eqref{7.3} over a very long interval [0,124] which is
about $100$ multiples of the period. Since the behaviours of ET2 and
ST2 in conserving the global charge and the energy are very similar
to those in Experiments \ref{7.1} and \ref{7.2}, they
are omitted here. The global errors of ET2 and ST2 in $l^{\infty}$
and $\frac{1}{\sqrt{NM}}l^{2}$ norms are shown in Fig. \ref{Q2DNLS}.

\begin{figure}[ptb]
\centering
\begin{tabular}[c]{cccc}%
  % Requires \usepackage{graphicx}
  \subfigure[]{\includegraphics[width=8cm,height=4cm]{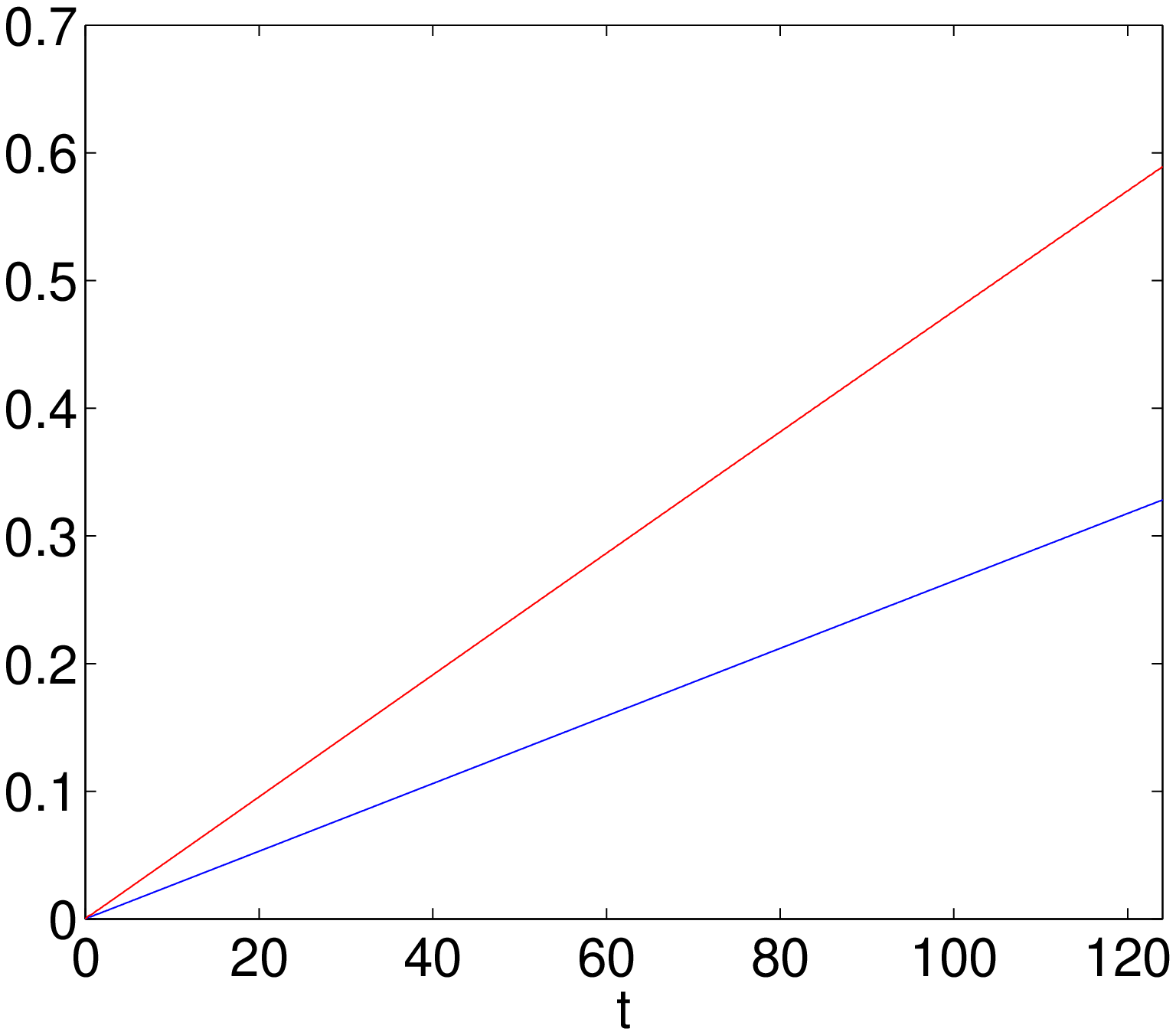}}
  \subfigure[]{\includegraphics[width=8cm,height=4cm]{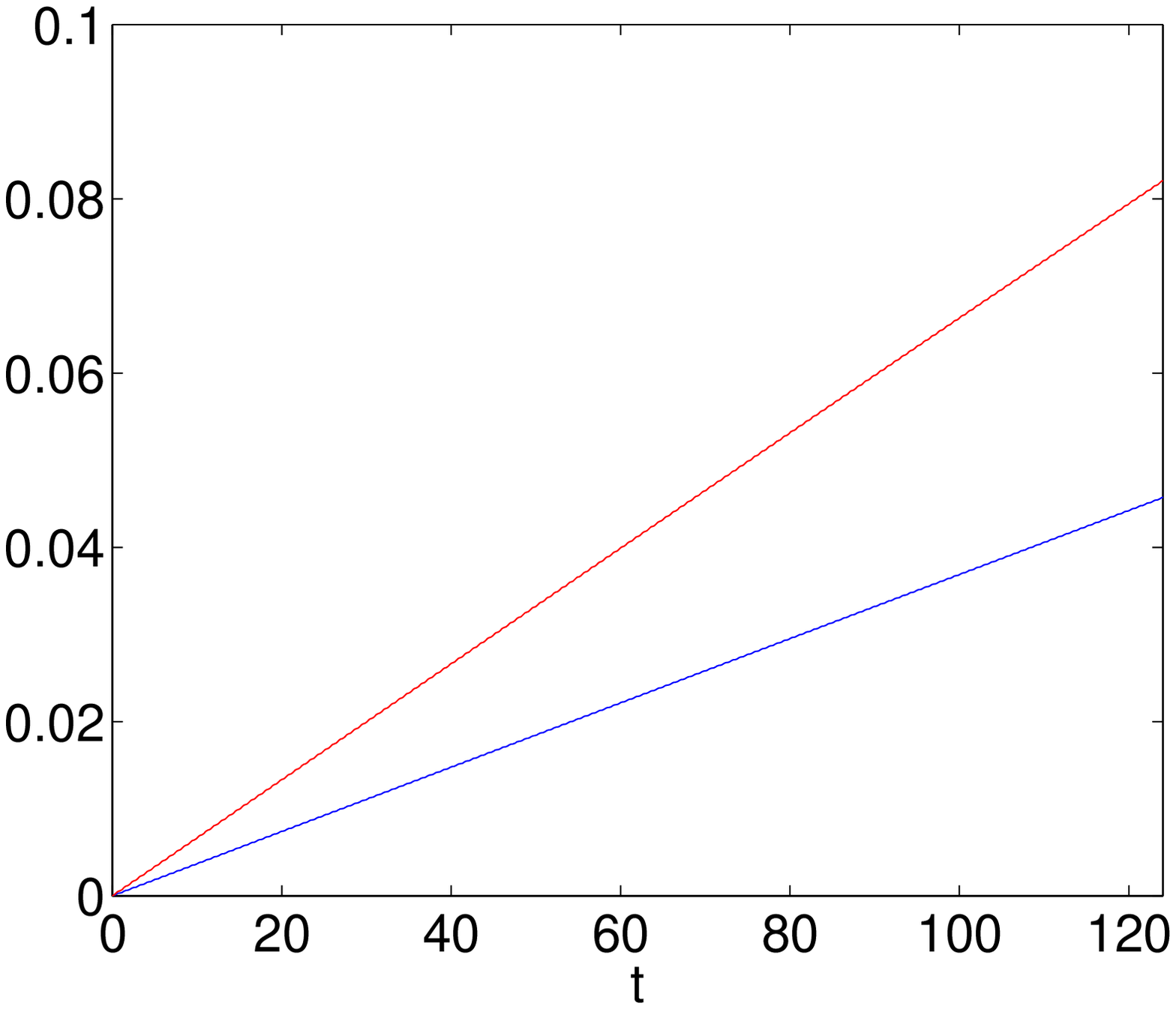}}
\end{tabular}
\caption{$l^{\infty}$ global  errors (left) and $\frac{1}{\sqrt{NM}}l^{2}$ global errors (right). The blue and
red curves are obtained by ET2 and ST2 respectively.}
\label{Q2DNLS}
\end{figure}

Clearly, in the quintic case, our method again wins over the
classical symplectic scheme ST2.
\end{myexp}

\section{Conclusions}
``For Hamiltonian differential equations there is a long-standing
dispute on the question whether in a numerical simulation it is more
important to preserve energy or symplecticity. Many people give more
weight on symplecticity, because it is known (by backward error
analysis arguments) that symplectic integrators conserve a modified
Hamiltonian" (Quote from Hairer's paper {\cite{Hairer2010}}).

However, due to the complexity of PDEs, the theory on
multi-symplectic integrators is still far from being satisfactory.
There are only a few results on some simple box schemes (e.g. the
Preissman and the Euler box scheme) and on special PDEs (e.g. the
nonlinear wave equation and the nonlinear Schr\"{o}dinger equation)
based on backward error analysis (see, e.g.
{\cite{Bridges2006,Islas2005,Moore2}}). These theories show that a
class of box schemes conserves the modified ECL and MCL(see, e.g.
{\cite{Islas2005}}). Besides, it seems there is no
robust theoretical results for the multi-symplectic (pseudo)
spectral scheme. Therefore, the local energy-preserving algorithms
may play a much more important role in PDEs than their counterparts
in ODEs.

In this paper, we presented a general local energy-preserving method
which can have arbitrarily high order for solving
multi-symplectic Hamiltonian PDEs. In our method, time is
discretized by a continuous Runge--Kutta method and space is
discretized by a pseudospectral method or a Gauss-Legendre
collocation method. It should be noted that the local energy
conservation law is admitted by more Hamiltonian PDEs than the
multi-symplectic conservation law is. Hence our local
energy-preserving methods can be more widely applied to
multi-symplectic Hamiltonian PDEs than multi-symplectic methods in
the literature. The numerical results accompanied in
this paper are plausible and promising. In the experiments on
CNLSs, our methods and the associated methods behave similarly to
the multi-symplectic methods of the same order. In the experiments
on 2DNLSs with external fields, our methods behave better than
symplectic methods in both cubic and quintic nonlinear problems.

%However, our method is implicit, the explicit and efficient energy-conserving method is under investigation.

\end{document}